\documentclass[a4paper, pdftex]{amsart}

\usepackage{etex}
\usepackage[utf8]{inputenc}
\usepackage[T1]{fontenc}
\usepackage{graphicx}
\usepackage{tikz}
\usepackage{pgfplots}
\usepackage{mathtools, eqparbox}
\usepackage{etoolbox}
\usepackage{adjustbox}
\usepackage{multirow}
\usepackage{import}
\usepackage{mathabx}
\usepackage{nextpage}

\makeatletter
\providerobustcmd*{\bigcupdot}{%
  \mathop{%
    \mathpalette\bigop@dot\bigcup
  }%
}
\newrobustcmd*{\bigop@dot}[2]{%
  \setbox0=\hbox{$\m@th#1#2$}%
  \vbox{%
    \lineskiplimit=\maxdimen
    \lineskip=-0.7\dimexpr\ht0+\dp0\relax
    \ialign{%
      \hfil##\hfil\cr
      $\m@th\cdot$\cr
      \box0\cr
    }%
  }%
}
\makeatother

\newcommand{\mystackrel}[3][T]{\stackrel{\eqmakebox[#1]{\scriptsize#2}}{#3}}
\usetikzlibrary{patterns}
\usetikzlibrary{arrows,decorations.markings}
\tikzset{
  schraffiert/.style={pattern=horizontal lines,pattern color=#1},
  schraffiert/.default=black
}
\usepackage[many]{tcolorbox}

\usepackage{float}
\usepackage{amsmath}
\makeatletter
\newcommand\numbereq{%
  \ifmeasuring@\else
    \refstepcounter{equation}%
  \fi
  \tag{\theequation}%
}
\makeatother
\allowdisplaybreaks
\usepackage{ltablex}
\usepackage{booktabs}
\usepackage{amsfonts}
\usepackage{amsthm}
\usepackage{amssymb}
\usepackage{setspace}	
\usepackage[format=hang]{caption}
\usepackage{thmtools}
\usepackage{thm-restate}
\usepackage{mathtools}
\usepackage{verbatim}
\usepackage[colorlinks,citecolor=blue]{hyperref}
\usepackage[capitalize]{cleveref}
\usepackage[all, cmtip]{xy}
\usepackage{chngcntr}
\usetikzlibrary{cd}
\usepackage{titlesec}
\titleformat{\section}
  {\normalfont\normalsize\bfseries}{\thesection.}{1em}{}
\titleformat{\subsection}
  {\normalfont\normalsize\bfseries}{\thesubsection.}{1em}{}
\titleformat{\subsubsection}
  {\normalfont\normalsize\bfseries}{\thesubsubsection.}{1em}{}
\titleformat{\listfigurename}
  {\normalfont\normalsize\bfseries}{Figures}{1em}{}
\counterwithin{figure}{section}

\usepackage[shortlabels]{enumitem}

\DeclareMathOperator{\Aut}{Aut}
\DeclareMathOperator{\Conf}{Conf}

\DeclareMathOperator{\MapOp}{Map}
\DeclareMathOperator{\PMapOp}{PMap}

\DeclareMathOperator{\HomeoOp}{Homeo^+}
\DeclareMathOperator{\PHomeoOp}{PHomeo}
\DeclareMathOperator{\Z}{Z}

\DeclareMathOperator{\B}{B}
\DeclareMathOperator{\Sym}{S}
\DeclareMathOperator{\PZ}{PZ}

\DeclareMathOperator{\ev}{ev}
\DeclareMathOperator{\id}{\mathrm{id}}

\title{Orbifold braid groups and complex braid groups}
\author{Jonas Flechsig}
\date{\today}
\makeatletter    
\def\l@section{\@tocline{1}{0,2pt}{2pc}{8mm}{\ \ }} 
\makeatletter    
\def\l@subsection{\@tocline{1}{0,2pt}{2pc}{8mm}{\ \ }} 

\renewcommand{\maketitle}{
    \begin{center}

        \phantom{.}  

        {\LARGE \bf \@title\par}
        \vspace{0.5cm}

    \end{center}
}\makeatother
\begin{document}
\newtheorem*{theorem*}{Theorem}
\newtheorem{theorem}{Theorem}[section]
\newtheorem{corollary}[theorem]{Corollary}
\newtheorem{lemma}[theorem]{Lemma}
\newtheorem{fact}[theorem]{Fact}
\newtheorem*{fact*}{Fact}
\newtheorem{proposition}[theorem]{Proposition}
\newtheorem{thmletter}{Theorem}
\newtheorem{observation}[theorem]{Observation}
\newtheorem{notation}[theorem]{Notation}
\renewcommand*{\thethmletter}{\Alph{thmletter}}
\theoremstyle{definition}
\newtheorem{example}[theorem]{Example}
\newtheorem{question}[theorem]{Question}
\newtheorem{definition}[theorem]{Definition}
\newtheorem{construction}[theorem]{Construction}
\theoremstyle{remark}
\newtheorem{remark}[theorem]{Remark}
\newtheorem{conjecture}[theorem]{Conjecture}
\newtheorem{case}{Case}
\newtheorem{claim}{Claim}
\newtheorem*{claim*}{Claim}
\newtheorem{step}{Step}
\counterwithin{case}{theorem}
\renewcommand{\thecase}{\arabic{case}}
\counterwithin{claim}{theorem}
\renewcommand{\theclaim}{\arabic{claim}}
\counterwithin{step}{theorem}
\renewcommand{\thestep}{\arabic{step}}

\newenvironment{intermediate}[1][\unskip]{%
\vspace*{5pt}
\par
\noindent
\textit{#1.}}
{}
\vspace*{5pt}

\newcommand{\doubletable}[1]{\begin{tabular}[l]{@{}l@{}}#1\end{tabular}}
\newcommand{\set}[1][ ]{\ensuremath{ \lbrace #1 \rbrace}}
\newcommand{\bsl}{\ensuremath{\setminus}}
\newcommand{\grep}[2]{\ensuremath{\left\langle #1 | #2\right\rangle}}
\renewcommand{\ll}{\left\langle}
\newcommand{\rr}{\right\rangle}
\newcommand{\cpxbrmn}{B(\TheOrder,\TheOrder,\TheStrand)}
\newcommand{\cpxbr}[2]{B(#1,#1,#2)}
\newcommand{\Map}[2]{\MapOp_{#1}\left({#2}\right)}
\newcommand{\PMap}[2]{\PMapOp_{#1}\left({#2}\right)}
\newcommand{\MapOrb}[2]{\MapOp_{#1}^{orb}\left({#2}\right)}
\newcommand{\PMapOrb}[2]{\PMapOp_{#1}^{orb}\left({#2}\right)}
\newcommand{\MapId}[2]{\MapOp_{#1}^{\id}\left({#2}\right)}
\newcommand{\PMapId}[2]{\PMapOp_{#1}^{\id}\left({#2}\right)}
\newcommand{\MapIdOrb}[2]{\MapOp_{#1}^{\id,orb}\left({#2}\right)}
\newcommand{\PMapIdOrb}[2]{\PMapOp_{#1}^{\id,orb}\left({#2}\right)}
\newcommand{\HomeoId}[2]{\HomeoOp_{#1}^{\id}({#2})}
\newcommand{\PHomeoId}[2]{\PHomeoOp_{#1}^{\id}({#2})}
\newcommand{\Homeo}[2]{\HomeoOp_{#1}({#2})}
\newcommand{\PHomeo}[2]{\PHomeoOp_{#1}({#2})}
\newcommand{\HomeoIdOrb}[2]{\HomeoOp_{#1}^{\id,orb}({#2})}
\newcommand{\PHomeoIdOrb}[2]{\PHomeoOp_{#1}^{\id,orb}({#2})}
\newcommand{\HomeoOrb}[2]{\HomeoOp_{#1}^{orb}({#2})}
\newcommand{\PHomeoOrb}[2]{\PHomeoOp_{#1}^{orb}({#2})}
\newcommand{\PHomeoOrbt}[3]{\PHomeoOp_{#1}^{orb}({#2})}

\newcommand{\omicron}{o}
\newcommand{\cp}{c}
\newcommand{\pct}{p}
\newcommand{\TheCone}{N}
\newcommand{\ThePct}{L}
\newcommand{\Pc}{\theta}
\newcommand{\Pct}{\iota}
\newcommand{\NPct}{\lambda}
\newcommand{\TheOrder}{m}
\newcommand{\Ord}{t}
\newcommand{\TheStrand}{n}
\newcommand{\Order}{l}
\newcommand{\Strr}{h}
\newcommand{\Str}{i}
\newcommand{\Strand}{j}
\newcommand{\NStrand}{k}
\newcommand{\NNStrand}{l}
\newcommand{\pStrand}{p}
\newcommand{\qStrand}{q}
\newcommand{\sStrand}{s}
\newcommand{\tStrand}{t}
\newcommand{\Subdivision}{i}
\newcommand{\TheSubdivision}{p}
\newcommand{\Dim}{i}
\newcommand{\NDim}{j}
\newcommand{\TheDim}{k}
\newcommand{\Subdiv}{i}
\newcommand{\TheSubdiv}{p}
\newcommand{\NSubdiv}{j}
\newcommand{\TheNSubdiv}{q}
\newcommand{\NNSubdiv}{k}
\newcommand{\TheNNSubdiv}{r}
\newcommand{\NNNSubdiv}{l}
\newcommand{\TheFrac}{\frac{2\pi}{\TheOrder}}
\newcommand{\HalfFrac}{\frac{\pi}{\TheOrder}}
\newcommand{\neigh}{\lambda}
\newcommand{\htwC}{h_1^\twsC}
\newcommand{\htw}{h_{\TheStrand-1}^{\twsC'}}
\newcommand{\col}{blue}
\newcommand{\colo}{olive}
\newcommand{\short}{red}
\newcommand{\mult}{orange}
\newcommand{\red}{red}
\newcommand{\green}{olive}
\newcommand{\gre}{green}
\newcommand{\blue}{blue}
\newcommand{\SGpct}{\Sigma_\freeprod(\ThePct,\TheCone)}
\newcommand{\SG}{\Sigma_\freeprod(\ThePct)}
\newcommand{\Spct}{\Sigma(\ThePct,\TheCone)}
\renewcommand{\S}{\Sigma(\ThePct)}
\newcommand{\Sk}[1]{\Sigma(#1)}
\newcommand{\orbtwo}{\Sigma_{\freeprodtwo}}
\newcommand{\G}{G}
\newcommand{\g}{g}
\newcommand{\freeprod}{\Gamma}
\newcommand{\freeprodtwo}{\cycm\ast\cyc{\TheOrder'}}
\newcommand{\freeprodex}{\Gamma_{\TheOrder_\nu}^\TheCone}
\newcommand{\freeprodexk}[1]{\Gamma_{\TheOrder_\nu}^{#1}}
\newcommand{\freegrp}[1]{F_{#1}}
\newcommand{\free}[1]{F^{(#1)}}
\newcommand{\cycm}{\ZZ_\TheOrder}
\newcommand{\cyc}[1]{\ZZ_{#1}}
\newcommand{\kernel}{K}
\newcommand{\inter}[1]{{#1}^{\circ}}
\newcommand{\interext}[1]{{#1}^{\circ,\text{ext}}}

\newcommand{\Twist}{A}
\newcommand{\TwistP}{B}
\newcommand{\TwistC}{C}
\newcommand{\TwP}[1]{T_{#1}}
\newcommand{\TwC}[1]{U_{#1}}
\newcommand{\twist}{a}
\newcommand{\twistP}{b}
\newcommand{\twistC}{c}
\newcommand{\twP}[1]{t_{#1}}
\newcommand{\twC}[1]{u_{#1}}
\newcommand{\twsP}{t}
\newcommand{\twsC}{u}
\newcommand{\Twistn}[1]{X_{#1}}
\newcommand{\TwistnP}[1]{Y_{#1}}
\newcommand{\TwistnC}[1]{Z_{#1}}
\newcommand{\twistn}[1]{x_{#1}}
\newcommand{\twistnP}[1]{y_{#1}}
\newcommand{\twistnC}[1]{z_{#1}}
\newcommand{\twistnsC}{z}
\newcommand{\TwsP}{T}
\newcommand{\TwsC}{U}
\newcommand{\FD}{F}

\newcommand{\MA}{\ensuremath{\mathcal{MA}_\TheStrand}}
\newcommand{\pMA}{\ensuremath{\mathcal{MA}_\TheStrand(F_\freeprod)}}

\newcommand{\HA}{\ensuremath{\mathcal{HA}_\TheStrand}}

\newcommand{\MAo}{\ensuremath{\mathcal{MA}_{\TheStrand}(\Sigma_\freeprod)}}
\newcommand{\pMAo}{\ensuremath{\mathcal{MA}_{\TheStrand}(\Sigma_\freeprod^\ast)}}

\newcommand{\MAoZ}{\ensuremath{\mathcal{MA}_{\TheStrand}(D_{\ZZ_\TheOrder})}}
\newcommand{\MAoD}{\ensuremath{\mathcal{MA}_{\TheStrand}(\CC_{D_\TheOrder})}}
\newcommand{\tMAo}{\ensuremath{\tilde{\mathcal{MA}}_{\TheStrand}(\Sigma_\freeprod)}}

\newcommand{\HAo}{\ensuremath{\mathcal{HA}_{\TheStrand}(\Sigma_\freeprod)}}

\newcommand{\bpHAk}[1]{\ensuremath{\mathcal{HA}_{\TheStrand,\TheStrand',#1}}}
\newcommand{\bpHAo}{\ensuremath{\mathcal{HA}_{\TheStrand,\TheStrand'}(\Sigma_\freeprod)}}
\newcommand{\bpHAok}[1]{\ensuremath{\mathcal{HA}_{\TheStrand,\TheStrand',#1}(\Sigma_\freeprod)}}
\newcommand{\bpMAo}{\ensuremath{\mathcal{MA}_{\TheStrand,\TheStrand'}(\Sigma_\freeprod)}}
\newcommand{\bpMAok}[1]{\ensuremath{\mathcal{MA}_{\TheStrand,\TheStrand',#1}(\Sigma_\freeprod)}}
\newcommand{\bpMAoF}{\ensuremath{\mathcal{MA}_{\TheStrand,\TheStrand',\ThePct}^{F_\freeprod}(\Sigma_\freeprod)}}

\newcommand{\pbpMAo}{\ensuremath{\mathcal{MA}_{\TheStrand,\TheStrand'}(\Sigma_\freeprod^\ast)}}
\newcommand{\tbpMAo}{\ensuremath{\mathcal{MA}_{\TheStrand,\TheStrand'}(\Sigma_\freeprod)}}
\newcommand{\bpMAos}{\ensuremath{\mathcal{MA}_{\TheStrand,\TheStrand'}^{sim}(\Sigma_\freeprod)}}
\newcommand{\bpMA}{\ensuremath{\mathcal{MA}_{\TheStrand,\TheStrand'}}}
\newcommand{\bpMAk}[2]{\ensuremath{\mathcal{MA}_{\TheStrand,\TheStrand',{#1}}}(#2)}
\newcommand{\pbpMA}{\ensuremath{\mathcal{MA}_{\TheStrand,\TheStrand'}(F_\freeprod^\ast)}}
\newcommand{\tbpMA}{\ensuremath{\mathcal{MA}_{\TheStrand,\TheStrand'}(\Sigma_\freeprod)}}

\newcommand{\bpM}{\ensuremath{\mathcal{M}_{\TheStrand,\TheStrand',k}}}
\newcommand{\seg}{s}
\newcommand{\st}{\mathrm{st}}
\newcommand{\map}{\rho}
\newcommand{\Rep}{T}

\newcommand{\GL}[2][\TheRank]{\ensuremath{\operatorname{GL_{#1}}(#2)}}
\newcommand{\hmu}[2]{h_{#2}^{\tau_{#1}}}
\newcommand{\Stab}{\operatorname{Stab}}
\newcommand{\CC}{\mathbb{C}}
\newcommand{\RR}{\mathbb{R}}
\newcommand{\ZZ}{\mathbb{Z}}
\newcommand{\NN}{\mathbb{N}}
\newcommand{\PP}{\mathbb{P}}
\newcommand{\HH}{\mathbb{H}}
\newcommand{\SSS}{\mathbb{S}}
\newcommand{\iotaPMap}{\iota_{\PMap_\TheStrand}}
\newcommand{\piPMap}{\pi_{\PMap_\TheStrand}}
\newcommand{\iotaMapast}{\iota_{\PMap_\TheStrand^\ast}}
\newcommand{\piMapast}{\pi_{\PMap_\TheStrand^\ast}}
\newcommand{\iotaPZ}{\iota_{\PZ_\TheStrand}}
\newcommand{\piPZ}{\pi_{\PZ_\TheStrand}}
\newcommand{\sPZ}{\mathrm{s}_{\PZ_\TheStrand}}
\newcommand{\iotaPZast}{\iota_{\PZ_\TheStrand^\ast}}
\newcommand{\piPZast}{\pi_{\PZ_\TheStrand^\ast}}
\newcommand{\iotaPZpct}{\iota_{\PZ_\TheStrand^\ast}}
\newcommand{\piPZpct}{\pi_{\PZ_\TheStrand^\ast}}
\newcommand{\Push}{\mathrm{Push}}
\newcommand{\Forget}{\mathrm{Forget}}
\newcommand{\PushPMap}{\mathrm{Push}_{\text{PMap}}}
\newcommand{\ForgetPMap}{\mathrm{Forget}_{\text{PMap}}}
\newcommand{\piMap}{\pi_{\text{Map}}}
\newcommand{\varphiMap}{\varphi_{\text{Map}}}

\author{J. Flechsig}
\address{Jonas Flechsig: Fakult\"at f\"ur Mathematik, Universit\"at Bielefeld, D-33501 Bielefeld, Germany}

\maketitle
\begin{center}
Jonas Flechsig
\\[5pt]
\today
\\[10pt]
\textbf{Abstract} 
\end{center}

A result of Allock \cite{Allcock2002} states that certain orbifold braid groups contain Artin groups of type $D_\TheStrand,\tilde{B}_\TheStrand$ and $\tilde{D}_\TheStrand$ as finite index subgroups. The underlying orbifolds have at most two cone points of order two. Based on \cite{Flechsig2023braid} and \cite{Flechsig2023mcg}, we generalize this result allowing cone points of arbitrary order. In these cases, the orbifold braid groups contain similar subgroups of finite index. We show that in many cases these subgroups can be identified as certain complex braid groups. 

\section{Introduction}

An orbifold braid group is an analog of Artin braid groups or, more generally, surface braid groups. Instead of considering point collections moving inside a disk or a surface, orbifold braids move inside a $2$-dimensional orbifold. Orbifold braid groups attracted interest since some of them contain spherical and affine Artin groups of type $D_\TheStrand,\tilde{B}_\TheStrand$ and $\tilde{D}_\TheStrand$ as finite index subgroups by work of Allcock \cite{Allcock2002}. For these Artin groups, the orbifold braid groups provide us with braid pictures. In the following, Roushon published several articles on the structure of orbifold braid groups \cite{Roushon2021,Roushon2022b,Roushon2023} and the contained Artin groups \cite{Roushon2021a}. Moreover, Crisp--Paris~\cite{CrispParis2005} studied the outer automorphism group of the orbifold braid group. The present article generalizes the result of Allcock \cite{Allcock2002}. 

As in the work of Roushon \cite{Roushon2021}, we consider braid groups on orbifolds with finitely many punctures and cone points (of possibly different orders). The underlying orbifolds are defined using the following data: Let $\freeprod$ be a free product of finitely many finite cyclic groups. As such $\freeprod$ acts on a planar, contractible surface $\Sigma$ (with boundary). This surface can be obtained by thickening the Bass--Serre tree, see Section \ref{subsec:orb_braid_grps} for details. If we add $\ThePct$ punctures, we obtain a similar orbifold as studied by Roushon, which we denote by $\Sigma_\freeprod(\ThePct)$. In contrast to his article, we consider orbifolds with non-empty boundary but this does not affect the structure of the orbifold braid groups. The only singular points in the orbifold $\Sigma_\freeprod(\ThePct)$ are cone points that correspond to the finite cyclic factors of the free product $\freeprod$. 

The elements of orbifold braid groups $\Z_\TheStrand(\Sigma_\freeprod(\ThePct))$ are represented by braid diagrams such as in Figure \ref{fig:braid_diagram_intro} with $\TheStrand$ strands (drawn in black), $\TheCone$ cone point bars (drawn in red with a cone at the top) and $\ThePct$ bars that represent the punctures (drawn in blue with a diamond at the top). The composition of these diagrams works as in Artin braid groups. 

\begin{figure}[H]
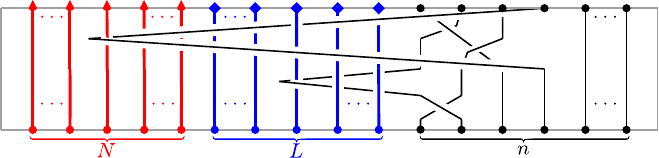
\caption{Example of an orbifold braid in $\Z_\TheStrand(\Sigma_\freeprod(\ThePct))$.}
\label{fig:braid_diagram_intro}
\end{figure}

Using the relation to the orbifold mapping class group introduced in \cite{Flechsig2023mcg}, one can obtain a finite presentation of the orbifold braid group, see Theorem \ref{thm:pres_Z_n} and \cite[Section 5.1]{Flechsig2023braid} for further details. 

The key difference to the Artin braid groups is that orbifold braid groups contain torsion elements, see Figure \ref{fig:fin_order_el} for an example. In this example braid all strands but one are constant and the non-constant strand encircles a single cone point bar but no other strands or punctures. See \cite[Remark 3.12]{Flechsig2023braid} for further details on these torsion elements. 

\begin{figure}[H]
\resizebox{0.6\textwidth}{!}{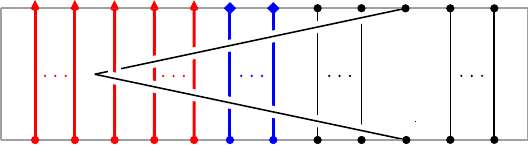
}
\caption{A torsion element in $Z_\TheStrand(\Sigma_\freeprod(\ThePct))$.}
\label{fig:fin_order_el}
\end{figure}

\subsection*{Main result}

If the orbifold $\Sigma_\freeprod(\ThePct)$ has one or two cone points and at most one puncture, the presentation of $\Z_\TheStrand(\Sigma_\freeprod(\ThePct))$ given in Theorem \ref{thm:pres_Z_n} allows us to generalize the above mentioned result of Allcock \cite{Allcock2002} about Artin groups of types $D_\TheStrand, \tilde{B}_\TheStrand$ and $\tilde{D}_\TheStrand$: 

\begin{thmletter}
\label{thm-intro:general_Allcock}
\leavevmode
\begin{enumerate}[label=\textup{(\arabic*)}]
\item 
\label{thm-intro:general_Allcock_it1}
For $\TheStrand\geq4$, the orbifold braid group $\Z_\TheStrand(\Sigma_{\freeprodtwo})$ is a semidirect product 
\[
\tilde{A}(\Delta_\TheStrand^{\TheOrder,\TheOrder'})\rtimes(\cycm\times\cyc{\TheOrder'}) 
\]
with $\cycm\times\cyc{\TheOrder'}=\langle\twsC,\twsC'\rangle$ and $\tilde{A}(\Delta_\TheStrand^{\TheOrder,\TheOrder'})=\langle h_1,...,h_{\TheStrand-1},\htwC,\htw\rangle$. For $\TheOrder=\TheOrder'=2$ and $\TheStrand\geq3$, the group $\tilde{A}(\Delta_\TheStrand^{2,2})$ is an Artin group of type $\tilde{D}_\TheStrand$. For $(\TheOrder,\TheOrder')\in\{(3,2),(3,3),(4,2)\}$, the group $\tilde{A}(\Delta_\TheStrand^{\TheOrder,\TheOrder'})$ has the structure of a complex braid group.

\begin{figure}[H]
\centerline{\resizebox{1.35\textwidth}{!}{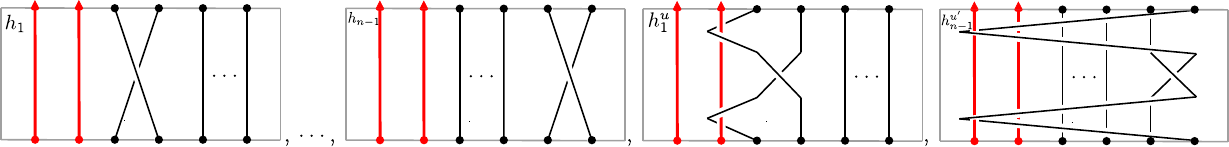}}
\caption{\small{Generators of the subgroup $\tilde{A}(\Delta_\TheStrand^{\TheOrder,\TheOrder'})$.}}
\label{fig:gens_complex_braid_grp_2cp}
\end{figure}

\item 
\label{thm-intro:general_Allcock_it2}
For $\TheStrand\geq2$, the orbifold braid group $\Z_\TheStrand(D_{\cycm})$ is a semidirect product
\[
\tilde{A}(\Delta_\TheStrand^\TheOrder)\rtimes\cycm
\]
with $\cycm=\langle\twsC\rangle$ and $\tilde{A}(\Delta_\TheStrand^\TheOrder)=\langle h_1,...,h_{\TheStrand-1},\htwC\rangle$. 
For $\TheOrder=2$, the group $\tilde{A}(\Delta_\TheStrand^2)$ is an Artin group of type $D_\TheStrand$. For $\TheStrand=2$ and $\TheOrder\geq2$, the group $\tilde{A}(\Delta_2^\TheOrder)$ is an Artin group of type $I_2(\TheOrder)$. For $\TheStrand,\TheOrder\geq3$, the group $\tilde{A}(\Delta_\TheStrand^\TheOrder)$ is the complex braid group of type $\cpxbr{\TheOrder}{\TheStrand}$. 

\begin{figure}[H]
\resizebox{0.9\textwidth}{!}{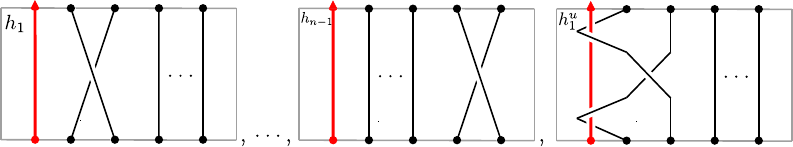}
\caption{\small{Generators of the subgroup $\tilde{A}(\Delta_\TheStrand^\TheOrder)$.}}
\label{fig:gens_complex_braid_grp_1cp}
\end{figure}

\item 
\label{thm-intro:general_Allcock_it3}
For $\TheStrand\geq3$, the orbifold braid group $\Z_\TheStrand(D_{\cycm}(1))$ is a semidirect product 
\[
\tilde{A}(\bar{\Delta}_\TheStrand^\TheOrder)\rtimes\cycm 
\]
with $\cycm=\langle\bar{\twsC}\rangle$ and $\tilde{A}(\bar{\Delta}_\TheOrder^\TheStrand)=\langle \twsP,h_1,...,h_{\TheStrand-1},h_{\TheStrand-1}^\twsC\rangle$. For $\TheOrder\in\{3,4\}$, the group $\tilde{A}(\bar{\Delta}_\TheStrand^\TheOrder)$ has the structure of a complex braid group. For $\TheOrder=2$, the group $\tilde{A}(\bar{\Delta}_\TheStrand^2)$ is an Artin group of type $\tilde{B}_\TheStrand$. For further details, see Chapter~\textup{\ref{sec:general_Allcock}}. 

\begin{figure}[H]
\centerline{\resizebox{1.35\textwidth}{!}{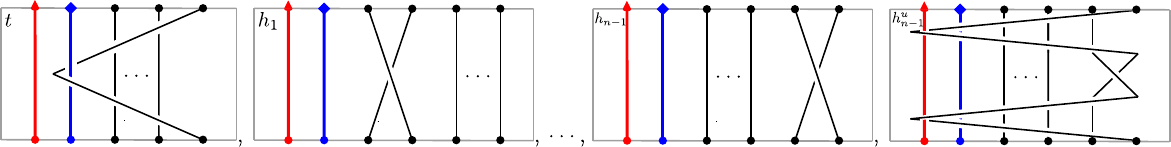}}
\caption{Generators of the subgroup $\tilde{A}(\bar{\Delta}_\TheOrder^\TheStrand)$.}
\label{fig:gens_complex_braid_grp_1cp_1pct}
\end{figure}
\end{enumerate}
\end{thmletter}

Moreover, we can identify non-trivial elements in the center of the orbifold braid group $\Z_\TheStrand(D_{\cycm})$, see Proposition \ref{prop:D_Z_m_central_el}. As an application of Theorem \ref{thm-intro:general_Allcock}\ref{thm-intro:general_Allcock_it2}, we obtain a new proof that the center of $\B(\TheOrder,\TheOrder,\TheStrand)$ is non-trivial, see Corollary \ref{cor:cbr_grp_Delta_n_m_center}. The center of $\B(\TheOrder,\TheOrder,\TheStrand)$ was calculated in \cite[Corollary 3.33]{BroueMalleRouquier1998} by different methods. 

\section*{Acknowledgments} 

I thank my adviser Kai-Uwe Bux for his support and many helpful discussions. Thanks are also due to Georges Neaime and Ivan Marin for answering my questions about Artin groups and complex braid groups. I am grateful to Elena Tielker for proofreading parts of this text. 

The author was funded by the Deutsche Forschungsgemeinschaft (DFG, German Research Foundation) – 426561549 and Bielefelder Nachwuchsfonds.

\renewcommand{\Twist}{A}
\renewcommand{\TwistP}{B}
\renewcommand{\TwistC}{C}
\renewcommand{\twist}{a}
\renewcommand{\twistP}{b}
\renewcommand{\twistC}{c}

\setlist[enumerate,1]{label=\textup{(\arabic*)}, ref=\thetheorem(\arabic*)}
\setlist[enumerate,2]{label=\textup{\alph*)}, ref=\alph*)}

\section{Preliminaries}

\subsection{Orbifold braid groups}
\label{subsec:orb_braid_grps}

In this article, we only consider orbifolds that are given as the quotient of a manifold (typically a surface) by a proper group action. Recall that an action 
\[
\phi:\G\rightarrow\Homeo{}{M},\g\mapsto\phi_\g 
\]
on a manifold $M$ is \textit{proper} if for each compact set $K\subseteq M$ the set 
\[
\{\g\in\G\mid\phi_\g(K)\cap K\neq\emptyset\} 
\]
is compact. Since we endow $\G$ with the discrete topology, i.e.\ the above set is finite. Orbifolds that appear as proper quotients of manifolds are called \textit{developable} in the terminology of Bridson--Haefliger \cite{BridsonHaefliger2011} and \textit{good} in the terminology of Thurston \cite{Thurston1979}. 

\begin{definition}[{Orbifolds}, {\cite[Chapter III.G,1.3]{BridsonHaefliger2011}}]
\label{def:good_orb}
Let $M$ be an orientable manifold, possibly with boundary, and $\G$ a group with a monomorphism 
\[
\phi:\G\rightarrow\Homeo{}{M} 
\]
such that $\G$ acts properly on $M$. Under these conditions the $3$-tuple $(M,\G,\phi)$ is called an \textit{orbifold}, which we denote by $M_\G$. If $\Stab_\G(\cp)\neq\{1\}$ for a point $\cp\in M$, the point $\cp$ is called a \textit{singular point} of $M_\G$. If $\Stab_\G(\cp)$ for a point $\cp\in M$ is cyclic of finite order $\TheOrder\geq2$, the point $\cp$ is called a \textit{cone point} of $M_\G$ of order~$\TheOrder$. 
\end{definition}
 
A first example of an orbifold is the following: 

\begin{example}
\label{ex:good_orb_D_cyc_m}
Let $\cycm$ be a cyclic group of order $\TheOrder$. The group $\cycm$ acts on a disk~$D$ by rotations around its center. The action is via isometries and the acting group is finite, i.e.\ the action is proper. Consequently, $D_{\cycm}$ is an orbifold with exactly one singular point in the center of $D$, which is a cone point. 
\end{example}

Example \ref{ex:good_orb_D_cyc_m} motivates a more general construction for a free product of finitely many finite cyclic groups which we describe briefly in the following. We will consider this generalization of the orbifold $D_{\cycm}$ throughout the article. For further details, we refer to the authors PhD thesis \cite[Section 2.1]{Flechsig2023}. 

\begin{example}
\label{ex:good_orb_free_prod}
Let $\freeprod$ be a free product of finite cyclic groups $\cyc{\TheOrder_1},...,\cyc{\TheOrder_\TheCone}$. The group $\freeprod$ is the fundamental group of the following graph of groups with trivial edge groups: 
\begin{figure}[H]
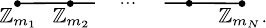
\end{figure}
As such $\freeprod$ acts on its Bass--Serre tree $T$. The fundamental domain of this action is a path with $N-1$ edges. The action is free on edges and the vertex stabilizers are conjugates $\gamma\cyc{\TheOrder_\nu}\gamma^{-1}$ with $\gamma\in\freeprod$ and $1\leq\nu\leq\TheCone$. By the choice of a generator $\gamma_\nu$ for each $\cyc{\TheOrder_\nu}$ with $1\leq\nu\leq\TheCone$, the link of each vertex carries a cyclic ordering. 

Let us consider a proper embedding of the Bass--Serre tree~$T$ into $\CC$ that respects the local cyclic order on each link. If we choose a regular neighborhood of $T$ inside~$\CC$, we obtain a planar, contractible surface $\Sigma$ (with boundary), see Figure \textup{\ref{fig:constr_fund_domain}} for an example.

\begin{figure}[H]
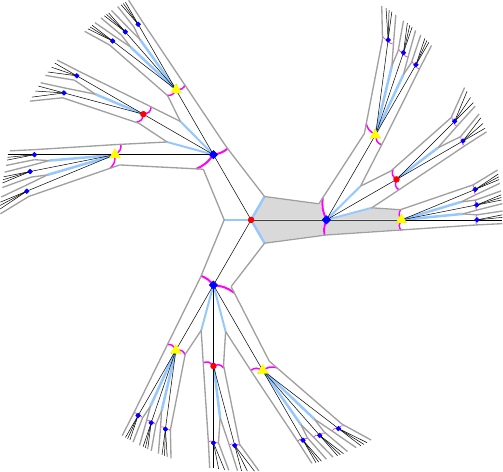
\caption{Thickened Bass--Serre tree for $\freeprod=\cyc{3}\ast\cyc{2}\ast\cyc{4}$ with fundamental domain shaded in gray. The red $\textcolor{red}{\bullet}$, blue~$\textcolor{blue}{\blacklozenge}$ and yellow $\textcolor{yellow}{\blacktriangle}$ vertices are conjugates of the free factors $\cyc{3}$, $\cyc{2}$ and $\cyc{4}$, respectively.}
\label{fig:constr_fund_domain}
\end{figure}

This surface $\Sigma$ inherits a proper $\freeprod$-action from the Bass--Serre tree such that vertex stabilizers act with respect to the cyclic order on the link of the stabilized vertex. Moreover, the action admits a fundamental domain corresponding to the fundamental domain in $T$. In particular, we obtain an orbifold structure $\Sigma_\freeprod$. 

A point in $\Sigma_\freeprod$ is a singular point if and only if it corresponds to a vertex of $T$. Hence, the singular points in $\Sigma_\freeprod$ are all cone points and decompose into $\TheCone$ orbits. The quotient $\Sigma/\freeprod$ is a disk with $\TheCone$ distinguished points that correspond to the orbits of the cone points. 

In general, we may choose a fundamental domain $F_\freeprod$ that is a disk as pictured in Figure \textup{\ref{fig:fund_domain}} and contains exactly $\TheCone$ cone points $\cp_1,...,\cp_\TheCone$ that lie on the boundary. 
\begin{figure}[H]
\resizebox{0.45\textwidth}{!}{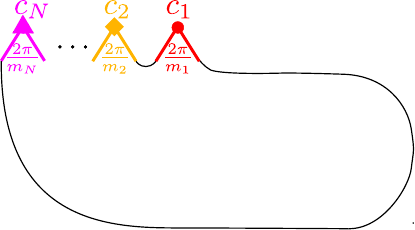}
\caption{}
\label{fig:fund_domain}
\end{figure}

If we remove the boundary of $\Sigma$, the quotient $\Sigma^\circ/\freeprod$ is homeomorphic to the complex plane with $\TheCone$ distinguished points and associated cyclic groups $\cyc{\TheOrder_\nu}$ for $1\leq\nu\leq\TheCone$. Adding $\freeprod$-orbits of punctures $\freeprod(r_\NPct)$ for $1\leq\NPct\leq\ThePct$ to $\Sigma$ such that $\freeprod(r_\Pc)\neq\freeprod(r_\NPct)$ for $\Pc\neq\NPct$, we obtain the orbifold called $\CC(\ThePct,\TheCone,\textbf{\TheOrder}) \text{ with } \textbf{\TheOrder}=(\TheOrder_1,...,\TheOrder_\TheCone)$ in \cite{Roushon2021}. Allcock \cite{Allcock2002} studies braids on these orbifolds for 
\[
(\ThePct,\TheCone,\textbf{\TheOrder})=(0,2,(2,2)), (0,1,(2)) \text{ and } (1,1,(2)). 
\]
\end{example}

\newcommand{\ConfnOrb}{\Conf_\TheStrand^{\G}(M_\G)}

The orbifold braid group $\Z_\TheStrand(M_\G)$ is defined as the orbifold fundamental group of the orbifold configuration space
\[
\ConfnOrb:=(M^\TheStrand\setminus\Delta_\TheStrand^\G(M))_{\G^\TheStrand\rtimes\Sym_\TheStrand} \; \text{ with }\; \Delta_\TheStrand^\G(M)=\{(x_1,...,x_\TheStrand)\in M^\TheStrand\mid x_\Str=\g(x_\Strand) \text{ for } \g\in\G, \; \Str\neq\Strand\}, 
\]
see \cite[Section 3]{Flechsig2023braid} for details. If $M_\G=\SG$, the elements in $\Z_\TheStrand(\Sigma_\freeprod(\ThePct))$ are specified by orbifold braid diagrams as in Figure \ref{fig:generators}. One easily proves, that $\Z_\TheStrand(\SG)$ is generated by $h_\Strand,\twP{\NPct}$ and $\twC{\nu}$ for $1\leq\Strand<\TheStrand, 1\leq\NPct\leq\ThePct$ and $1\leq\nu\leq\TheCone$. 
\begin{figure}[H]
\resizebox{0.55\textwidth}{!}{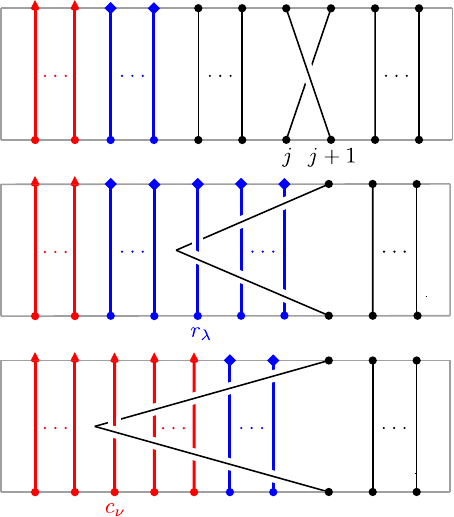}
\caption{Generators $h_\Strand$, $\twP{\NPct}$ and $\twC{\nu}$ (from top to bottom).}
\label{fig:generators}
\end{figure}

Moreover, we will use the following abbreviations for $1\leq\Str,\Strand,\NStrand\leq\TheStrand, \Str<\Strand$, $1\leq\NPct\leq\ThePct$ and $1\leq\nu\leq\TheCone$: 

\begin{align*}
\label{eq:def_a_ji}
\twist_{\Strand\Str}:= & h_{\Strand-1}^{-1}...h_{\Str+1}^{-1}h_\Str^2h_{\Str+1}...h_{\Strand-1}, \numbereq
\\
\label{eq:def_a_kr}
\twistP_{\NStrand\NPct}:= & h_{\NStrand-1}^{-1}...h_1^{-1}\twP{\NPct}h_1...h_{\NStrand-1} \text{ and } \numbereq 
\\
\label{eq:def_a_kc}
\twistC_{\NStrand\nu}:= & h_{\NStrand-1}^{-1}...h_1^{-1}\twC{\nu}h_1...h_{\NStrand-1}. \numbereq
\end{align*}

The corresponding braid diagrams are pictured in Figure \ref{fig:pure_generators}. 
\begin{figure}[H]
\resizebox{0.65\textwidth}{!}{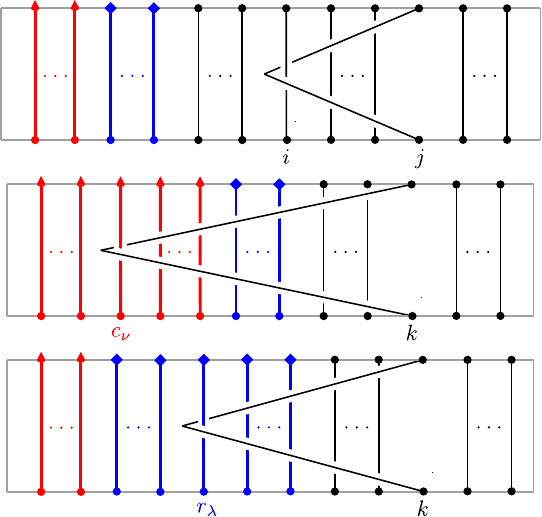}
\caption{Orbifold braids $\twist_{\Strand\Str}$, $\twistP_{\NStrand\NPct}$ and $\twistC_{\NStrand\nu}$ (from top to bottom).}
\label{fig:pure_generators}
\end{figure}

The orbifold braid diagrams in particular allow for all the transformations known from the Artin braid diagrams. Moreover, they satisfy additional relations $\twistC_{\NStrand\nu}^{\TheOrder_\nu}=1$ for each $1\leq\NStrand\leq\TheStrand$ and $1\leq\nu\leq\TheCone$, see \cite[Remark 3.12]{Flechsig2023braid} for details. In terms of braid diagrams this relation represents an additional orbifold Reidemeister move, see Figure \ref{fig:orb-Reidemeister-move}. 

\begin{figure}[H]
\centerline{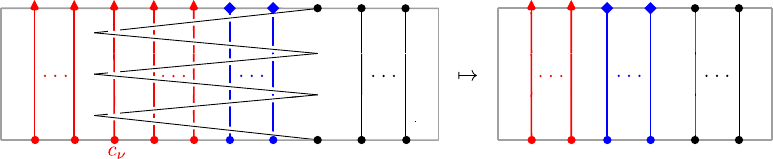}
\caption{Additional Reidemeister move for orbifold braid diagrams in $\Z_\TheStrand(\Sigma_\freeprod(\ThePct))$ for $\TheOrder_\nu=3$.}
\label{fig:orb-Reidemeister-move}
\end{figure}

Comparing the following presentation of $\Z_\TheStrand(\SG)$ with the well known presentation of the Artin braid group, shows that the orbifold Reidemeister moves are the only essential difference between orbifold braid groups and Artin braid groups. 

\begin{theorem}[{\cite[Corollary 5.7]{Flechsig2023braid}}]
\label{thm:pres_Z_n}
The orbifold braid group $\Z_\TheStrand(\SG)$ is presented by generators 
\[
h_\Strand, \twP{\NPct} \; \text{ and } \; \twC{\nu} \; \text{ for } \; 1\leq\Strand<\TheStrand, 1\leq\NPct\leq\ThePct \; \text{ and } \; 1\leq\nu\leq\TheCone 
\]
and the following defining relations for $2\leq\Strand<\TheStrand$, $1\leq\Pc,\NPct\leq\ThePct$ with $\Pc<\NPct$ and $1\leq\mu,\nu\leq\TheCone$ with $\mu<\nu$: 
\begin{enumerate}
\item \label{cor:pres_orb_braid_free_prod_rel1} 
$\twC{\nu}^{\TheOrder_\nu}=1$, 
\item \label{cor:pres_orb_braid_free_prod_rel2} 
$h_{\Strand-1}h_\Strand h_{\Strand-1}=h_\Strand h_{\Strand-1}h_\Strand$ and $[h_\NStrand,h_\NNStrand]=1$ \; for \; $1\leq\NStrand,\NNStrand<\TheStrand$ with $\vert\NStrand-\NNStrand\vert\geq2$, 
\item \label{cor:pres_orb_braid_free_prod_rel3} 
$[\twP{\NPct},h_\Strand]=1$ \; and \; $[\twC{\nu},h_\Strand]=1$, 
\item \label{cor:pres_orb_braid_free_prod_rel4} 
$[h_1\twP{\NPct}h_1,\twP{\NPct}]=1$ \; and \; $[h_1\twC{\nu}h_1,\twC{\nu}]=1$, 
\item \label{cor:pres_orb_braid_free_prod_rel5}
$[\twP{\Pc},\twistP_{2\NPct}]=1$, $[\twC{\mu},\twistC_{2\nu}]=1$ \; and \; $[\twP{\NPct},\twistC_{2\nu}]=1$
\\
with $\twistP_{2\NPct}=h_1^{-1}\twP{\NPct}h_1$ and $\twistC_{2\nu}=h_1^{-1}\twC{\nu} h_1$. 
\end{enumerate}
\end{theorem}

\begin{theorem}[{\cite[Corollary 5.6]{Flechsig2023braid}}]
\label{thm:pres_pure_free_prod}
The pure orbifold braid group $\PZ_\TheStrand(\Sigma_\freeprod(\ThePct))$ is presented by generators 
\[
\twist_{\Strand\Str}, \twistP_{\NStrand\NPct} \; \text{ and } \; \twistC_{\NStrand\nu}, \; \text{ for } \; 1\leq\Str,\Strand,\NStrand\leq\TheStrand, \Str<\Strand, 1\leq\NPct\leq\ThePct \; \text{ and } \; 1\leq\nu\leq\TheCone 
\]
and the following defining relations for $1\leq\Str,\Strand,\NStrand,\NNStrand\leq\TheStrand$ with $\Str<\Strand<\NStrand<\NNStrand$, $1\leq\Pc,\NPct\leq\ThePct$ with $\Pc<\NPct$ and $1\leq\mu,\nu,\leq\TheCone$ with $\mu<\nu$: 
\begin{enumerate}
\item 
\label{cor:pres_pure_free_prod_rel1}
$\twistC_{\NStrand\nu}^{\TheOrder_\nu}=1$, 
\item 
\label{cor:pres_pure_free_prod_rel2}
$[\twist_{\Strand\Str},\twist_{\NNStrand\NStrand}]=1$, \; 
$[\twistP_{\Strand\NPct},\twist_{\NNStrand\NStrand}]=1$ \; and \; 
$[\twistC_{\Strand\nu},\twist_{\NNStrand\NStrand}]=1$, 
\item 
\label{cor:pres_pure_free_prod_rel3}
$[\twist_{\NNStrand\Str},\twist_{\NStrand\Strand}]=1$, \; $[\twistP_{\NNStrand\NPct},\twist_{\NStrand\Strand}]=1$, \; 
$[\twistP_{\NNStrand\NPct},\twistP_{\NStrand\Pc}]=1$, \; 
$[\twistC_{\NNStrand\nu},\twist_{\NStrand\Strand}]=1$, \; 
$[\twistC_{\NNStrand\nu},\twistP_{\NStrand\NPct}]=1$ \; and \; 
$[\twistC_{\NNStrand\nu},\twistC_{\NStrand\mu}]=1$, 
\item 
\label{cor:pres_pure_free_prod_rel4}
$[\twist_{\NNStrand\NStrand}\twist_{\NNStrand\Strand}\twist_{\NNStrand\NStrand}^{-1},\twist_{\NStrand\Str}]=1$, \; 
$[\twist_{\NStrand\Strand}\twist_{\NStrand\Str}\twist_{\NStrand\Strand}^{-1},\twistP_{\Strand\NPct}]=1$, \; 
$[\twist_{\NStrand\Strand}\twistP_{\NStrand\Pc}\twist_{\NStrand\Strand}^{-1},\twistP_{\Strand\NPct}]=1$, 
\\
$[\twist_{\NStrand\Strand}\twist_{\NStrand\Str}\twist_{\NStrand\Strand}^{-1},\twistC_{\Strand\nu}]=1$ \; and \; 
$[\twist_{\NStrand\Strand}\twistC_{\NStrand\mu}\twist_{\NStrand\Strand}^{-1},\twistC_{\Strand\nu}]=1$, 
\item 
\label{cor:pres_pure_free_prod_rel5}
$\twist_{\NStrand\Strand}\twist_{\NStrand\Str}\twist_{\Strand\Str}=\twist_{\Strand\Str}\twist_{\NStrand\Strand}\twist_{\NStrand\Str}=\twist_{\NStrand\Str}\twist_{\Strand\Str}\twist_{\NStrand\Strand}$, 
\\
$\twist_{\Strand\Str}\twistP_{\Strand\NPct}\twistP_{\Str\NPct}=\twistP_{\Str\NPct}\twist_{\Strand\Str}\twistP_{\Strand\NPct}=\twistP_{\Strand\NPct}\twistP_{\Str\NPct}\twist_{\Strand\Str}$ and 
\\
$\twist_{\Strand\Str}\twistC_{\Strand\nu}\twistC_{\Str\nu}=\twistC_{\Str\nu}\twist_{\Strand\Str}\twistC_{\Strand\nu}=\twistC_{\Strand\nu}\twistC_{\Str\nu}\twist_{\Strand\Str}$. 
\end{enumerate}
\end{theorem}

That the group $\Z_\TheStrand(\Sigma_\freeprod(\ThePct))$ and its pure subgroup $\PZ_\TheStrand(\Sigma_\freeprod(\ThePct))$ are generated by the above generating sets follows as in the case of Artin braid groups, see \cite[Section 3]{Flechsig2023} for details. Moreover, a discussion of orbifold mapping class groups $\MapIdOrb{\TheStrand}{\SG}$ and an evaluation map $\ev:\MapIdOrb{\TheStrand}{\SG}\rightarrow\Z_\TheStrand(\Sigma_\freeprod(\ThePct))$ yields the defining relations from above, see \cite[Sections 5 and 6]{Flechsig2023braid} for details. 

\subsubsection{Alternative generating sets in the case $\TheCone+\ThePct\leq2$}
\label{subsubsec:alt_gen_sets}

We close this chapter giving a shorter description of the generating set from Theorem~\ref{thm:pres_Z_n} in three special cases. If $\Sigma$ is a disk without punctures on which the cyclic group of order $\TheOrder$ acts (see Example \ref{ex:good_orb_D_cyc_m}), the group $\Z_\TheStrand(D_{\cycm})$ is generated by $h_1,...,h_{\TheStrand-1}$ and $\twC{1}$. In this case, we shorten our notation using $\twsC$ instead of $\twC{1}$. The braid $\twsC$ projects to the diagram below.  
\begin{center}
\begingroup%
  \makeatletter%
  \providecommand\color[2][]{%
    \errmessage{(Inkscape) Color is used for the text in Inkscape, but the package 'color.sty' is not loaded}%
    \renewcommand\color[2][]{}%
  }%
  \providecommand\transparent[1]{%
    \errmessage{(Inkscape) Transparency is used (non-zero) for the text in Inkscape, but the package 'transparent.sty' is not loaded}%
    \renewcommand\transparent[1]{}%
  }%
  \providecommand\rotatebox[2]{#2}%
  \newcommand*\fsize{\dimexpr\f@size pt\relax}%
  \newcommand*\lineheight[1]{\fontsize{\fsize}{#1\fsize}\selectfont}%
  \ifx\svgwidth\undefined%
    \setlength{\unitlength}{96.969685bp}%
    \ifx\svgscale\undefined%
      \relax%
    \else%
      \setlength{\unitlength}{\unitlength * \real{\svgscale}}%
    \fi%
  \else%
    \setlength{\unitlength}{\svgwidth}%
  \fi%
  \global\let\svgwidth\undefined%
  \global\let\svgscale\undefined%
  \makeatother%
  \begin{picture}(1,0.71491647)%
    \lineheight{1}%
    \setlength\tabcolsep{0pt}%
    \put(0,0){\includegraphics[width=\unitlength,page=1]{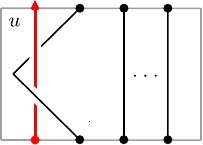}}%
  \end{picture}%
\endgroup%

\end{center}

\newpage

If $D$ contains additional punctures in $\cycm(r_1)$, the group $\Z_\TheStrand(D_{\cycm}(1))$ is generated by $h_1,...,h_{\TheStrand-1},\twC{1}$ and $\twP{1}$. Similar as above, we use $\twsP$ to abbreviate $\twP{1}$ and further we introduce $\bar{\twsC}:=h_{\TheStrand-1}^{-1}...h_1^{-1}\twC{1}h_1...h_{\TheStrand-1}\stackrel{\eqref{eq:def_a_kc}}=\twistC_{\TheStrand1}$. The element $\bar{\twsC}$ can be homotoped such that it fixes the first $\TheStrand-1$ strands. In this case, the braids $\twsP$ and $\bar{\twsC}$ project to the diagrams below. 
\begin{center}
\begingroup%
  \makeatletter%
  \providecommand\color[2][]{%
    \errmessage{(Inkscape) Color is used for the text in Inkscape, but the package 'color.sty' is not loaded}%
    \renewcommand\color[2][]{}%
  }%
  \providecommand\transparent[1]{%
    \errmessage{(Inkscape) Transparency is used (non-zero) for the text in Inkscape, but the package 'transparent.sty' is not loaded}%
    \renewcommand\transparent[1]{}%
  }%
  \providecommand\rotatebox[2]{#2}%
  \newcommand*\fsize{\dimexpr\f@size pt\relax}%
  \newcommand*\lineheight[1]{\fontsize{\fsize}{#1\fsize}\selectfont}%
  \ifx\svgwidth\undefined%
    \setlength{\unitlength}{237.53497319bp}%
    \ifx\svgscale\undefined%
      \relax%
    \else%
      \setlength{\unitlength}{\unitlength * \real{\svgscale}}%
    \fi%
  \else%
    \setlength{\unitlength}{\svgwidth}%
  \fi%
  \global\let\svgwidth\undefined%
  \global\let\svgscale\undefined%
  \makeatother%
  \begin{picture}(1,0.29458192)%
    \lineheight{1}%
    \setlength\tabcolsep{0pt}%
    \put(0,0){\includegraphics[width=\unitlength,page=1]{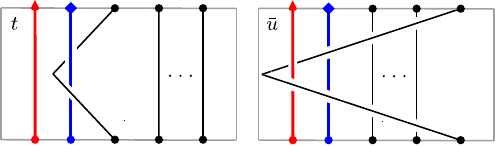}}%
  \end{picture}%
\endgroup%

\end{center}

If $\Sigma$ is the tree-shaped surface without punctures on which the group $\freeprodtwo$ acts (see Example \ref{ex:good_orb_free_prod}), the group $\Z_\TheStrand(\Sigma_{\freeprodtwo})$ is generated by $h_1,...,h_{\TheStrand-1},\twC{1}$ and $\twC{2}$. For the generator $\twC{1}$, we use our short notation $\twsC$ again. Instead of the generator $\twC{2}$, we consider the conjugate $\twsC':=h_{\TheStrand-1}^{-1}...h_1^{-1}\twC{2}h_1...h_{\TheStrand-1}\stackrel{(\ref{eq:def_a_kc})}=\twistC_{\TheStrand2}$. While $\twsC$ fixes the last $\TheStrand-1$ strands, $\twsC'$ is homotopic to a braid that fixes the first $\TheStrand-1$ strands. These braids project to the diagrams below.  
\begin{center}
\begingroup%
  \makeatletter%
  \providecommand\color[2][]{%
    \errmessage{(Inkscape) Color is used for the text in Inkscape, but the package 'color.sty' is not loaded}%
    \renewcommand\color[2][]{}%
  }%
  \providecommand\transparent[1]{%
    \errmessage{(Inkscape) Transparency is used (non-zero) for the text in Inkscape, but the package 'transparent.sty' is not loaded}%
    \renewcommand\transparent[1]{}%
  }%
  \providecommand\rotatebox[2]{#2}%
  \newcommand*\fsize{\dimexpr\f@size pt\relax}%
  \newcommand*\lineheight[1]{\fontsize{\fsize}{#1\fsize}\selectfont}%
  \ifx\svgwidth\undefined%
    \setlength{\unitlength}{237.53497319bp}%
    \ifx\svgscale\undefined%
      \relax%
    \else%
      \setlength{\unitlength}{\unitlength * \real{\svgscale}}%
    \fi%
  \else%
    \setlength{\unitlength}{\svgwidth}%
  \fi%
  \global\let\svgwidth\undefined%
  \global\let\svgscale\undefined%
  \makeatother%
  \begin{picture}(1,0.29458198)%
    \lineheight{1}%
    \setlength\tabcolsep{0pt}%
    \put(0,0){\includegraphics[width=\unitlength,page=1]{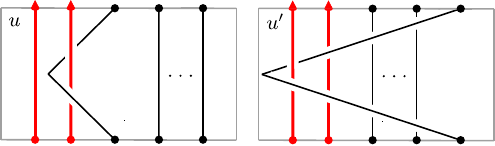}}%
  \end{picture}%
\endgroup%

\end{center}

Using the new notations, we obtain: 

\begin{corollary}
\label{cor:gen_sets}
\leavevmode
\begin{enumerate}
\item \label{cor:gen_sets_1fact} The orbifold braid group $\Z_\TheStrand(D_{\cycm})$ is generated by 
\[
h_1, ..., h_{\TheStrand-1} \; \text{ and } \; \twsC. 
\]
\item \label{cor:gen_sets_1fact_pct} The orbifold braid group $\Z_\TheStrand(D_{\cycm}(1))$ is generated by 
\[
h_1, ..., h_{\TheStrand-1}, \twsP \; \text{ and } \; \bar{\twsC}. 
\]
\item \label{cor:gen_sets_2fact} The orbifold braid group $\Z_\TheStrand(\Sigma_{\freeprodtwo})$ is generated by 
\[
h_1, ..., h_{\TheStrand-1}, \twsC \; \text{ and } \; \twsC'. 
\]
\end{enumerate}
\end{corollary}

\subsection{Artin groups and complex braid groups}

\textit{Artin groups} and \textit{complex braid groups} are both generalizations of the Artin braid groups $\B_\TheStrand$. While Artin groups generalize them on a combinatorial level, complex braid groups define a geometric analog of $\B_\TheStrand$. 

Artin groups are defined by the following data:  

\begin{definition}[Weighted graphs and Artin groups]
Let $\Delta$ be a finite graph with vertex set $V=\{v_1,...,v_k\}$ and a set of edges $E$. The tuple $(\Delta,W)$ is called a \textit{weighted graph} if $W$ is a subset of $\NN_{\geq3}\cup\{\infty\}$ with a map $E\rightarrow W,\{v_i,v_j\}\mapsto w_{ij}$. The number $w_{ij}$ is called the \textit{weight} of the edge $\{v_i,v_j\}$. We attach a group $A(\Delta,W)$ to the weighted graph $(\Delta,W)$ that is presented by 
\[
\langle v_1,...,v_k\mid \langle v_i,v_j\rangle_{w_{ij}}=\langle v_j,v_i\rangle_{w_{ij}} \text{ for } w_{ij}<\infty \text{ and } [v,w]=1 \text{ if } \{v,w\}\not\in E\rangle. 
\]
This group is called the \textit{Artin group} associated to the weighted graph $(\Delta,W)$. If we add the relation $v_i^2=1$ for each $1\leq i\leq k$, we obtain the associated \textit{Coxeter group}. An Artin group or a Coxeter group is called \textit{irreducible} if the underlying graph $\Delta$ is connected. In particular, in the chosen setting generators from different connected components of $\Delta$ commute. 
\end{definition}

Even if the definition of Artin groups and Coxeter groups are similar, the current state of research about them differs greatly: Coxeter groups are well-understood, see \cite{Humphreys1990} for an introduction. In contrast, only little is known about general Artin groups. In particular, the following basic questions remain open: 
\begin{itemize}
\item Is each Artin group torsion free? 
\item What is the center of an Artin group?
\item Does each Artin group have a solvable word problem? 
\item What can we say about the cohomology of an Artin group? 
\end{itemize}
An approach to solve these questions is to identify classes in the set of Artin groups that share a certain property. In particular, the following classes gained importance: 
\begin{itemize}
\item \textit{right-angled Artin groups} (also known as \textit{graph groups}), where all weights are $\infty$, 
\item \textit{spherical Artin groups}, where the associated Coxeter group is finite and 
\item \textit{affine Artin groups}, where the associated Coxeter group is affine, i.e.\ the Coxeter group is isomorphic to a group of affine reflections in a Euclidean space. 
\end{itemize}

After almost 50 years of research the above questions were answered for these classes: 
\begin{table}[H]
\begin{tabular}{l|c|c|c} 
Artin groups that are... & spherical & right-angled  & affine \\ \hline \hline
...are torsion free. & \cite{BrieskornSaito1972},\cite{Deligne1972} & \cite{CrispGodelleWiest2009} & \cite{McCammondSulway2017} \\ \hline
...have cyclic or trivial center. & \cite{BrieskornSaito1972},\cite{Deligne1972} & \cite{Servatius1989} & \cite{McCammondSulway2017} \\ \hline
...have solvable word problem. & \cite{BrieskornSaito1972},\cite{Deligne1972} & \cite{CrispGodelleWiest2009} & \cite{McCammondSulway2017} \\ \hline
...satisfy the $K(\pi,1)$ conjecture. & \cite{Deligne1972} & \cite{CrispGodelleWiest2009} & \cite{PaoliniSalvetti2020}
\end{tabular}
\end{table}
For further information about the research on the four basic questions from above, we refer to \cite{GodelleParis2012}. 

Let us discuss some examples of Artin groups. Therefore, we illustrate a graph by drawing a node for each vertex and a connecting arc between two nodes $v_i$ and $v_j$ if there exists an edge $\{v_i,v_j\}\in E$. The weights are denoted next to the edges using the convention that weights $3$ are omitted. 

Historically, the first example of an Artin group was the Artin braid group $\B_\TheStrand$. This group is associated to the weighted graph in Figure~\ref{fig:weighted_graph_B_n}, see, for instance, \cite[Theorem 1.12]{KasselTuraev2008}. 
\begin{figure}[H]
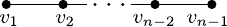
\caption{Weighted graph of the Artin braid group $\B_\TheStrand$.}
\label{fig:weighted_graph_B_n}
\end{figure}
The associated Coxeter group is the symmetric group $\Sym_\TheStrand$. This implies that $\B_\TheStrand$ is a spherical Artin group. Moreover, the Artin groups associated to the weighted graphs from Figure \ref{fig:weighted_graph_spher_aff} belong to the above classes of Artin groups. The Artin groups for $D_\TheStrand$ and $I_2(\TheOrder)$ are spherical and the Artin groups for $\tilde{D}_\TheStrand$ and $\tilde{B}_\TheStrand$ are affine. These Artin groups will pop up in the following discussion of orbifold braid groups. 
\begin{figure}
\centerline{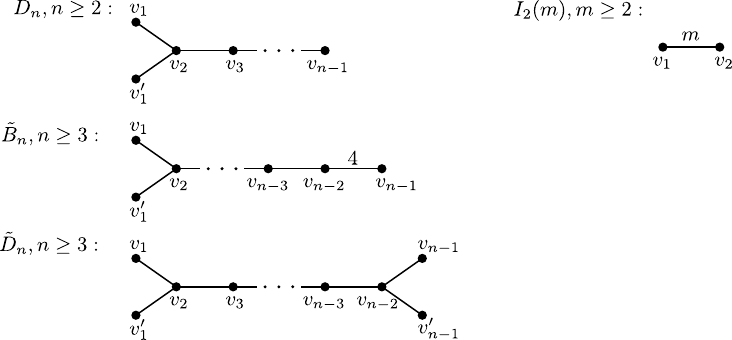}
\caption{Further examples of weighted graphs.}
\label{fig:weighted_graph_spher_aff}
\end{figure}

Before we step into this discussion, we introduce \textit{complex braid groups} as the following geometric generalization of Artin braid groups: 

\begin{definition}[Complex reflection groups and complex braid groups]
\label{def:compl_refl_braid_grp}
Let $V$ be a finite dimensional complex vector space. A group $W$ is called a \textit{complex reflection group} if it is a finite subgroup of $\text{GL}(V)$ that is generated by a subset $R$ such that $H_r:=\ker(r-\id_V)$ is of codimension one for each $r\in R$. The fundamental group $\pi_1((V\setminus\bigcup_{r\in R}H_r)/W)$ is called the associated \textit{complex braid group}. 
\end{definition}

\begin{remark}
Although the name complex reflection group suggests it, we do not require that the elements in $R$ are reflections in a geometric sense. However, as in the case of Artin braid groups, we obtain a short exact sequence 
\[
1\rightarrow\pi_1((V\setminus\bigcup_{r\in R}H_r))\rightarrow\pi_1((V\setminus\bigcup_{r\in R}H_r)/W)\rightarrow W\rightarrow1. 
\]
\end{remark}

For an introduction to complex reflection groups, we refer to \cite[Chapter 1]{LehrerTaylor2009}. An introduction to complex braid groups is given in \cite[Chapter 4]{Broue2010}. 

In particular, the complex braid groups include the Artin braid group $\B_\TheStrand$: 

\begin{example}
For $V=\CC^\TheStrand$, the group $W=\Sym_\TheStrand$ acts on $\CC^\TheStrand$ by permutation of coordinates. A generating set of $W$ is given by the set of transpositions 
\[
R=\{(\sigma_{\Str\Strand})_{1\leq\Str,\Strand\leq\TheStrand,\Str<\Strand}\mid \sigma_{\Str\Strand}(\Str)=\Strand, \sigma_{\Str\Strand}(\Strand)=\Str \;\text{ and }\; \sigma_{\Str\Strand}(\NStrand)=\NStrand \;\text{ for }\; \NStrand\neq\Str,\Strand\}. 
\]
The kernel $\ker(\sigma_{\Str\Strand}-\id_{\CC^\TheStrand})$ coincides with $\{(x_1,\dots,x_\TheStrand)\in\CC^\TheStrand\mid x_\Str=x_\Strand\}$. Hence, the associated space $(\CC^n\setminus\bigcup_{1\leq\Str<\Strand\leq\TheStrand}H_{\sigma_{\Str\Strand}})/\Sym_\TheStrand$ coincides with the $\TheStrand$-th configuration space $\Conf_\TheStrand(\CC)$, i.e.\ the fundamental group of this space is the braid group $\B_\TheStrand$. 
\end{example}

Moreover, the complex braid groups include the following groups, see \cite[Section~2]{LehrerTaylor2009} for details. 

\begin{example}
Let $\TheOrder,p,\TheStrand$ be three integer parameters such that $p$ divides $\TheOrder$. We associate a group of complex $\TheStrand\times\TheStrand$ matrices such that 
\begin{itemize}
\item each matrix is \text{monomial}, i.e.\ each column and each row contains exactly one non-zero entry, 
\item for each $1\leq\Strand\leq\TheStrand$ the non-zero entry in the $\Strand$-th column is $\vartheta^{k_\Strand}$, where $\vartheta$ is a primitive $\TheOrder$-th root of unity and 
\item the sum of exponents $\sum_{\Strand=1}^\TheStrand k_\Strand\equiv0$ (mod $p$). 
\end{itemize}
We denote the group by $G(\TheOrder,p,\TheStrand)$. This group is a semidirect product such that the normal subgroup is a group of diagonal matrices with diagonal given by $(\vartheta^{k_1},...,\vartheta^{k_\TheStrand})$ and the quotient is $\Sym_\TheStrand$. The action of $\Sym_\TheStrand$ is given by permutation of coordinates. Consequently, the group is generated by the block diagonal matrices 
\[
M_{\vartheta^p}:=(\vartheta^p,1,...,1) \text{ for } p\neq\TheOrder \text{ \; and \; } M_{(12)}^\vartheta:=\left(\begin{pmatrix}
0 & \vartheta
\\
\vartheta^{-1} & 0
\end{pmatrix},1,...,1\right)  
\]
and the permutation matrices $\sigma_{\Str\Strand}$ for $1\leq\Str<\Strand\leq\TheStrand$. Since 
\begin{align*}
\ker(M_{\vartheta^p}-\id_{\CC^\TheStrand}) & =\{(0,x_2,...,x_\TheStrand)\mid x_\Strand\in\CC \text{ for } 2\leq\Strand\leq\TheStrand\} \text{ and }
\\
\ker(M_{(12)}^\vartheta-\id_{\CC^\TheStrand}) & =\{(\vartheta x_2,x_2,...,x_\TheStrand)\mid x_\Strand\in\CC \text{ for } 2\leq\Strand\leq\TheStrand\}, 
\end{align*}
these generators satisfy the condition from Definition \textup{\ref{def:compl_refl_braid_grp}}, i.e.\ this group is a complex reflection group. The complex braid group associated to $G(\TheOrder,p,\TheStrand)$ is denoted by $\B(\TheOrder,p,\TheStrand)$. 
\end{example}

For suitable parameters, we will describe a presentation of $\B(\TheOrder,p,\TheStrand)$. This will allow us to identify complex braid groups inside orbifold braid groups by their presentations. We use the following notation: 

\begin{definition}
Let $(\Delta,W)$ be a weighted graph. If a subgraph $\Delta'\leq\Delta$ with three vertices $v_r,v_s,v_t$ is complete with weights $\TheOrder_{rs}=\TheOrder_{rt}=3$, we may add a relation 
\[
(v_rv_sv_t)^2=(v_sv_tv_r)^2
\]
in the presentation of $A(\Delta,W)$. The quotient group of $A(\Delta,W)$ with such additional relations is denoted by $\tilde{A}(\Delta,W)$. If the weight of the edge $\{v_r,v_s\}$ is $\TheOrder$, the additional relation in the weighted graph is illustrated by two additional edges adjacent to $v_r$ pointing towards the edge $\{v_s,v_t\}$: 
\begin{center}
\begingroup%
  \makeatletter%
  \providecommand\color[2][]{%
    \errmessage{(Inkscape) Color is used for the text in Inkscape, but the package 'color.sty' is not loaded}%
    \renewcommand\color[2][]{}%
  }%
  \providecommand\transparent[1]{%
    \errmessage{(Inkscape) Transparency is used (non-zero) for the text in Inkscape, but the package 'transparent.sty' is not loaded}%
    \renewcommand\transparent[1]{}%
  }%
  \providecommand\rotatebox[2]{#2}%
  \newcommand*\fsize{\dimexpr\f@size pt\relax}%
  \newcommand*\lineheight[1]{\fontsize{\fsize}{#1\fsize}\selectfont}%
  \ifx\svgwidth\undefined%
    \setlength{\unitlength}{46.02101289bp}%
    \ifx\svgscale\undefined%
      \relax%
    \else%
      \setlength{\unitlength}{\unitlength * \real{\svgscale}}%
    \fi%
  \else%
    \setlength{\unitlength}{\svgwidth}%
  \fi%
  \global\let\svgwidth\undefined%
  \global\let\svgscale\undefined%
  \makeatother%
  \begin{picture}(1,0.98515234)%
    \lineheight{1}%
    \setlength\tabcolsep{0pt}%
    \put(0,0){\includegraphics[width=\unitlength,page=1]{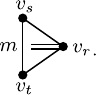}}%
  \end{picture}%
\endgroup%

\end{center}
\end{definition}

In particular, we are interested in the groups induced by the following weighted graphs with additional relations: 
\begin{center}
\centerline{
\begingroup%
  \makeatletter%
  \providecommand\color[2][]{%
    \errmessage{(Inkscape) Color is used for the text in Inkscape, but the package 'color.sty' is not loaded}%
    \renewcommand\color[2][]{}%
  }%
  \providecommand\transparent[1]{%
    \errmessage{(Inkscape) Transparency is used (non-zero) for the text in Inkscape, but the package 'transparent.sty' is not loaded}%
    \renewcommand\transparent[1]{}%
  }%
  \providecommand\rotatebox[2]{#2}%
  \newcommand*\fsize{\dimexpr\f@size pt\relax}%
  \newcommand*\lineheight[1]{\fontsize{\fsize}{#1\fsize}\selectfont}%
  \ifx\svgwidth\undefined%
    \setlength{\unitlength}{226.88217321bp}%
    \ifx\svgscale\undefined%
      \relax%
    \else%
      \setlength{\unitlength}{\unitlength * \real{\svgscale}}%
    \fi%
  \else%
    \setlength{\unitlength}{\svgwidth}%
  \fi%
  \global\let\svgwidth\undefined%
  \global\let\svgscale\undefined%
  \makeatother%
  \begin{picture}(1,0.74797522)%
    \lineheight{1}%
    \setlength\tabcolsep{0pt}%
    \put(0,0){\includegraphics[width=\unitlength,page=1]{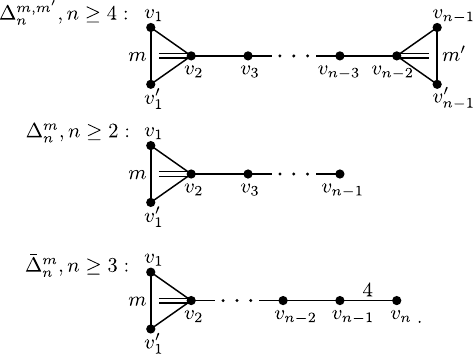}}%
  \end{picture}%
\endgroup%
}
\end{center}

We are interested in the associated groups $\tilde{A}(\Delta_\TheStrand^{\TheOrder,\TheOrder'}),\tilde{A}(\Delta_\TheStrand^\TheOrder)$ and $\tilde{A}(\bar{\Delta}_\TheStrand^\TheOrder)$ because they present complex braid groups for suitable $\TheStrand,\TheOrder,\TheOrder'\in\NN$. For $\TheStrand=2$, the group $\tilde{A}(\Delta_2^\TheOrder)$ coincides with the Artin group of type $I_2(\TheOrder)$. 

\begin{theorem}
The following isomorphisms are known for complex braid groups: 
\begin{align*}
\B(\TheOrder,\TheOrder,\TheStrand) & \cong \tilde{A}(\Delta_\TheStrand^\TheOrder) && \text{ for } \TheStrand\geq2, \text{ see \cite[Section 3.C]{BroueMalleRouquier1998},} 
\\
\B(3,3,\TheStrand) & \cong \tilde{A}(\Delta_\TheStrand^{3,3}) && \text{ for } \TheStrand\geq4, \text{ see \cite[Table I.A]{Malle1996},}  
\\ 
\B(6,3,\TheStrand)_1 & \cong \tilde{A}(\bar{\Delta}_\TheStrand^3) && \text{ for } \TheStrand\geq3, \text{ see  \cite[Table I.A]{Malle1996},}  
\\ 
\B(6,6,\TheStrand) & \cong \tilde{A}(\Delta_\TheStrand^{3,2}) && \text{ for } \TheStrand\geq4, \text{ see \cite[Table I.A]{Malle1996},} 
\\ 
\B(4,2,\TheStrand)_1^\ast & \cong \tilde{A}(\bar{\Delta}_\TheStrand^4) && \text{ for } \TheStrand\geq3, \text{ see  \cite[Table I.B]{Malle1996},}  
\\  
\B(4,4,\TheStrand) & \cong \tilde{A}(\Delta_\TheStrand^{4,2}) && \text{ for } \TheStrand\geq4, \text{ see \cite[Table I.B]{Malle1996}.} 
\end{align*}
Here above, $\B(6,3,\TheStrand)_1$ is the complex braid group associated to $G(6,3,\TheStrand)_1$, i.e.\ a subgroup of $G(6,1,\TheStrand)$, see \cite[Lemma 3.3]{Malle1996} for details, and $\B(4,2,\TheStrand)_1^\ast$ is the complex braid group associated to $G(4,2,\TheStrand)_1^\ast$, i.e.\ a subgroup of an extension of $G(4,1,\TheStrand)$ by a suitable lattice, see \cite[p.\ 263, Lemma 3.2]{Malle1996} for details. 
\end{theorem}

\subsection{Semidirect products}

\begin{definition}
\label{def:semidir_prod}
A group $G$ is a \textit{semidirect product} with \textit{normal subgroup}~$N$ and \textit{quotient} $H$ if there exists a short exact sequence 
\[
1\rightarrow N\xrightarrow{\iota}G\xrightarrow{\pi}H\rightarrow1
\]
that has a 
section $s:H\rightarrow G$. In this case, we denote $G=N\rtimes H$. 
\end{definition}

In the following, presentations of groups will be an important tool for us. In particular, presentations allow us to define group homomorphisms by assignments defined on generating sets. 

\begin{definition}
Let $G$ be a group with presentation 
\[
\langle X\mid R\rangle=\langle x_1,...,x_k\mid r_1=s_1,...,r_l=s_l\rangle 
\]
and $H$ a group generated by a set of elements $\{y_1,...,y_p\}$ with $p\geq k$. Moreover, let us assume that the words $r_j$ and $s_j$ are given by $x_{j_1}^{\varepsilon_1}...x_{j_q}^{\varepsilon_q}$ and $x_{\tilde{j}_1}^{\delta_1}...x_{\tilde{j}_r}^{\delta_r}$, respectively. Given \textit{assignments} $\phi:x_i\mapsto y_i$ for $1\leq i\leq k$, we apply them letterwise to words mapping $x_i^{-1}$ to $y_i^{-1}$. We say that the assignments $\phi$ \textit{preserve the relations in $R$} if the relation $y_{j_1}^{\varepsilon_1}...y_{j_q}^{\varepsilon_q}=y_{\tilde{j}_1}^{\delta_1}...y_{\tilde{j}_r}^{\delta_r}$ is valid in $H$ for each $1\leq j\leq l$. 
\end{definition}

\begin{theorem}[{von Dyck}, {\cite[p.\ 346]{Rotman2012}}]
\label{thm:von_Dyck}
Let $G$ be a group with presentation $\langle X\mid R\rangle$ as above and $H$ a group generated by $\{y_1,...,y_p\}$ with $p\geq k$. If the assignments 
\[
\phi:x_i\mapsto y_i \;\text{ for }\; 1\leq i\leq k
\]
preserve the relations in $R$, then these assignments induce a homomorphism 
\[
\phi:G\rightarrow H. 
\]
\end{theorem}

\begin{lemma}[{\cite[Lemma 5.17]{Flechsig2023}}]
\label{lem:semidir_prod_pres}
Let $N$ and $H$ be groups given by presentations $N=\langle X\mid R\rangle$ and $H=\langle Y\mid S\rangle$. Then the following are equivalent: 
\begin{enumerate}
\item \label{lem:semidir_prod_pres_it1}
$G$ is a semidirect product with normal subgroup $N$ and quotient $H$. 
\item \label{lem:semidir_prod_pres_it2}
$G$ has a presentation 
\[
G=\langle X,Y\mid R,S,y^{\pm1}xy^{\mp1}=\phi_{y^{\pm1}}(x) \text{ for all } x\in X, y\in Y\rangle
\]
such that $\phi_{y^{\pm1}}(x)$ is a word in the alphabet~$X$ for all $x\in X$ and $y\in Y$. Moreover, for each $y\in Y$, the assignments 
\begin{equation}
\label{lem:semidir_prod_it2_cond_auto}
x\mapsto\phi_y(x)
\end{equation}
induce an automorphism $\phi_y\in\Aut(N)$ and the assignments 
\begin{equation}
\label{lem:semidir_prod_it2_cond_homo}
y\mapsto\phi_y 
\end{equation}
induce a homomorphism $H\rightarrow\Aut(N)$. 
\end{enumerate}
\end{lemma}

\begin{remark}
\label{rem:semidir_prod_pres}
Lemma \ref{lem:semidir_prod_pres} will be applied in the rest of the article several times. We will apply it given a group $G$ with a presentation to deduce that $G$ is a semidirect product. In this case, we want to show that the given presentation satisfies the conditions from Lemma \ref{lem:semidir_prod_pres_it2}. 

The base to prove that is to divide the generating set into two disjoint subsets $X$ and $Y$ such that $X$ generates the normal subgroup and $Y$ generates the quotient. 

Further, we divide the relations into three disjoint subsets $R, S$ and $C$ such that~$R$ contains all relations in letters from $X$, $S$ contains all relations in letters from $Y$ and $C$ contains all the remaining relations. In particular, the relations in $C$ should be given in the form $y^{\pm1}xy^{\mp1}=\phi_{y^{\pm1}}(x)$ for all $x\in X$ and $y\in Y$. To deduce a semidirect product structure, it remains to check that the relations from $C$ satisfy the conditions on the assignments \eqref{lem:semidir_prod_it2_cond_auto} and \eqref{lem:semidir_prod_it2_cond_homo}. 

It is reasonable to check these conditions in the following order: 

\begin{step}
\label{rem:semidir_prod_pres_step1}
Using Theorem \ref{thm:von_Dyck}, the first step will always be to check that the assignments $\phi:y\mapsto\phi_y$ from \eqref{lem:semidir_prod_it2_cond_homo} preserve the relations from $S$. If $S$ contains a relation $y_1^{\varepsilon_1}...y_q^{\varepsilon_q}=\tilde{y}_1^{\delta_1}...\tilde{y}_r^{\delta_r}$, this requires that the assignments
\[
x\mapsto\phi_{y_1^{\varepsilon_1}}\circ...\circ\phi_{y_q^{\varepsilon_q}}(x) \; \text{ and } \; x\mapsto\phi_{\tilde{y}_1^{\delta_1}}\circ...\circ\phi_{\tilde{y}_r^{\delta_r}}(x)
\]
coincide on each letter $x\in X$ (up to relations in $R$). 

In particular, we may check if $\phi$ induces a homomorphism independently of the fact if $x\mapsto\phi_y(x)$ induces an automorphism of the group presented by $\langle X\mid R\rangle$. 
\end{step}

\begin{step}
\label{rem:semidir_prod_pres_step2}
In the second step, we check that the assignments $\phi_y:x\mapsto\phi_y(x)$ induce an automorphism of the group presented by $\langle X\mid R\rangle$ for all $y\in Y$. 

To apply Theorem \ref{thm:von_Dyck}, we check if for all $y\in Y$ the assignments $\phi_y$ preserve the relations from $R$. If this is the case, the assignments $\phi_y$ induce an endomorphism of the group presented by $\langle X\mid R\rangle$. By the first step, we further have 
\[
\phi_{y^{-1}}\circ\phi_y=\id_{\langle X\mid R\rangle}=\phi_y\circ\phi_{y^{-1}}, 
\]
i.e.\ the endomorphism induced by $x\mapsto\phi_y(x)$ is bijective and therefore an automorphism of the group presented by $\langle X\mid R\rangle$. 
\end{step}
\end{remark}

\section{Relating orbifold braids and complex braid groups}

\label{sec:general_Allcock}
\renewcommand{\Twist}[2]{A_{{#1}{#2}}}
\renewcommand{\TwistP}[2]{B_{{#1}{#2}}}
\renewcommand{\TwistC}[2]{C_{{#1}{#2}}}
\renewcommand{\twist}[2]{a_{{#1}{#2}}}
\renewcommand{\twistP}[2]{b_{{#1}{#2}}}
\renewcommand{\twistC}[2]{c_{{#1}{#2}}}

In \cite{Allcock2002}, Allcock proved the following structural result on orbifold braid groups: 
\begin{align*}
\Z_\TheStrand(D_{\cyc{2}}) & = A(D_\TheStrand)\rtimes\cyc{2}, 
\\
\Z_\TheStrand(D_{\cyc{2}}(1)) & = A(\tilde{B}_\TheStrand)\rtimes\cyc{2} \text{ and }
\\
\Z_\TheStrand(\Sigma_{\cyc{2}\ast\cyc{2}}) & = A(\tilde{D}_\TheStrand)\rtimes(\cyc{2}\times\cyc{2}). 
\end{align*}
Allcock's proofs are based on results of Brieskorn \cite{Brieskorn1971} and Nguyen \cite{Nguyen1983} about orbit spaces of spherical and affine reflection groups, respectively. This restricts the approach to orbifold braid groups with cone points of order two such that the sum of the number of cone points and punctures is at most two. 

In contrast, the following characterizations are based on the presentation of the orbifold braid group $\Z_\TheStrand(\Sigma_\freeprod(\ThePct))$ given in Theorem \ref{thm:pres_Z_n}. This allows us to relax the condition on the order of the cone points. The restriction to the case where the sum of the number of cone points~$\TheCone$ and the number of punctures~$\ThePct$ is at most two reflects the geometry of the braids: the strands are ordered in a row with two ends and we can either attach a cone point or a puncture at each of the two ends. Thus, we have to keep this restriction.

For the rest of the chapter, let $\Z_\TheStrand(\Sigma_\freeprod(\ThePct))$ be an orbifold braid group with at most two cone points of order $\TheOrder:=\TheOrder_1$ and $\TheOrder':=\TheOrder_2$ such that $\TheCone+\ThePct\leq2$. 

Besides the presentation of $\Z_\TheStrand(\Sigma_\freeprod(\ThePct))$ from Theorem \ref{thm:pres_Z_n}, we recall the discussion from Section \ref{subsubsec:alt_gen_sets}. There we introduced the abbreviations 
\[
\twsP=\twP{1}, \; \twsC=\twC{1}, \; \twsC'=\twistC{\TheStrand}{2} \; \text{ and } \; \bar{\twsC}=\twistC{\TheStrand}{1} 
\]
for certain torsion elements. In terms of these and further elements, we gave alternative generating sets in the cases with one or two cone points (see Corollary~\ref{cor:gen_sets}): 
\begin{align*}
\Z_\TheStrand(D_{\cycm}) & = \langle h_1,...,h_{\TheStrand-1},\twsC\rangle, 
\\
\Z_\TheStrand(D_{\cycm}(1)) & = \langle h_1,...,h_{\TheStrand-1},\twsP,\bar{\twsC}\rangle \text{ and } 
\\
\Z_\TheStrand(\Sigma_{\freeprodtwo}) & = \langle h_1,...,h_{\TheStrand-1},\twsC,\twsC'\rangle. 
\end{align*}

For the following characterization, we add further redundant generators to these sets. In the first case, we add $\htwC:=\twsC h_1\twsC^{-1}$, in the second we add $h_{\TheStrand-1}^{\twsC}:=\bar{\twsC}h_{\TheStrand-1}\bar{\twsC}^{-1}$ and in the third case we adjoin $\htwC$ and $\htw:=\twsC'h_{\TheStrand-1}\twsC'^{-1}$. In particular, in each of the three groups the generators $h_1,...,h_{\TheStrand-1}$ and the additional generators $\htwC, h_{\TheStrand-1}^{\twsC}$ and $\htw$, respectively, satisfy the following \textit{braid and commutator relations}: 
\begin{itemize}
\item $h_\Str^\ast h_{\Str+1}^\times h_\Str^\ast=h_{\Str+1}^\times h_\Str^\ast h_{\Str+1}^\times$ \; for \; $1\leq\Str\leq\TheStrand-2$, 
\item $[h_\Strand^\ast,h_\NStrand^\times]=1$ \; for \; $1\leq\Strand,\NStrand<\TheStrand,\vert\Strand-\NStrand\vert\geq2$
\end{itemize}
where $\ast,\times$ stand either for $\twsC, \twsC'$ or no exponent. 

On the basis of the known presentation, we will determine another presentation in terms of the new generators. For each of the three groups, this presentation will satisfy the conditions from Lemma~\ref{lem:semidir_prod_pres_it2}. Particularly, we will obtain the structure of a semidirect product with the quotient generated by the torsion elements of the generating set, and the normal subgroup generated by $h_1,...,h_{\TheStrand-1}$ and their conjugates. In the end, it turns out that the normal subgroup in many cases carries the structure of an Artin group or a complex braid group. 

We begin with a technical lemma, which records additional relations in terms of the generators from Theorems \ref{thm:pres_Z_n} and \ref{thm:pres_pure_free_prod}. These additional relations will allow us to deduce a presentation of $\Z_\TheStrand(\Sigma_\freeprod(\ThePct))$ in terms of the generators from Corollary~\ref{cor:gen_sets} and the redundant generators introduced above (see Propositions~\ref{prop:sec_pres_free_prod_2fact} and \ref{prop:sec_pres_free_prod_subgrp_1fact}). 

The labels with letters R, S and C in the following lemma already allude to the role the relations will play with respect to the identification of the semidirect product structure: The relations with label R will be relations in the normal subgroup and the relations with label S will be relations in the quotient of the semidirect product. The remaining relations with label C will describe the conjugation of the quotient on the normal subgroup. 

In the following lemma, the notation $h_\Strand$ for $1\leq\Strand<\TheStrand$ refers to the orbifold braid defined by the first diagram in Figure \ref{fig:generators}. The generators $\twistC{\TheStrand}{1}$, $\twistC{\TheStrand}{2}$, $\twC{1}=\twistC{1}{1}$ and $\twP{1}=\twistP{1}{1}$ are defined with respect to \eqref{eq:def_a_kr} and \eqref{eq:def_a_kc}. The diagrams of the generators $\twistC{\NStrand}{\nu}$ and $\twistP{\NStrand}{\NPct}$ are given in Figure \ref{fig:pure_generators}. 

\begin{lemma}
\label{lem:psi_rel}
For $\TheStrand\geq3$, the relations from Theorems \textup{\ref{thm:pres_Z_n}} and \textup{\ref{thm:pres_pure_free_prod}} imply that the elements $\twC{1},\twP{1},h_1,...,h_{\TheStrand-1},\twistC{\TheStrand}{1}$ and $\twistC{\TheStrand}{2}$ in $\Z_\TheStrand(\Sigma_\freeprod(\ThePct))$ satisfy the following relations: 
\begin{enumerate}[label={\textup{(R\arabic*)}},ref={\thetheorem(R\arabic*)}]
\item \label{lem:psi_rel_R1}
$\langle h_1,\twC{1}h_1\twC{1}^{-1}\rangle_\TheOrder=\langle\twC{1}h_1\twC{1}^{-1},h_1\rangle_\TheOrder$, 
\\
where $\langle a,b\rangle_k$ denotes the word $aba...$ with $k$ letters, 
\item \label{lem:psi_rel_R2}
$\langle h_{\TheStrand-1},\twistC{\TheStrand}{1}h_{\TheStrand-1}\twistC{\TheStrand}{1}^{-1}\rangle_\TheOrder=\langle \twistC{\TheStrand}{1}h_{\TheStrand-1}\twistC{\TheStrand}{1}^{-1},h_{\TheStrand-1}\rangle_\TheOrder$ and 
\\
$\langle h_{\TheStrand-1},\twistC{\TheStrand}{2}h_{\TheStrand-1}\twistC{\TheStrand}{2}^{-1}\rangle_{\TheOrder'}=\langle \twistC{\TheStrand}{2}h_{\TheStrand-1}\twistC{\TheStrand}{2}^{-1},h_{\TheStrand-1}\rangle_{\TheOrder'}$, 
\item \label{lem:psi_rel_R3}
$(h_1\twC{1}h_1\twC{1}^{-1}h_2)^2=(h_2h_1\twC{1}h_1\twC{1}^{-1})^2$, 
\item \label{lem:psi_rel_R4}
$(h_{\TheStrand-1}\twistC{\TheStrand}{\nu}h_{\TheStrand-1}\twistC{\TheStrand}{\nu}^{-1}h_{\TheStrand-2})^2=(h_{\TheStrand-2}h_{\TheStrand-1}\twistC{\TheStrand}{\nu}h_{\TheStrand-1}\twistC{\TheStrand}{\nu}^{-1})^2$ \; for \;  $1\leq\nu\leq2$, 
\end{enumerate}
\begin{enumerate}[label={\textup{(S\arabic*)}},ref={\thetheorem(S\arabic*)}]
\item \label{lem:psi_rel_S1} 
$\twistC{\TheStrand}{1}^\TheOrder=1=\twistC{\TheStrand}{2}^{\TheOrder'}$, 
\item \label{lem:psi_rel_S2}
$[\twC{1},\twistC{\TheStrand}{2}]=1$, 
\end{enumerate}
\begin{enumerate}[label={\textup{(C\arabic*)}},ref={\thetheorem(C\arabic*)}]
\item \label{lem:psi_rel_C1}
$[h_\Strand,\twistC{\TheStrand}{\nu}]=1$ \; for \; $1\leq\Strand\leq\TheStrand-2$ \; and \; $1\leq\nu\leq2$, 
\item \label{lem:psi_rel_C2}
$[h_{\TheStrand-1}\twistC{\TheStrand}{\nu}h_{\TheStrand-1},\twistC{\TheStrand}{\nu}]=1$ \; for \; $1\leq\nu\leq2$ and 
\item \label{lem:psi_rel_C3}
$[\twP{1},\twistC{\TheStrand}{1}]=1$. 
\end{enumerate}
\end{lemma}
\begin{proof}
The relations \ref{lem:psi_rel_S1}, \ref{lem:psi_rel_S2} and \ref{lem:psi_rel_C3} only involve pure braids. The relation \ref{lem:psi_rel_S1} follows directly from \ref{cor:pres_pure_free_prod_rel1}, 
the relations \ref{lem:psi_rel_S2} and \ref{lem:psi_rel_C3} are covered by \ref{cor:pres_pure_free_prod_rel3}. 

Using the definition of $\twistC{\TheStrand}{\nu}$ from \eqref{eq:def_a_kc}, the relation \ref{lem:psi_rel_C2} is equivalent to 
\begin{align*}
& h_{\TheStrand-1}h_{\TheStrand-1}^{-1}
h_{\TheStrand-2}^{-1}...h_1^{-1}\twC{\nu}h_1...h_{\TheStrand-2}h_{\TheStrand-1}h_{\TheStrand-1}h_{\TheStrand-1}^{-1}h_{\TheStrand-2}^{-1}...h_1^{-1}\twC{\nu}h_1...h_{\TheStrand-1}
\\
= & h_{\TheStrand-1}^{-1}...h_1^{-1}\twC{\nu}h_1...h_{\TheStrand-1}h_{\TheStrand-1}h_{\TheStrand-1}^{-1}h_{\TheStrand-2}^{-1}...h_1^{-1}\twC{\nu}h_1...h_{\TheStrand-2}h_{\TheStrand-1}h_{\TheStrand-1}. 
\end{align*}
On the other hand, using $\twist{\TheStrand,}{\TheStrand-1}\stackrel{\eqref{eq:def_a_ji}}=h_{\TheStrand-1}^2$ and $\twistC{\Strand}{\nu}\stackrel{\eqref{eq:def_a_kc}}=h_{\Strand-1}^{-1}...h_1^{-1}\twC{\nu}h_1...h_{\Strand-1}$, the above relation also reads   
\[
\twistC{\TheStrand-1,}{\nu}\twist{\TheStrand,}{\TheStrand-1}\twistC{\TheStrand}{\nu}=\twistC{\TheStrand}{\nu}\twistC{\TheStrand-1,}{\nu}\twist{\TheStrand,}{\TheStrand-1}. 
\]
This is covered by relation \ref{cor:pres_pure_free_prod_rel5}. 

For $1\leq\Strand\leq\TheStrand-2$, the following observation yields relation~\ref{lem:psi_rel_C1}: 
\begin{align*}
h_\Strand \twistC{\TheStrand}{\nu} \mystackrel{\eqref{eq:def_a_kc}}= & \textcolor{\col}{h_\Strand h_{\TheStrand-1}^{-1}..}..h_1^{-1}\twC{\nu}h_1...h_{\TheStrand-1} & 
\\
\mystackrel{\ref{cor:pres_orb_braid_free_prod_rel2}}= & h_{\TheStrand-1}^{-1}...h_{\Strand+2}^{-1}\textcolor{\col}{h_\Strand h_{\Strand+1}^{-1}h_\Strand^{-1}}... h_1^{-1}\twC{\nu}h_1...h_{\TheStrand-1} & 
\\
\mystackrel{\ref{cor:pres_orb_braid_free_prod_rel2}}= & h_{\TheStrand-1}^{-1}...h_{\Strand+2}^{-1}h_{\Strand+1}^{-1}h_\Strand^{-1}\textcolor{\col}{h_{\Strand+1}h_{\Strand-1}^{-1}...h_1^{-1}\twC{\nu}h_1..}..h_{\TheStrand-1} & 
\\
\mystackrel{\ref{cor:pres_orb_braid_free_prod_rel2},\ref{cor:pres_orb_braid_free_prod_rel3}}= & h_{\TheStrand-1}^{-1}...h_1^{-1}\twC{\nu}h_1...h_{\Strand-1}\textcolor{\col}{h_{\Strand+1}h_\Strand h_{\Strand+1}}...h_{\TheStrand-1} & 
\\
\mystackrel{\ref{cor:pres_orb_braid_free_prod_rel2}}= & h_{\TheStrand-1}^{-1}...h_1^{-1}\twC{\nu}h_1...h_\Strand h_{\Strand+1}\textcolor{\col}{h_\Strand h_{\Strand+2}...h_{\TheStrand-1}} & 
\\
\mystackrel{\ref{cor:pres_orb_braid_free_prod_rel2}}= & h_{\TheStrand-1}^{-1}...h_1^{-1}\twC{\nu}h_1...h_{\TheStrand-1}h_\Strand & \mystackrel{\eqref{eq:def_a_kc}}= \twistC{\TheStrand}{\nu}h_\Strand. 
\end{align*}

Proving that the relation $\langle h_1,\twC{1}h_1\twC{1}^{-1}\rangle_\TheOrder=\langle\twC{1}h_1\twC{1}^{-1},h_1\rangle_\TheOrder$ from \ref{lem:psi_rel_R1} is satisfied, relies on $[h_1\twC{1}h_1,\twC{1}]=~1$ from relation $\ref{cor:pres_orb_braid_free_prod_rel4}$ and $\twC{1}^\TheOrder=1$ from $\ref{cor:pres_orb_braid_free_prod_rel1}$. 
For $\TheOrder=2l$, we have 
\begin{align}
\label{prop:sec_pres_free_prod_2fact_eq_even}
& \langle\twC{1} h_1\twC{1}^{-1},h_1\rangle_\TheOrder & \mystackrel{}= & (\twC{1}h_1\twC{1}^{-1}h_1)^l & \numbereq
\\
\mystackrel{}= & \twC{1}h_1\textcolor{\col}{(\twC{1}^{-1}h_1\twC{1}h_1)^{l-1}}\twC{1}^{-1}h_1 & \mystackrel{\ref{cor:pres_orb_braid_free_prod_rel4}}= & \twC{1}h_1(h_1\twC{1}h_1)^{l-1}\textcolor{\col}{\twC{1}^{-l}}h_1 & \nonumber
\\
\mystackrel{\ref{cor:pres_orb_braid_free_prod_rel1}}= & \twC{1}h_1\textcolor{\col}{(h_1\twC{1}h_1)^{l-1}\twC{1}^l}h_1 & \mystackrel{\ref{cor:pres_orb_braid_free_prod_rel4}}= & \twC{1}h_1(\twC{1}h_1\twC{1}h_1)^{l-1}\twC{1}h_1 & \nonumber
\\
\mystackrel{}= & (\textcolor{\col}{\twC{1}h_1\twC{1}h_1})^l & \mystackrel{\ref{cor:pres_orb_braid_free_prod_rel4}}= & (\textcolor{\col}{h_1\twC{1}h_1\twC{1}})^l & \nonumber
\\
\mystackrel{\ref{cor:pres_orb_braid_free_prod_rel4}}= & (h_1\twC{1}h_1)^l\textcolor{\col}{\twC{1}^l} & \mystackrel{\ref{cor:pres_orb_braid_free_prod_rel1}}= & (\textcolor{\col}{h_1\twC{1}h_1})^l \textcolor{\col}{\twC{1}}^{-l} & \nonumber
\\
\mystackrel{\ref{cor:pres_orb_braid_free_prod_rel4}}= & (h_1\twC{1}h_1\twC{1}^{-1})^l & \mystackrel{}= &  \langle h_1,\twC{1}h_1\twC{1}^{-1}\rangle_\TheOrder. \nonumber
\end{align}
For $\TheOrder=2l+1$, the same relations imply 
\begin{align}
\label{prop:sec_pres_free_prod_2fact_eq_odd}
& \langle\twC{1} h_1\twC{1}^{-1},h_1\rangle_\TheOrder & \mystackrel{}= & (\twC{1}h_1\twC{1}^{-1}h_1)^l\twC{1}h_1\twC{1}^{-1} & \numbereq
\\
\mystackrel{}= & \twC{1}h_1\textcolor{\col}{(\twC{1}^{-1}h_1\twC{1}h_1)^l}\twC{1}^{-1} & \mystackrel{\ref{cor:pres_orb_braid_free_prod_rel4}}= & \twC{1}h_1(h_1\twC{1}h_1)^l\textcolor{\col}{\twC{1}^{-(l+1)}} & \nonumber
\\
\mystackrel{\ref{cor:pres_orb_braid_free_prod_rel1}}= & \twC{1}h_1\textcolor{\col}{(h_1\twC{1}h_1)^l\twC{1}^l} & \mystackrel{\ref{cor:pres_orb_braid_free_prod_rel4}}= & \twC{1}h_1(\twC{1}h_1\twC{1}h_1)^l & \nonumber
\\
\mystackrel{}= & (\textcolor{\col}{\twC{1}h_1\twC{1}h_1})^l\twC{1}h_1 & \mystackrel{\ref{cor:pres_orb_braid_free_prod_rel4}}= & \textcolor{\col}{(h_1\twC{1}h_1\twC{1})^l}\twC{1}h_1 & \nonumber
\\
\mystackrel{\ref{cor:pres_orb_braid_free_prod_rel4}}= & (h_1\twC{1}h_1)^l\textcolor{\col}{\twC{1}^{l+1}}h_1 & \mystackrel{\ref{cor:pres_orb_braid_free_prod_rel1}}= & \textcolor{\col}{(h_1\twC{1}h_1)^l \twC{1}^{-l}}h_1 & \nonumber
\\
\mystackrel{\ref{cor:pres_orb_braid_free_prod_rel4}}= & (h_1\twC{1}h_1\twC{1}^{-1})^lh_1 & \mystackrel{}= &  \langle h_1,\twC{1}h_1\twC{1}^{-1}\rangle_\TheOrder. \nonumber
\end{align}
For the analogous relations from \ref{lem:psi_rel_R2}, we recall the relations \ref{lem:psi_rel_S1} and \ref{lem:psi_rel_C2}:  
\[
\twistC{\TheStrand}{1}^\TheOrder=1=\twistC{\TheStrand}{2}^{\TheOrder'} \; \text{ and } \; [h_{\TheStrand-1}\twistC{\TheStrand}{\nu}h_{\TheStrand-1},\twistC{\TheStrand}{\nu}]=1. 
\]
Hence, $\langle h_{\TheStrand-1},\twistC{\TheStrand}{1}h_{\TheStrand-1}\twistC{\TheStrand}{1}^{-1}\rangle_\TheOrder=\langle\twistC{\TheStrand}{1}h_{\TheStrand-1}\twistC{\TheStrand}{1}^{-1},h_{\TheStrand-1}\rangle_\TheOrder$ and $\langle h_{\TheStrand-1},\twistC{\TheStrand}{2}h_{\TheStrand-1}\twistC{\TheStrand}{2}^{-1}\rangle_{\TheOrder'}=\langle\twistC{\TheStrand}{2}h_{\TheStrand-1}\twistC{\TheStrand}{2}^{-1},h_{\TheStrand-1}\rangle_{\TheOrder'}$ from \ref{lem:psi_rel_R2} follow as in \eqref{prop:sec_pres_free_prod_2fact_eq_even} and \eqref{prop:sec_pres_free_prod_2fact_eq_odd}. 

For relation \ref{lem:psi_rel_R3}, we observe: 
\\
\begin{minipage}{1.25\textwidth}
\begin{align}
\label{prop:sec_pres_free_prod_1fact_eq_add_rel}
& \twC{1}h_1\textcolor{\col}{\twC{1}^{-1}h_2}h_1\twC{1} h_1\twC{1}^{-1}h_2 & \mystackrel{\ref{cor:pres_orb_braid_free_prod_rel3}}= & \twC{1}h_1h_2\textcolor{\col}{\twC{1}^{-1}h_1\twC{1}h_1}\twC{1}^{-1}h_2 \numbereq
\\
\mystackrel{\ref{cor:pres_orb_braid_free_prod_rel4}}= & \twC{1}h_1h_2h_1\twC{1}h_1\textcolor{\col}{\twC{1}^{-2}h_2} & \mystackrel{\ref{cor:pres_orb_braid_free_prod_rel3}}= & \twC{1}\textcolor{\col}{h_1h_2h_1}\twC{1} h_1h_2\twC{1}^{-2} \nonumber
\\
\mystackrel{\ref{cor:pres_orb_braid_free_prod_rel2}}= & \twC{1}h_2h_1\textcolor{\col}{h_2\twC{1}}h_1h_2\twC{1}^{-2} & \mystackrel{\ref{cor:pres_orb_braid_free_prod_rel3}}= & \twC{1}h_2h_1\twC{1}\textcolor{\col}{h_2h_1h_2}\twC{1}^{-2} \nonumber
\\
\mystackrel{\ref{cor:pres_orb_braid_free_prod_rel2}}= & \textcolor{\col}{\twC{1}h_2}h_1\twC{1}h_1h_2h_1\twC{1}^{-2} & \mystackrel{\ref{cor:pres_orb_braid_free_prod_rel3}}= & h_2\textcolor{\col}{\twC{1}h_1\twC{1}h_1}h_2h_1\twC{1}^{-2} \nonumber
\\
\mystackrel{\ref{cor:pres_orb_braid_free_prod_rel4}}= & h_2h_1\twC{1}h_1\textcolor{\col}{\twC{1}h_2}h_1\twC{1}^{-2} & \mystackrel{\ref{cor:pres_orb_braid_free_prod_rel3}}= & h_2h_1\twC{1}h_1\twC{1}^{-1}h_2\textcolor{\col}{\twC{1}^2h_1\twC{1}^{-2}} \nonumber
\\
\mystackrel{\ref{cor:pres_orb_braid_free_prod_rel4}}= & h_2h_1\twC{1}h_1\textcolor{\col}{\twC{1}^{-1}h_2\twC{1}}h_1^{-1}\twC{1}^{-1}h_1\twC{1}h_1\twC{1}^{-1} & \mystackrel{\ref{cor:pres_orb_braid_free_prod_rel3}}= & h_2h_1\twC{1}\textcolor{\col}{h_1h_2h_1^{-1}}\twC{1}^{-1}h_1\twC{1}h_1\twC{1}^{-1} \nonumber
\\
\mystackrel{\ref{cor:pres_orb_braid_free_prod_rel2}}= & h_2h_1\textcolor{\col}{\twC{1}h_2^{-1}}h_1\textcolor{\col}{h_2\twC{1}^{-1}}h_1\twC{1}h_1\twC{1}^{-1} & \mystackrel{\ref{cor:pres_orb_braid_free_prod_rel3}}= & \textcolor{\col}{h_2h_1h_2^{-1}}\twC{1}h_1\twC{1}^{-1}h_2h_1\twC{1} h_1\twC{1}^{-1} \nonumber
\\
\mystackrel{\ref{cor:pres_orb_braid_free_prod_rel2}}= & h_1^{-1}h_2h_1\twC{1}h_1\twC{1}^{-1}h_2h_1\twC{1} h_1\twC{1}^{-1}. \nonumber
\end{align} 
\end{minipage}
By left multiplication with $h_1$, we obtain \ref{lem:psi_rel_R3}. Using observations \ref{lem:psi_rel_C1} and \ref{lem:psi_rel_C2} instead of \ref{cor:pres_orb_braid_free_prod_rel3} and \ref{cor:pres_orb_braid_free_prod_rel4}, a verbatim calculation shows relation~\ref{lem:psi_rel_R4}. 
\end{proof}

\subsection*{Orbifold braid groups without punctures}

Lemma \ref{lem:psi_rel} allows us to identify another presentation of $\Z_\TheStrand(\Sigma_{\freeprodtwo})$ in terms of the generators $\htwC,h_1,...,h_{\TheStrand-1},\htw,\twsC,\twsC'$. By the same arguments, a similar presentation can be obtained for $\Z_\TheStrand(D_{\cycm})$ if we omit the generators $\htw$ and $\twsC'$. 

\begin{proposition}
\label{prop:sec_pres_free_prod_2fact}
For $\TheStrand\geq3$, the presentation of $\Z_\TheStrand(\Sigma_{\freeprodtwo})$ given in Theorem~\textup{\ref{thm:pres_Z_n}} is isomorphic to the group generated by $\htwC,h_1,...,h_{\TheStrand-1},\htw,\twsC,\twsC'$ with defining relations 
\begin{enumerate}[label={\textup{(R\arabic*)}},ref={\thetheorem(R\arabic*)}]
\item 
\label{prop:sec_pres_free_prod_2fact_relR1} 
braid and commutator relations for the generators $\htwC,h_1,...,h_{\TheStrand-1},\htw$, 
\item 
\label{prop:sec_pres_free_prod_2fact_relR2}
$\langle h_1,\htwC\rangle_\TheOrder=\langle\htwC,h_1\rangle_\TheOrder$ \; and \; $\langle h_{\TheStrand-1},\htw\rangle_{\TheOrder'}=\langle \htw,h_{\TheStrand-1}\rangle_{\TheOrder'}$, 
\item 
\label{prop:sec_pres_free_prod_2fact_relR3}
$(h_1\htwC h_2)^2=(h_2h_1\htwC)^2$ \; and \; $(h_{\TheStrand-1}\htw h_{\TheStrand-2})^2=(h_{\TheStrand-2}h_{\TheStrand-1}\htw)^2$, 
\end{enumerate}
\begin{enumerate}[label={\textup{(S\arabic*)}},ref={\thetheorem(S\arabic*)}]
\item 
\label{prop:sec_pres_free_prod_2fact_relS1} 
$\twsC^\TheOrder=(\twsC')^{\TheOrder'}=1$ \; and \; $[\twsC,\twsC']=1$, 
\end{enumerate}
\begin{enumerate}[label={\textup{(C\arabic*)}},ref={\thetheorem(C\arabic*)}]
\item 
\label{prop:sec_pres_free_prod_2fact_relC1}
$\twsC h_\NStrand\twsC^{-1}=h_\NStrand$ \; for \; $2\leq\NStrand<\TheStrand\;$ and $\;\twsC'h_\Strand \twsC'^{-1}=h_\Strand$ \; for \; $1\leq\Strand\leq\TheStrand-2$, 
\item 
\label{prop:sec_pres_free_prod_2fact_relC2}
$\twsC\htw\twsC^{-1}=\htw\;$ and $\;\twsC'\htwC\twsC'^{-1}=\htwC$, 
\item 
\label{prop:sec_pres_free_prod_2fact_relC3}
$\twsC h_1\twsC^{-1}=\htwC\;$ and $\;\twsC'h_{\TheStrand-1}\twsC'^{-1}=\htw$ and 
\item 
\label{prop:sec_pres_free_prod_2fact_relC4}
$\twsC \htwC\twsC^{-1}=\left(\htwC\right)^{-1}h_1\htwC\;$ and $\;\twsC'\htw\twsC'^{-1}=\left(\htw\right)^{-1}h_{\TheStrand-1}\htw$.  
\end{enumerate}
\end{proposition}
\begin{proof}
Let $A$ be the group given by the presentation from Theorem \ref{thm:pres_Z_n} in the case $\TheCone=2$ and $\ThePct=0$. Moreover, let $B$ be the group given by the presentation above. With respect to the definition of $\twsC,\twsC',\htwC$ and $\htw$, we define the assignments: 
\begin{align*}
\label{prop:sec_pres_free_prod_2fact_eq_def_varphi_psi}
\varphi:A & \rightarrow B, & & \text{and} & \psi: B & \rightarrow A, \numbereq
\\
h_\Strand & \mapsto h_\Strand \;\text{ for }\; 1\leq\Strand<\TheStrand, & & & h_\Strand & \mapsto h_\Strand \;\text{ for }\; 1\leq\Strand<\TheStrand,
\\
\twC{1} & \mapsto \twsC, & & & \twsC & \mapsto \twC{1},
\\
\twC{2} & \mapsto h_1...h_{\TheStrand-1}\twsC'h_{\TheStrand-1}^{-1}...h_1^{-1} & & & \twsC' & \mapsto h_{\TheStrand-1}^{-1}...h_1^{-1}\twC{2}h_1...h_{\TheStrand-1}, 
\\
& & & & \htwC & \mapsto \twC{1}h_1\twC{1}^{-1}, 
\\
& & & & \htw & \mapsto \psi(\twsC')h_{\TheStrand-1}\psi(\twsC')^{-1}. 
\end{align*}
By \eqref{eq:def_a_kc}, the element $\psi(\twsC')$ coincides with $\twistC{\TheStrand}{2}$.  

The definitions of $\varphi$ and $\psi$ directly imply that the assignments $\psi\circ\varphi$ describe the identical map on the letters $h_\Strand,\twC{1}$ and $\twC{2}$. Moreover, the assignments $\varphi\circ\psi$ yield the identical map on the generators of $B$ (modulo relations in $B$). 

To deduce that these assignments induce homomorphisms between the groups given by the presentations $A$ and $B$, we will apply Theorem \ref{thm:von_Dyck}. Therefore, it remains to check that $\varphi$ and $\psi$ preserve the relations of $A$ and $B$, respectively. For the map $\psi$, this follows from Lemma~\ref{lem:psi_rel}. We summarize this observation in Table~\ref{tab:rel_psi}.  

\renewcommand{\arraystretch}{1.75}
{\setlength\LTleft{-0.125\textwidth}
\begin{tabularx}{1.25\linewidth}{rl|l}
\caption{}
\label{tab:rel_psi} \\
& Relations from $B$ & Images under $\psi$ \\
\cline{2-3}
\noalign{\vskip\doublerulesep\vskip-\arrayrulewidth}
\cline{2-3}
\endfirsthead
& Relations from $B$ & Images under $\psi$ \\
\cline{2-3}
\noalign{\vskip\doublerulesep\vskip-\arrayrulewidth}
\cline{2-3}
\endhead
\multicolumn{3}{r}{\footnotesize(To be continued)}
\endfoot
\cline{2-3}
\noalign{\vskip\doublerulesep\vskip-\arrayrulewidth}
\cline{2-3}
\endlastfoot
\multirow{6}{3.5em}{\ref{prop:sec_pres_free_prod_2fact_relR1}$\left\{\rule{0cm}{3.3cm}\right.$}& \doubletable{braid and commutator relations \\ for $h_1,...,h_{\TheStrand-1}$} & \doubletable{braid and commutator relations \\ for $h_1,...,h_{\TheStrand-1}$ are covered by \ref{cor:pres_orb_braid_free_prod_rel2}} \\ \cline{2-3}
& $[\htwC,h_\NStrand]=1$ \; for \; $\NStrand\geq3$ & $[\twC{1} h_1\twC{1}^{-1},h_\NStrand]=1$ by \ref{cor:pres_orb_braid_free_prod_rel2} and \ref{cor:pres_orb_braid_free_prod_rel3} \\ \cline{2-3}
& $[\htwC,\htw]=1$ & \doubletable{$[\twC{1}h_1\twC{1}^{-1},\twistC{\TheStrand}{2}h_{\TheStrand-1}\twistC{\TheStrand}{2}^{-1}]=1$ \\by \ref{cor:pres_orb_braid_free_prod_rel2}, \ref{cor:pres_orb_braid_free_prod_rel3}, \ref{lem:psi_rel_S2} and \ref{lem:psi_rel_C1}} \\ \cline{2-3}
& $\htwC h_2\htwC=h_2\htwC h_2$ & $\twC{1}h_1\twC{1}^{-1}h_2\twC{1}h_1\twC{1}^{-1}=h_2\twC{1}h_1\twC{1}^{-1}h_2$ by \ref{cor:pres_orb_braid_free_prod_rel2} and \ref{cor:pres_orb_braid_free_prod_rel3} \\ \cline{2-3} 
& $[\htw,h_\Strand]=1$ \; for \; $\Strand\leq\TheStrand-3$ & $[\twistC{\TheStrand}{2}h_{\TheStrand-1}\twistC{\TheStrand}{2}^{-1},h_\Strand]=1$ by \ref{cor:pres_orb_braid_free_prod_rel2} and \ref{lem:psi_rel_C1} \\ \cline{2-3}
& $\htw h_{\TheStrand-2}\htw=h_{\TheStrand-2}\htw h_{\TheStrand-2}$ & \doubletable{$\twistC{\TheStrand}{2}h_{\TheStrand-1}\twistC{\TheStrand}{2}^{-1}h_{\TheStrand-2}\twistC{\TheStrand}{2}h_{\TheStrand-1}\twistC{\TheStrand}{2}^{-1}$\\$=h_{\TheStrand-2}\twistC{\TheStrand}{2}h_{\TheStrand-1}\twistC{\TheStrand}{2}^{-1}h_{\TheStrand-2}$ by \ref{cor:pres_orb_braid_free_prod_rel2} and \ref{lem:psi_rel_C1}} \\ \cline{2-3} 
\multirow{2}{3.5em}{\ref{prop:sec_pres_free_prod_2fact_relR2}$\left\{\rule{0cm}{0.75cm}\right.$}& $\langle h_1,\htwC\rangle_\TheOrder=\langle\htwC,h_1\rangle_\TheOrder$ & $\langle h_1,\twC{1} h_1\twC{1}^{-1}\rangle_\TheOrder\stackrel{\text{\ref{lem:psi_rel_R1}}}=\langle\twC{1} h_1\twC{1}^{-1},h_1\rangle_\TheOrder$ \\ \cline{2-3}
& $\langle h_{\TheStrand-1},\htw\rangle_{\TheOrder'}=\langle\htw,h_{\TheStrand-1}\rangle_{\TheOrder'}$ & 
$\langle h_{\TheStrand-1},\twistC{\TheStrand}{2}h_{\TheStrand-1}\twistC{\TheStrand}{2}^{-1}\rangle_{\TheOrder'}\stackrel{\text{\ref{lem:psi_rel_R2}}}=\langle \twistC{\TheStrand}{2}h_{\TheStrand-1}\twistC{\TheStrand}{2}^{-1},h_{\TheStrand-1}\rangle_{\TheOrder'}$ \\ \cline{2-3} 
\multirow{2}{3.5em}{\ref{prop:sec_pres_free_prod_2fact_relR3}$\left\{\rule{0cm}{1.15cm}\right.$}& $(h_1\htwC h_2)^2=(h_2h_1\htwC)^2$ & $(h_1\twC{1} h_1\twC{1}^{-1}h_2)^2\stackrel{\text{\ref{lem:psi_rel_R3}}}=(h_2h_1\twC{1}h_1\twC{1}^{-1})^2$ \\ \cline{2-3}
& \doubletable{$(h_{\TheStrand-1}\htw h_{\TheStrand-2})^2$
\\
=\ $(h_{\TheStrand-2}h_{\TheStrand-1}\htw)^2$} & 
\doubletable{$(h_{\TheStrand-1}\twistC{\TheStrand}{2}h_{\TheStrand-1}\twistC{\TheStrand}{2}^{-1}h_{\TheStrand-2})^2$
\\
$\stackrel{\text{\ref{lem:psi_rel_R4}}}=(h_{\TheStrand-2}h_{\TheStrand-1}\twistC{\TheStrand}{2}h_{\TheStrand-1}\twistC{\TheStrand}{2}^{-1})^2$} \\ \cline{2-3}
\multirow{3}{3.5em}{\hspace*{-2pt} \ref{prop:sec_pres_free_prod_2fact_relS1}$\left\{\rule{0cm}{1.2cm}\right.$}& $\twsC^\TheOrder=1$ & $\twC{1}^\TheOrder\stackrel{\ref{cor:pres_orb_braid_free_prod_rel1}}=1$ \; with \; $\TheOrder=\TheOrder_1$ \\ \cline{2-3}
& $(\twsC')^{\TheOrder'}=1$ & $\twistC{\TheStrand}{2}^{\TheOrder'}\stackrel{\text{\ref{lem:psi_rel_S1}}}=1$ \; with \; $\TheOrder'=\TheOrder_2$ \\ \cline{2-3}
& $[\twsC,\twsC']=1$ & $[\twC{1},\twistC{\TheStrand}{2}]\stackrel{\text{\ref{lem:psi_rel_S2}}}=1$ \\ \cline{2-3}
\multirow{2}{3.5em}{\ref{prop:sec_pres_free_prod_2fact_relC1}$\left\{\rule{0cm}{0.75cm}\right.$}& $\twsC h_\NStrand\twsC^{-1}=h_\NStrand$ \; for \; $\NStrand\geq2$ & $\twC{1}h_\NStrand\twC{1}^{-1}=h_\NStrand$ by \ref{cor:pres_orb_braid_free_prod_rel3} \\ \cline{2-3}
& $\twsC'h_\Strand\twsC'^{-1}=h_\Strand$ \; for \; $\Strand\leq\TheStrand-2$ & $\twistC{\TheStrand}{2}h_\Strand\twistC{\TheStrand}{2}^{-1}=h_\Strand$ by \ref{lem:psi_rel_C1} \\ \cline{2-3} 
\multirow{2}{3.5em}{\ref{prop:sec_pres_free_prod_2fact_relC2}$\left\{\rule{0cm}{0.75cm}\right.$}& $\twsC\htw\twsC^{-1}=\htw$ & $\twC{1}\twistC{\TheStrand}{2}h_{\TheStrand-1}\twistC{\TheStrand}{2}^{-1}\twC{1}^{-1}=\twistC{\TheStrand}{2}h_{\TheStrand-1}\twistC{\TheStrand}{2}^{-1}$ by \ref{lem:psi_rel_S2} and \ref{cor:pres_orb_braid_free_prod_rel3} \\ \cline{2-3}
& $\twsC'\htwC\twsC'^{-1}=\htwC$ & $\twistC{\TheStrand}{2}\twC{1}h_1\twC{1}^{-1}\twistC{\TheStrand}{2}^{-1}=\twC{1}h_1\twC{1}^{-1}$ by \ref{lem:psi_rel_S2} and \ref{lem:psi_rel_C1} \\ \cline{2-3} 
\multirow{2}{3.5em}{\ref{prop:sec_pres_free_prod_2fact_relC3}$\left\{\rule{0cm}{0.7cm}\right.$}& $\twsC h_1\twsC^{-1}=\htwC$ & $\twC{1}h_1\twC{1}^{-1}=\twC{1}h_1\twC{1}^{-1}$ \\ \cline{2-3}
& $\twsC'h_{\TheStrand-1}\twsC'^{-1}=\htw$ & $\twistC{\TheStrand}{2}h_{\TheStrand-1}\twistC{\TheStrand}{2}^{-1}=\twistC{\TheStrand}{2}h_{\TheStrand-1}\twistC{\TheStrand}{2}^{-1}$ \\ \cline{2-3} 
\multirow{2}{3.5em}{\ref{prop:sec_pres_free_prod_2fact_relC4}$\left\{\rule{0cm}{0.75cm}\right.$}& $\twsC\htwC\twsC^{-1}=(\htwC)^{-1}h_1\htwC$ & $\twC{1}^2h_1\twC{1}^{-2}=\twC{1}h_1^{-1}\twC{1}^{-1}h_1\twC{1}h_1\twC{1}^{-1}$ by \ref{cor:pres_orb_braid_free_prod_rel4}\\ \cline{2-3} 
& $\twsC'\htw\twsC'^{-1}=(\htw)^{-1}h_{\TheStrand-1}\htw$ & $\twistC{\TheStrand}{2}^2h_{\TheStrand-1}\twistC{\TheStrand}{2}^{-2}=\twistC{\TheStrand}{2}h_{\TheStrand-1}^{-1}\twistC{\TheStrand}{2}^{-1}h_{\TheStrand-1}\twistC{\TheStrand}{2}h_{\TheStrand-1}\twistC{\TheStrand}{2}^{-1}$ by \ref{lem:psi_rel_C2} 
\end{tabularx}
}

Table \ref{tab:rel_varphi} summarizes which relations are necessary to deduce that $\varphi$ preserves the relations from $A$. For the most of the relations, this follows directly from $B$; for the less obvious relations, we refer to the explanations below Table \ref{tab:rel_varphi}. 

\renewcommand{\arraystretch}{1.75}
{\setlength\LTleft{-0.125\textwidth}
\begin{tabularx}{1.25\linewidth}{rl|l}
\caption{}
\label{tab:rel_varphi} \\ 
& Relations from $A$ & Images under $\varphi$ \\
\cline{2-3}
\noalign{\vskip\doublerulesep\vskip-\arrayrulewidth}
\cline{2-3}
\endfirsthead
& Relations from $A$ & Images under $\varphi$ \\
\cline{2-3}
\noalign{\vskip\doublerulesep\vskip-\arrayrulewidth}
\cline{2-3}
\endhead
\multicolumn{3}{r}{\footnotesize(To be continued)}
\endfoot
\cline{2-3}
\noalign{\vskip\doublerulesep\vskip-\arrayrulewidth}
\cline{2-3}
\endlastfoot
\multirow{2}{3em}{\ref{cor:pres_orb_braid_free_prod_rel1}$\left\{\rule{0cm}{0.75cm}\right.$}& $\twC{1}^\TheOrder=1$ & $\twsC^\TheOrder\stackrel{\textup{\ref{prop:sec_pres_free_prod_2fact_relS1}}}=1$ \\ \cline{2-3}
& $\twC{2}^{\TheOrder'}=1$ & $(h_1...h_{\TheStrand-1}\twsC'h_{\TheStrand-1}^{-1}...h_1^{-1})^{\TheOrder'}=1$ by \ref{prop:sec_pres_free_prod_2fact_relS1} \\ \cline{2-3}
\multirow{1}{3em}{\ref{cor:pres_orb_braid_free_prod_rel2}$\left\{\rule{0cm}{0.65cm}\right.$}& \doubletable{braid and commutator relations \\ for $h_1,...,h_{\TheStrand-1}$} & \doubletable{braid and commutator relations \\ for $h_1,...,h_{\TheStrand-1}$ are covered by \ref{prop:sec_pres_free_prod_2fact_relR1}} \\ \cline{2-3}
\multirow{2}{3em}{\ref{cor:pres_orb_braid_free_prod_rel3}$\left\{\rule{0cm}{0.7cm}\right.$} & $[\twC{1},h_\Strand]=1$ \; for \; $\Strand\geq2$ & $[\twsC_,h_\Strand]\stackrel{\textup{\ref{prop:sec_pres_free_prod_2fact_relC1}}}=1$ \\ \cline{2-3}
& $[\twC{2},h_\Strand]=1$ \; for \; $\Strand\geq2$ & $[\varphi(\twC{2}),h_\Strand]\stackrel{\ref{prop:sec_pres_free_prod_2fact_par_tau_hj}}=1$ \\ \cline{2-3}
\multirow{2}{3em}{\ref{cor:pres_orb_braid_free_prod_rel4}$\left\{\rule{0cm}{0.75cm}\right.$} & $[h_1\twC{1}h_1,\twC{1}]=1$ & $[h_1\twsC h_1,\twsC]\stackrel{\ref{prop:sec_pres_free_prod_2fact_par_th1th1=h1th1t}}=1$ \\ \cline{2-3}
& $[h_1\twC{2}h_1,\twC{2}]=1$ & $[h_1\varphi(\twC{2})h_1,\varphi(\twC{2})]\stackrel{\ref{prop:sec_pres_free_prod_2fact_par_h1_tau_h1_tau}}=1$ \\ \cline{2-3}
\multirow{1}{3em}{\ref{cor:pres_orb_braid_free_prod_rel5}$\left\{\rule{0cm}{0.433cm}\right.$}& $[\twC{1},h_1^{-1}\twC{2}h_1]=1$ & $[\twsC,h_2...h_{\TheStrand-1}\twsC'h_{\TheStrand-1}^{-1}...h_2^{-1}]=1$ by \ref{prop:sec_pres_free_prod_2fact_relS1} and \ref{prop:sec_pres_free_prod_2fact_relC1} 
\end{tabularx}
}

\begin{enumerate}[leftmargin=1.1cm,label={\textbf{\thetheorem.\arabic*.}},ref={\thetheorem.\arabic*}]
\item 
\label{prop:sec_pres_free_prod_2fact_par_tau_hj}
For $2\leq\Strand<\TheStrand$, we observe 
\vspace*{-4pt}
\end{enumerate}
\begin{align*}
h_\Strand\varphi(\twC{2})\mystackrel{\eqref{prop:sec_pres_free_prod_2fact_eq_def_varphi_psi}}= & \textcolor{\col}{h_\Strand h_1..}..h_{\TheStrand-1}\twsC'h_{\TheStrand-1}^{-1}...h_1^{-1} & 
\\
\mystackrel{\ref{prop:sec_pres_free_prod_2fact_relR1}}= & h_1...h_{\Strand-2}\textcolor{\col}{h_\Strand h_{\Strand-1}h_\Strand}h_{\Strand+1}...h_{\TheStrand-1}\twsC'h_{\TheStrand-1}^{-1}...h_1^{-1} & 
\\
\mystackrel{\ref{prop:sec_pres_free_prod_2fact_relR1}}= & h_1...h_{\Strand-2}h_{\Strand-1}h_\Strand \textcolor{\col}{h_{\Strand-1}h_{\Strand+1}...h_{\TheStrand-1}\twsC'h_{\TheStrand-1}^{-1}..}..h_1^{-1} & 
\\
\mystackrel{\ref{prop:sec_pres_free_prod_2fact_relR1},\ref{prop:sec_pres_free_prod_2fact_relC1}}= & h_1...h_{\TheStrand-1}\twsC'h_{\TheStrand-1}^{-1}...h_{\Strand+1}^{-1}\textcolor{\col}{h_{\Strand-1}h_\Strand^{-1}h_{\Strand-1}^{-1}}...h_1^{-1} & 
\\
\mystackrel{\ref{prop:sec_pres_free_prod_2fact_relR1}}= & h_1...
h_{\TheStrand-1}\twsC'h_{\TheStrand-1}^{-1}...h_{\Strand+1}^{-1}h_\Strand^{-1}h_{\Strand-1}^{-1}\textcolor{\col}{h_\Strand h_{\Strand-2}^{-1}...h_1^{-1}} & 
\\
\mystackrel{\ref{prop:sec_pres_free_prod_2fact_relR1}}= & h_1...h_{\TheStrand-1}\twsC'h_{\TheStrand-1}^{-1}...h_1^{-1}h_\Strand & \stackrel{\eqref{prop:sec_pres_free_prod_2fact_eq_def_varphi_psi}}= \varphi(\twC{2})h_\Strand, 
\end{align*}
i.e.\ the assignments $\varphi$ preserve $[\twC{2},h_\Strand]=1$ for $ \Strand\geq2$. 

\begin{enumerate}[resume,leftmargin=1.1cm,label={\textbf{\thetheorem.\arabic*.}},ref={\thetheorem.\arabic*}]
\item
\label{prop:sec_pres_free_prod_2fact_par_th1th1=h1th1t}
The following observation proves that $\varphi$ preserves the relation $[h_1\twC{1}h_1,\twC{1}]=~1$\vspace*{-4pt}
\end{enumerate}
\noindent from \ref{cor:pres_orb_braid_free_prod_rel4}: 
\begin{align*}
& \twsC\textcolor{\col}{h_1^\twsC}\twsC^{-1}\stackrel{\textup{\ref{prop:sec_pres_free_prod_2fact_relC4}}}=\textcolor{\col}{(h_1^\twsC)^{-1}}h_1\textcolor{\col}{h_1^\twsC}
\\
\mystackrel{\ref{prop:sec_pres_free_prod_2fact_relC3}}\Leftrightarrow &  \twsC^2 h_1\twsC^{-2}=\twsC h_1^{-1}\twsC^{-1}h_1\twsC h_1\twsC^{-1}. 
\end{align*}
By multiplication with $\twsC h_1\twsC^{-1}$ from the left and $\twsC^2$ from the right, we obtain $\twsC h_1\twsC h_1=h_1\twsC h_1\twsC$. 

\newpage 

\begin{enumerate}[resume,leftmargin=1.1cm,label={\textbf{\thetheorem.\arabic*.}},ref={\thetheorem.\arabic*}]
\item 
\label{prop:sec_pres_free_prod_2fact_par_h1_tau_h1_tau}
Similarly, we want to deduce that $\varphi$ preserves the relation $[h_1\twC{2}h_1,\twC{2}]=1$ 
\vspace*{-4pt}
\end{enumerate}
\noindent from \ref{cor:pres_orb_braid_free_prod_rel4}. Therefore, we observe that the relation $\twsC'\htw\twsC'^{-1}=(\htw)^{-1}h_{\TheStrand-1}\htw$ from \ref{prop:sec_pres_free_prod_2fact_relC4} as in \ref{prop:sec_pres_free_prod_2fact_par_th1th1=h1th1t} implies 
\begin{equation}
\label{prop:sec_pres_free_prod_2fact_eq_t'hnt'hn=hnt'hnt'}
[h_{\TheStrand-1}\twsC'h_{\TheStrand-1},\twsC']=1. 
\end{equation}
This allows us to prove that $\varphi$ preserves $[h_1\twC{2}h_1,\twC{2}]=1$. The proof requires two iterative arguments which we describe in the following: 
\begin{align*}
\label{prop:sec_pres_free_prod_2fact_eq1_h1_tau_h1_tau}
h_1\varphi(\twC{2})h_1\varphi(\twC{2})\mystackrel{(\ref{prop:sec_pres_free_prod_2fact_eq_def_varphi_psi})}= & \textcolor{\col}{h_1h_1}h_2...h_{\TheStrand-1}\twsC'h_{\TheStrand-1}^{-1}...\textcolor{\short}{h_1^{-1}h_1}h_1...h_{\TheStrand-1}\twsC'h_{\TheStrand-1}^{-1}...h_1^{-1} \numbereq 
\\
\mystackrel{}= & h_1^2h_2...h_{\TheStrand-1}\twsC'h_{\TheStrand-1}^{-1}...\textcolor{\col}{h_2^{-1}h_1h_2}...h_{\TheStrand-1}\twsC'h_{\TheStrand-1}^{-1}...h_1^{-1} 
\\
\mystackrel{\ref{prop:sec_pres_free_prod_2fact_relR1}}= & h_1^2h_2\textcolor{\col}{...h_{\TheStrand-1}\twsC'h_{\TheStrand-1}^{-1}...h_3^{-1}h_1}h_2\textcolor{\col}{h_1^{-1}h_3...h_{\TheStrand-1}\twsC'h_{\TheStrand-1}^{-1}..}..h_1^{-1} 
\\
\mystackrel{\ref{prop:sec_pres_free_prod_2fact_relR1},\ref{prop:sec_pres_free_prod_2fact_relC1}}= & h_1^{\textcolor{\col}{2}}\textcolor{\col}{h_2h_1}h_3...h_{\TheStrand-1}\twsC'h_{\TheStrand-1}^{-1}...h_3^{-1}h_2h_3...h_{\TheStrand-1}\twsC'h_{\TheStrand-1}^{-1}...h_3^{-1}\textcolor{\col}{h_1^{-1}h_2^{-1}h_1^{-1}} 
\\
\mystackrel{\ref{prop:sec_pres_free_prod_2fact_relR1}}= & h_1h_2h_1h_2h_3...h_{\TheStrand-1}\twsC'h_{\TheStrand-1}^{-1}...h_3^{-1}h_2h_3...h_{\TheStrand-1}\twsC'h_{\TheStrand-1}^{-1}...h_3^{-1}h_2^{-1}h_1^{-1}h_2^{-1}. 
\end{align*}
Comparing the second and the last line of (\ref{prop:sec_pres_free_prod_2fact_eq1_h1_tau_h1_tau}), we observe
\begin{align*}
\label{prop:sec_pres_free_prod_2fact_eq2_h1_tau_h1_tau}
& h_1h_1h_2...h_{\TheStrand-1}\twsC'h_{\TheStrand-1}^{-1}...h_2^{-1}h_1h_2...h_{\TheStrand-1}\twsC'h_{\TheStrand-1}^{-1}...h_1^{-1} \numbereq
\\
= \ & h_1h_2h_1h_2...h_{\TheStrand-1}\twsC'h_{\TheStrand-1}^{-1}...h_3^{-1}h_2h_3...h_{\TheStrand-1}\twsC'h_{\TheStrand-1}^{-1}...h_2^{-1}h_1^{-1}h_2^{-1}. 
\end{align*}
In particular, both lines have the form 
\[
h_1...h_\Strand h_1...h_{\TheStrand-1}\twsC'h_{\TheStrand-1}^{-1}...h_{\Strand+1}^{-1}h_\Strand h_{\Strand+1}...h_{\TheStrand-1}\twsC'h_{\TheStrand-1}^{-1}...h_1^{-1}h_\Strand^{-1}...h_2^{-1}. 
\]
Moreover, the relations $h_{\Strand+1}^{-1}h_\Strand h_{\Strand+1}=h_\Strand h_{\Strand+1}h_\Strand^{-1}$ and $[h_\Strand,h_\NStrand]=1$ for $\Strand+1<\NStrand<\TheStrand$ from \ref{prop:sec_pres_free_prod_2fact_relR1} and $[h_\Strand,\twsC']=1$ from \ref{prop:sec_pres_free_prod_2fact_relC1} used for $\Strand=1$ to deduce (\ref{prop:sec_pres_free_prod_2fact_eq2_h1_tau_h1_tau}) hold similarly for each $1\leq\Strand\leq\TheStrand-2$. Iteratively, this allows us to deduce that 
\begin{align*}
\label{prop:sec_pres_free_prod_2fact_eq3_h1_tau_h1_tau}
h_1\varphi(\twC{2})h_1\varphi(\twC{2})\mystackrel{(\ref{prop:sec_pres_free_prod_2fact_eq1_h1_tau_h1_tau})}= & h_1h_2h_1h_2...h_{\TheStrand-1}\twsC'h_{\TheStrand-1}^{-1}...h_3^{-1}h_2h_3...h_{\TheStrand-1}\twsC'h_{\TheStrand-1}^{-1}...h_2^{-1}h_1^{-1}h_2^{-1} \numbereq 
\\
\mystackrel{\textit{it.}}= & h_1...h_{\TheStrand-1}h_1...h_{\TheStrand-1}\twsC'h_{\TheStrand-1}\twsC'h_{\TheStrand-1}^{-1}...h_1^{-1}h_{\TheStrand-1}^{-1}...h_2^{-1}. 
\end{align*}
Further, we obtain 
\begin{align*}
\label{prop:sec_pres_free_prod_2fact_eq4_h1_tau_h1_tau}
h_1\varphi(\twC{2})h_1\varphi(\twC{2})\mystackrel{(\ref{prop:sec_pres_free_prod_2fact_eq3_h1_tau_h1_tau})}= & h_1...h_{\TheStrand-1}h_1...\textcolor{\col}{h_{\TheStrand-1}\twsC'h_{\TheStrand-1}\twsC'h_{\TheStrand-1}^{-1}}...h_1^{-1}h_{\TheStrand-1}^{-1}...h_2^{-1} \numbereq 
\\
\mystackrel{(\ref{prop:sec_pres_free_prod_2fact_eq_t'hnt'hn=hnt'hnt'})}= & h_1...h_{\TheStrand-1}\textcolor{\col}{h_1...h_{\TheStrand-2}\twsC'}h_{\TheStrand-1}\textcolor{\col}{\twsC'h_{\TheStrand-2}^{-1}...h_1^{-1}}h_{\TheStrand-1}^{-1}...h_2^{-1} 
\\
\mystackrel{\ref{prop:sec_pres_free_prod_2fact_relC1}}= & h_1...h_{\TheStrand-1}\twsC'h_1...\textcolor{\col}{h_{\TheStrand-2}h_{\TheStrand-1}h_{\TheStrand-2}^{-1}}...h_1^{-1}\twsC'h_{\TheStrand-1}^{-1}...h_2^{-1} 
\\
\mystackrel{\ref{prop:sec_pres_free_prod_2fact_relR1}}= & h_1...h_{\TheStrand-1}\twsC'\textcolor{\col}{h_1...h_{\TheStrand-3}h_{\TheStrand-1}^{-1}}h_{\TheStrand-2}\textcolor{\col}{h_{\TheStrand-1}h_{\TheStrand-3}^{-1}...h_1^{-1}}\twsC'h_{\TheStrand-1}^{-1}...h_2^{-1} 
\\
\mystackrel{\ref{prop:sec_pres_free_prod_2fact_relR1}}= & h_1...h_{\TheStrand-1}\twsC'h_{\TheStrand-1}^{-1}h_1...h_{\TheStrand-3}h_{\TheStrand-2}h_{\TheStrand-3}^{-1}...h_1^{-1}h_{\TheStrand-1}\twsC'h_{\TheStrand-1}^{-1}...h_2^{-1}.  
\end{align*}
Comparing the third and the last line of (\ref{prop:sec_pres_free_prod_2fact_eq4_h1_tau_h1_tau}), we observe
\begin{align*}
\label{prop:sec_pres_free_prod_2fact_eq5_h1_tau_h1_tau}
& h_1...h_{\TheStrand-1}\twsC'h_1...h_{\TheStrand-2}h_{\TheStrand-1}h_{\TheStrand-2}^{-1}...h_1^{-1}\twsC'h_{\TheStrand-1}^{-1}...h_2^{-1} \numbereq
\\
= \ & h_1...h_{\TheStrand-1}\twsC'h_{\TheStrand-1}^{-1}h_1...h_{\TheStrand-3}h_{\TheStrand-2}h_{\TheStrand-3}^{-1}...h_1^{-1}h_{\TheStrand-1}\twsC'h_{\TheStrand-1}^{-1}...h_2^{-1}. 
\end{align*}
In particular, both lines have the form 
\[
h_1...h_{\TheStrand-1}\twsC'h_{\TheStrand-1}^{-1}...h_{\Strand+1}^{-1}h_1...h_{\Strand-1}h_\Strand h_{\Strand-1}^{-1}...h_1^{-1}h_{\Strand+1}...h_{\TheStrand-1}\twsC'h_{\TheStrand-1}^{-1}...h_2^{-1}. 
\]
Moreover, the relations $h_{\Strand-1}h_\Strand h_{\Strand-1}^{-1}=h_\Strand^{-1}h_{\Strand-1}h_\Strand$ and $[h_\Str,h_\Strand]=1$ for $1\leq\Str<\Strand-1$ from \ref{prop:sec_pres_free_prod_2fact_relR1} used for $\Strand=\TheStrand-1$ to deduce (\ref{prop:sec_pres_free_prod_2fact_eq5_h1_tau_h1_tau}) hold similarly for every $2\leq\Strand<\TheStrand$. Iteratively, this allows us to deduce that 
\begin{align*}
h_1\varphi(\twC{2})h_1\varphi(\twC{2})\mystackrel{(\ref{prop:sec_pres_free_prod_2fact_eq4_h1_tau_h1_tau})}= & h_1...h_{\TheStrand-1}\twsC'h_{\TheStrand-1}^{-1}h_1...h_{\TheStrand-3}h_{\TheStrand-2}h_{\TheStrand-3}^{-1}...h_1^{-1}h_{\TheStrand-1}\twsC'h_{\TheStrand-1}^{-1}...h_2^{-1} 
\\
\mystackrel{\textit{it.}}= & h_1...h_{\TheStrand-1}\twsC'h_{\TheStrand-1}^{-1}...h_2^{-1}h_1h_2...h_{\TheStrand-1}\twsC'h_{\TheStrand-1}^{-1}...h_2^{-1} 
\\
\mystackrel{\eqref{prop:sec_pres_free_prod_2fact_eq_def_varphi_psi}}= & \varphi(\twC{2})h_1\varphi(\twC{2})h_1. 
\end{align*}

Consequently, by Theorem \ref{thm:von_Dyck}, the assignments $\varphi$ and $\psi$ induce inverse homomorphisms. Hence, the groups $A$ and $B$ are isomorphic, i.e.\ $\Z_\TheStrand(\Sigma_{\freeprodtwo})$ has the above presentation. 
\end{proof}

The presentation from Proposition \ref{prop:sec_pres_free_prod_2fact} allows us to deduce the semidirect product structure from Theorem \ref{thm-intro:general_Allcock}\ref{thm-intro:general_Allcock_it1}. For $\Z_\TheStrand(D_{\cycm})$, we can obtain the semidirect product structure from Theorem \ref{thm-intro:general_Allcock}\ref{thm-intro:general_Allcock_it2} in the same way, see Corollary \ref{cor:sec_pres_free_prod_1fact} for details. 

\begin{theorem}
\label{thm:semidir_prod_2fact}
For $\TheStrand\geq4$, the presentation from Proposition \textup{\ref{prop:sec_pres_free_prod_2fact}} implies that $\Z_\TheStrand(\Sigma_{\freeprodtwo})$ has a semidirect product structure 
\[
\tilde{A}(\Delta_\TheStrand^{\TheOrder,\TheOrder'})\rtimes(\cycm\times\cyc{\TheOrder'})
\]
with $\cycm\times\cyc{\TheOrder'}=\langle\twsC,\twsC'\rangle$ and $\tilde{A}(\Delta_\TheStrand^{\TheOrder,\TheOrder'})=\langle \htwC,h_1,...,h_{\TheStrand-1},\htw\rangle$. 

For $\TheOrder=\TheOrder'=2$ and $\TheStrand\geq3$, the presentation from Proposition \textup{\ref{prop:sec_pres_free_prod_2fact}} induces a semidirect product structure 
\[
A(\tilde{D}_\TheStrand)\rtimes(\cyc{2}\times\cyc{2})
\]
with $\cyc{2}\times\cyc{2}=\langle\twsC,\twsC'\rangle$ and $A(\tilde{D}_\TheStrand)=\langle \htwC,h_1,...,h_{\TheStrand-1},\htw\rangle$. 
\end{theorem}
\begin{proof}
To prove that the presentation from Proposition \ref{prop:sec_pres_free_prod_2fact} induces the above semidirect product structure, we want to apply Lemma \ref{lem:semidir_prod_pres}. 

Therefore, we identify the generators and relations of the semidirect factors: 
The normal subgroup is generated by $\htwC,h_1,...,h_{\TheStrand-1}$ and $\htw$ with the relations from \ref{prop:sec_pres_free_prod_2fact_relR1}-\ref{prop:sec_pres_free_prod_2fact_relR3}, i.e.\ the presentation that we denote by $\tilde{A}(\Delta_\TheStrand^{\TheOrder,\TheOrder'})$. The quotient is generated by $\twsC$ and $\twsC'$ with relations from \ref{prop:sec_pres_free_prod_2fact_relS1}, i.e.\ the group $\cycm\times\cyc{\TheOrder'}$. Moreover, the presentation contains the conjugation relations \ref{prop:sec_pres_free_prod_2fact_relC1}-\ref{prop:sec_pres_free_prod_2fact_relC4} which are of the form $xhx^{-1}=\phi_x(h)$ with 
\[
h\in\{\htwC,h_1,...,h_{\TheStrand-1},\htw\} \text{ and } x\in\{\twsC,\twsC'\} 
\] 
and $\phi_x(h)$ a word in the generators of $\tilde{A}(\Delta_\TheStrand^{\TheOrder,\TheOrder'})$. It remains to check that this presentation satisfies the conditions from Lemma \ref{lem:semidir_prod_pres_it2}. There we require that the assignments $\phi_x:h\mapsto\phi_x(h)$ induce an automorphism $\phi_x\in\Aut(\tilde{A}(\Delta_\TheStrand^{\TheOrder,\TheOrder'}))$ and the assignments $\phi:x\mapsto\phi_x$ induce a homomorphism $\cycm\times\cyc{\TheOrder'}\rightarrow\Aut(\tilde{A}(\Delta_\TheStrand^{\TheOrder,\TheOrder'}))$. For the proof, we follow the Steps \ref{rem:semidir_prod_pres_step1} and \ref{rem:semidir_prod_pres_step2} described in Remark \ref{rem:semidir_prod_pres}.

\begin{step}
\label{thm:semidir_prod_2fact_step1}
The assignments $\phi:x\mapsto\phi_x$ induce a homomorphism. 
\end{step}

The relations from \ref{prop:sec_pres_free_prod_2fact_relC1}-\ref{prop:sec_pres_free_prod_2fact_relC4} induce the following assignments: 
\begin{align*}
\phi_{\twsC}: & \htwC & \mapsto \ & (\htwC)^{-1}h_1\htwC , && \text{ and } & \phi_{\twsC'}: & 
\htwC & \mapsto \ & \htwC, & 
\\
& h_1 & \mapsto \ & \htwC, &&  &  & h_\Strand & \mapsto \ & h_\Strand \; \text{ for } \; 1\leq\Strand\leq\TheStrand-2, & 
\\
& h_\NStrand & \mapsto \ & h_\NStrand \; \text{ for } \; 2\leq\NStrand<\TheStrand, && &  & h_{\TheStrand-1} & \mapsto \ & h_{\TheStrand-1}^{\twsC'}, &  
\\
& h_{\TheStrand-1}^{\twsC'} & \mapsto \ & h_{\TheStrand-1}^{\twsC'} &&  &  & h_{\TheStrand-1}^{\twsC'}  & \mapsto \ & (h_{\TheStrand-1}^{\twsC'})^{-1}h_{\TheStrand-1}h_{\TheStrand-1}^{\twsC'}. &
\end{align*}

At first, we verify that the conjugation relations \ref{prop:sec_pres_free_prod_2fact_relC1}-\ref{prop:sec_pres_free_prod_2fact_relC4} imply that $\phi_\twsC^\TheOrder(h)=h$ for each $h\in\{\htwC,h_1,...,h_{\TheStrand-1},\htw\}$.

For this purpose, we begin with proving 
\begin{equation}
\label{prop:sec_pres_free_prod_2fact_eq_conj_tau^k}
\twsC^kh_1\twsC^{-k}=\begin{cases}
\langle(\htwC)^{-1},h_1^{-1}\rangle_{k-1}\langle\htwC,h_1\rangle_k & \text{ for } k \text{ odd} 
\\
\langle(\htwC)^{-1},h_1^{-1}\rangle_{k-1}\langle h_1,\htwC\rangle_k & \text{ for } k \text{ even} 
\end{cases}
\end{equation}
by induction on $k$. 

For $k=1$, by \ref{prop:sec_pres_free_prod_2fact_relC3}, we have $\twsC h_1\twsC^{-1}=\htwC$ and further, for $k=2$, by \ref{prop:sec_pres_free_prod_2fact_relC4} implies $\twsC^2h_1\twsC^{-2}=(\htwC)^{-1}h_1\htwC$. 

Now let us assume \eqref{prop:sec_pres_free_prod_2fact_eq_conj_tau^k} for a fixed $k$. For $k=2l+1$, we have 
\begin{align*}
\twsC^{k+1}h_1\twsC^{-(k+1)} & \mystackrel{(\ref{prop:sec_pres_free_prod_2fact_eq_conj_tau^k})}= \twsC\langle(\htwC)^{-1},h_1^{-1}\rangle_{k-1}\langle\htwC,h_1\rangle_k\twsC^{-1} 
\\
& \mystackrel{k=2l+1}= \twsC((\htwC)^{-1}h_1^{-1})^l(\htwC h_1)^l\htwC\twsC^{-1} 
\\
& \mystackrel{\ref{prop:sec_pres_free_prod_2fact_relC3},\ref{prop:sec_pres_free_prod_2fact_relC4}}= ((\htwC)^{-1}h_1^{-1}\textcolor{\short}{\htwC(\htwC)^{-1}})^l((\htwC)^{-1}h_1\htwC\textcolor{\short}{\htwC})^l\textcolor{\short}{(\htwC)^{-1}}h_1\htwC 
\\
& \mystackrel{}= ((\htwC)^{-1}h_1^{-1})^l(\htwC)^{-1}(h_1\htwC)^lh_1\htwC 
\\
& \mystackrel{k=2l+1}= \langle(\htwC)^{-1},h_1^{-1}\rangle_k\langle h_1,\htwC\rangle_{k+1}.  
\end{align*}
For $k=2l$, this implies 
\begin{align*}
\twsC^{k+1}h_1\twsC^{-(k+1)} & \mystackrel{(\ref{prop:sec_pres_free_prod_2fact_eq_conj_tau^k})}= \twsC\langle(\htwC)^{-1},h_1^{-1}\rangle_{k-1}\langle h_1,\htwC\rangle_k\twsC^{-1} 
\\
& \mystackrel{k=2l}= \twsC((\htwC)^{-1}h_1^{-1})^{l-1}(\htwC)^{-1}(h_1\htwC )^l\twsC^{-1} 
\\
& \mystackrel{\ref{prop:sec_pres_free_prod_2fact_relC3},\ref{prop:sec_pres_free_prod_2fact_relC4}}= ((\htwC)^{-1}h_1^{-1}\textcolor{\short}{\htwC(\htwC)^{-1}})^{l-1}(\htwC)^{-1}h_1^{-1}\htwC(\textcolor{\short}{\htwC(\htwC)^{-1}}h_1\htwC)^l 
\\
& \mystackrel{}= ((\htwC)^{-1}h_1^{-1})^l\htwC(h_1\htwC)^l 
\\
& \mystackrel{k=2l}= \langle(\htwC)^{-1},h_1^{-1}\rangle_k\langle\htwC,h_1\rangle_{k+1}. 
\end{align*}
This proves that the formula from (\ref{prop:sec_pres_free_prod_2fact_eq_conj_tau^k}) holds for every $k\in\NN$. Using relation~\ref{prop:sec_pres_free_prod_2fact_relR2}, this implies 
\begin{equation}
\label{prop:sec_pres_free_prod_2fact_eq_conj_tau_h1}
\twsC^\TheOrder h_1\twsC^{-\TheOrder}\stackrel{\eqref{prop:sec_pres_free_prod_2fact_eq_conj_tau^k}}=\left\lbrace\begin{matrix}
\langle(\htwC)^{-1},h_1^{-1}\rangle_{\TheOrder-1}\textcolor{\col}{\langle\htwC,h_1\rangle_\TheOrder} 
\\
\langle(\htwC)^{-1},h_1^{-1}\rangle_{\TheOrder-1}\textcolor{\col}{\langle h_1,\htwC\rangle_\TheOrder}
\end{matrix}\stackrel{\text{\ref{prop:sec_pres_free_prod_2fact_relR2}}}=\begin{matrix}
\langle(\htwC)^{-1},h_1^{-1}\rangle_{\TheOrder-1}\langle h_1,\htwC\rangle_\TheOrder 
\\
\langle(\htwC)^{-1},h_1^{-1}\rangle_{\TheOrder-1}\langle\htwC,h_1\rangle_\TheOrder
\end{matrix}\right\rbrace=h_1 \text{ for } \begin{matrix}
\TheOrder \text{ odd, }
\\
\TheOrder \text{ even. }
\end{matrix}
\end{equation}
This observation further implies 
\begin{equation}
\label{prop:sec_pres_free_prod_2fact_eq_conj_tau_htau}
\twsC^\TheOrder\htwC\twsC^{-\TheOrder}\stackrel{\text{\ref{prop:sec_pres_free_prod_2fact_relC3}}}=\twsC^{\TheOrder+1} h_1\twsC^{-(\TheOrder+1)}\stackrel{\text{\eqref{prop:sec_pres_free_prod_2fact_eq_conj_tau_h1}}}=\twsC h_1\twsC^{-1}\stackrel{\text{\ref{prop:sec_pres_free_prod_2fact_relC3}}}=\htwC. 
\end{equation}
Moreover, by relation \ref{prop:sec_pres_free_prod_2fact_relC1}, $\phi_\twsC(h_\NStrand)=h_\NStrand$ for $\NStrand\geq2$ and by relation \ref{prop:sec_pres_free_prod_2fact_relC2} $\phi_\twsC(\htw)=\htw$. Thus, we obtain that $\phi_\twsC^\TheOrder$ induces the identity map in $\tilde{A}(\Delta_\TheStrand^{\TheOrder,\TheOrder'})$. An analogous argument shows that $\phi_{\twsC'}^{\TheOrder'}$ induces the identity map in $\tilde{A}(\Delta_\TheStrand^{\TheOrder,\TheOrder'})$. Furthermore, the compositions $\phi_\twsC\circ\phi_{\twsC'}$ and $\phi_{\twsC'}\circ\phi_\twsC$ both induce the assignments 
\begin{align*}
\htwC & \mapsto (\htwC)^{-1}h_1\htwC, 
\\
h_1 & \mapsto \htwC, & 
\\
h_\Strand & \mapsto h_\Strand \; \text{ for } \; 2\leq\Strand\leq\TheStrand-2, 
\\
h_{\TheStrand-1} & \mapsto \htw, 
\\
\htw & \mapsto (\htw)^{-1}h_{\TheStrand-1}\htw, 
\end{align*}
i.e.\ the assignments $\phi$ preserve the relations from \ref{prop:sec_pres_free_prod_2fact_relS1}. Hence, $\phi$ by Theorem~\ref{thm:von_Dyck} induces a homomorphism $\cycm\times\cyc{\TheOrder'}\rightarrow\Aut(\tilde{A}(\Delta_\TheStrand^{\TheOrder,\TheOrder'}))$. 

\begin{step}
\label{thm:semidir_prod_2fact_step2}
The assignments $\phi_x:h\mapsto\phi_x(h)$ induce an automorphism in $\Aut(\tilde{A}(\Delta_\TheStrand^{\TheOrder,\TheOrder'}))$. 
\end{step}

As explained in Step \ref{rem:semidir_prod_pres_step2} of Remark \ref{rem:semidir_prod_pres}, this requires to check that the assignments $\phi_x:h\mapsto\phi_x(h)$ for $x=\twsC$ and $x=\twsC'$ preserve the relations of $\tilde{A}(\Delta_\TheStrand^{\TheOrder,\TheOrder'})$.

\newpage 
\setcounter{page}{22}

Let us focus on $\phi_\twsC$ at first. In Table \ref{tab:conj_tau} we list the defining relations of $\tilde{A}(\Delta_\TheStrand^{\TheOrder,\TheOrder'})$ and their $\phi_\twsC$-images. We omit relations $w=w'$ if $\phi_\twsC(w)=w$ and $\phi_\twsC(w')=w'$. 
For the most of the relations from \ref{prop:sec_pres_free_prod_2fact_relR1}-\ref{prop:sec_pres_free_prod_2fact_relR3}, it follows directly that they are preserved by $\phi_u$. For the less obvious relations, we refer to the explanations below Table \ref{tab:conj_tau}. 

\renewcommand{\arraystretch}{1.75}
{\setlength\LTleft{-0.125\textwidth}
\begin{tabularx}{1.25\linewidth}{rl|l}
\caption{}
\label{tab:conj_tau} \\
& Relations from $\tilde{A}(\Delta_\TheStrand^{\TheOrder,\TheOrder'}), \TheStrand\geq4$ & Images under $\phi_\twsC$ \\
\cline{2-3}
\noalign{\vskip\doublerulesep\vskip-\arrayrulewidth}
\cline{2-3}
\endfirsthead
& Relations from $\tilde{A}(\Delta_\TheStrand^{\TheOrder,\TheOrder'}), \TheStrand\geq4$ & Images under $\phi_\twsC$ \\
\cline{2-3}
\noalign{\vskip\doublerulesep\vskip-\arrayrulewidth}
\cline{2-3}
\endhead
\multicolumn{3}{r}{\footnotesize(To be continued)}
\endfoot
\cline{2-3}
\noalign{\vskip\doublerulesep\vskip-\arrayrulewidth}
\cline{2-3}
\endlastfoot
\multirow{4}{3.5em}{\ref{prop:sec_pres_free_prod_2fact_relR1}$\left\{\rule{0cm}{1.5cm}\right.$}& $h_1h_2h_1=h_2h_1h_2$ & $\htwC h_2\htwC\stackrel{\text{\ref{prop:sec_pres_free_prod_2fact_relR1}}}=h_2\htwC h_2$, 
\\ \cline{2-3}
& $[h_1,h_\Strand]=1$ \; for \; $\Strand\geq3$ & $[\htwC,h_\Strand]\stackrel{\text{\ref{prop:sec_pres_free_prod_2fact_relR1}}}=1$, 
\\ \cline{2-3}
& $\htwC h_2\htwC=h_2\htwC h_2$ & $(\htwC)^{-1}h_1\htwC h_2(\htwC)^{-1}h_1\htwC=h_2(\htwC)^{-1}h_1\htwC h_2$ by \ref{prop:sec_pres_free_prod_2fact_par_cpxbr_rel}, 
\\ \cline{2-3}
& $[\htwC,h_\Strand]=1$ \; for \; $\Strand\geq3$ & $[(\htwC)^{-1}h_1\htwC,h_\Strand]=1$ by \ref{prop:sec_pres_free_prod_2fact_relR1}, 
\\ \cline{2-3}
\multirow{1}{3.5em}{\ref{prop:sec_pres_free_prod_2fact_relR2}$\left\{\rule{0cm}{0.4cm}\right.$}& $\langle h_1,\htwC\rangle_\TheOrder=\langle\htwC,h_1\rangle_\TheOrder$ & $\langle\htwC,(\htwC)^{-1}h_1\htwC\rangle_\TheOrder=\langle(\htwC)^{-1}h_1\htwC,\htwC\rangle_\TheOrder$ by \ref{prop:sec_pres_free_prod_2fact_par_conj_Artin-rel_length_m}, 
\\ \cline{2-3}
\multirow{1}{3.5em}{\ref{prop:sec_pres_free_prod_2fact_relR3}$\left\{\rule{0cm}{0.4cm}\right.$}& $(h_1\htwC h_2)^2=(h_2h_1\htwC)^2$ & $(\textcolor{\short}{\htwC(\htwC)^{-1}}h_1\htwC h_2)^2\stackrel{\text{\ref{prop:sec_pres_free_prod_2fact_relR3}}}=(h_2\textcolor{\short}{\htwC(\htwC)^{-1}}h_1\htwC)^2$ 
\end{tabularx}
}

\begin{enumerate}[leftmargin=1.3cm,label={\textbf{\thetheorem.\arabic*.}},ref={\thetheorem.\arabic*}]
\item 
\label{prop:sec_pres_free_prod_2fact_par_cpxbr_rel}
\vspace*{-8pt}To deduce that $\phi_\twsC$ preserves the relation $\htwC h_2\htwC=h_2\htwC h_2$ from \ref{prop:sec_pres_free_prod_2fact_relR1}, 
\end{enumerate}
\noindent we observe: 
\begin{align*}
 & (\htwC)^{-1}h_1\htwC h_2(\htwC)^{-1}h_1\htwC & = \ & h_2(\htwC)^{-1}h_1\htwC h_2 
\\
\mystackrel{$\vee$}\Leftrightarrow & (\htwC)^{-1}(h_2^{-1}h_2)h_1\htwC h_2(\htwC)^{-1}h_1\htwC & = \ & \textcolor{\col}{h_2(\htwC)^{-1}(h_2^{-1}}h_2)h_1\htwC h_2 
\\
\mystackrel{\ref{prop:sec_pres_free_prod_2fact_relR1}}\Leftrightarrow & \textcolor{\short}{(\htwC)^{-1}h_2^{-1}}h_2h_1\htwC h_2(\htwC)^{-1}h_1\htwC & = \ & \textcolor{\short}{(\htwC)^{-1}h_2^{-1}}\htwC h_2h_1\htwC h_2 \quad\vert h_2\htwC\cdot 
\\
\mystackrel{}\Leftrightarrow & h_2h_1\textcolor{\col}{\htwC h_2(\htwC)^{-1}}h_1\htwC & = \ & \htwC h_2h_1\htwC h_2 
\\
\mystackrel{\ref{prop:sec_pres_free_prod_2fact_relR1}}\Leftrightarrow & \textcolor{\col}{h_2h_1h_2^{-1}}\htwC h_2h_1\htwC & = \ & \htwC h_2h_1\htwC h_2 
\\
\mystackrel{\ref{prop:sec_pres_free_prod_2fact_relR1}}\Leftrightarrow & \textcolor{\short}{h_1^{-1}}h_2h_1\htwC h_2h_1\htwC & = \ & \htwC h_2h_1\htwC h_2 \hspace*{54pt}\vert h_1\cdot 
\\
\mystackrel{}\Leftrightarrow & h_2h_1\htwC h_2h_1\htwC & = \ & h_1\htwC h_2h_1\htwC h_2. 
\end{align*}
The final relation is contained in \ref{prop:sec_pres_free_prod_2fact_relR3}. Hence, $\phi_\twsC$ preserves the relation $\htwC h_2\htwC=h_2\htwC h_2$. 
\vspace*{5pt}
\begin{enumerate}[resume,leftmargin=1.3cm,label={\textbf{\thetheorem.\arabic*.}},ref={\thetheorem.\arabic*}]
\item 
\label{prop:sec_pres_free_prod_2fact_par_conj_Artin-rel_length_m}
\vspace*{-8pt}To deduce that $\phi_\twsC$ preserves the relation $\langle h_1,\htwC\rangle_\TheOrder=\langle\htwC,h_1\rangle_\TheOrder$ from
\end{enumerate}
\ref{prop:sec_pres_free_prod_2fact_relR2}, we prove the following equivalence: 
\[
\langle \htwC,(\htwC)^{-1}h_1\htwC\rangle_\TheOrder=\langle(\htwC)^{-1}h_1\htwC,\htwC\rangle_\TheOrder \Leftrightarrow \langle h_1,\htwC\rangle_\TheOrder=\langle\htwC,h_1\rangle_\TheOrder.  
\]
For $\TheOrder=2l$, this follows from 
\begin{align*}
& (\textcolor{\short}{\htwC(\htwC)^{-1}}h_1\htwC)^l=((\htwC)^{-1}h_1\htwC\htwC)^l && 
\\
\Leftrightarrow & (h_1\textcolor{\short}{\htwC})^l=\textcolor{\short}{(\htwC)^{-1}}(h_1\htwC)^l\textcolor{\short}{\htwC} && \vert \htwC\cdot\;\cdot(\htwC)^{-1}
\\
\Leftrightarrow & (\htwC h_1)^l=(h_1\htwC)^l. && 
\end{align*}
For $\TheOrder=2l+1$, this follows from 
\begin{align*}
& (\textcolor{\short}{\htwC(\htwC)^{-1}}h_1\htwC)^l\htwC=((\htwC)^{-1}h_1\htwC\textcolor{\short}{\htwC})^l\textcolor{\short}{(\htwC)^{-1}}h_1\htwC && 
\\
\Leftrightarrow & (h_1\htwC)^l\textcolor{\short}{\htwC}=\textcolor{\short}{(\htwC)^{-1}}(h_1\htwC)^lh_1\textcolor{\short}{\htwC} && \vert\htwC\cdot\;\cdot(\htwC)^{-1}
\\
\Leftrightarrow & \htwC(h_1\htwC)^l=(h_1\htwC)^lh_1. && 
\end{align*}

Table \ref{tab:conj_tau} implies that the assignments $h\mapsto\phi_\twsC(h)$ induce an endomorphism of $\tilde{A}(\Delta_\TheStrand^{\TheOrder,\TheOrder'})$ for $\TheStrand\geq4$. Since $\phi$ preserves the relation $\twsC^\TheOrder=1$, the inverse of the endomorphism induced by $h\mapsto\phi_\twsC(h)$ is given by assignments $h\mapsto\phi_\twsC^{\TheOrder-1}(h)$, i.e.\ $h\mapsto\phi_\twsC(h)$ induces an automorphism. 

Due to the symmetry of the presentation of $\tilde{A}(\Delta_\TheStrand^{\TheOrder,\TheOrder'})$, similar observations hold for the map $h\mapsto\phi_{\twsC'}(h)$. Hence, the above presentation satisfies the conditions from Lemma \ref{lem:semidir_prod_pres_it2} and describes the semidirect product $\tilde{A}(\Delta_\TheStrand^{\TheOrder,\TheOrder'})\rtimes(\cycm\times\cyc{\TheOrder'})$ for~$\TheStrand\geq4$. 

The only part where the additional relation $(h_1\htwC h_2)^2=(h_2h_1\htwC)^2$ from \ref{prop:sec_pres_free_prod_2fact_relR3} was required was the Step \ref{thm:semidir_prod_2fact_step2}. There, the relation was necessary to make sure that $\phi_\twsC$ preserves the relation $\htwC h_2\htwC=h_2\htwC h_2$, see \ref{prop:sec_pres_free_prod_2fact_par_cpxbr_rel}. Analogously, the relation $(h_{\TheStrand-1}\htw h_{\TheStrand-2})^2=(h_{\TheStrand-2}h_{\TheStrand-1}\htw)^2$ is required when we apply $\phi_{\twsC'}$ to the relation $\htw h_{\TheStrand-2}\htw=h_{\TheStrand-2}\htw h_{\TheStrand-2}$. 

In the following, we observe that this additional relation is redundant in the case where both cone points are of order two. In this case, the relations from \ref{prop:sec_pres_free_prod_2fact_relR2} are given by $h_1\htwC=\htwC h_1$ and $h_{\TheStrand-1}\htw=\htw h_{\TheStrand-1}$. Hence, the $\phi_\twsC$-image of $\htwC h_2\htwC=h_2\htwC h_2$ simplifies to $\htwC h_2\htwC=h_2\htwC h_2$ which is covered by \ref{prop:sec_pres_free_prod_2fact_relR1}. Analogously, the $\phi_{\twsC'}$-image of $\htw h_{\TheStrand-2}\htw=h_{\TheStrand-2}\htw h_{\TheStrand-2}$ simplifies to $\htw h_{\TheStrand-2}\htw=h_{\TheStrand-2}\htw h_{\TheStrand-2}$ which is also covered by \ref{prop:sec_pres_free_prod_2fact_relR1}. Thus, we may omit relation \ref{prop:sec_pres_free_prod_2fact_relR3} in the presentation from Proposition \ref{prop:sec_pres_free_prod_2fact} for $\TheOrder=\TheOrder'=2$ and the presentation still satisfies the conditions from Lemma \ref{lem:semidir_prod_pres_it2}. That is $\Z_\TheStrand(\Sigma_{\cyc{2}\ast\cyc{2}})=A(\tilde{D}_\TheStrand)\rtimes(\cyc{2}\times\cyc{2})$. 
\end{proof}

\begin{remark}
\label{rem:additional_rel}
We also obtain a similar semidirect product structure 
\[
\Z_\TheStrand(\Sigma_{\freeprodtwo})=\langle\htwC,h_1,h_2,h_2^{\twsC'}\rangle\rtimes(\cycm\times\cyc{\TheOrder'})
\]
for $\TheStrand=3$ and $\max\{\TheOrder,\TheOrder'\}>2$. The only difference in this case is that the relations in \ref{prop:sec_pres_free_prod_2fact_relR3} are replaced by the following relations: 
\begin{enumerate}[label={(R\arabic*')},ref={\thetheorem(R\arabic*')}]
\setcounter{enumi}{2}
\item
\label{prop:sec_pres_free_prod_2fact_relR3_n3}
\begin{align*}
(h_1\htwC((h_2^{\twsC'})^{-1}h_2^{-1})^kh_2^{\twsC'}(h_2h_2^{\twsC'})^k)^2 & =((h_2^{\twsC'})^{-1}h_2^{-1})^kh_2^{\twsC'}(h_2h_2^{\twsC'})^kh_1\htwC)^2 \text{ and }
\\
(h_1\htwC((h_2^{\twsC'})^{-1}h_2^{-1})^{k'}h_2(h_2h_2^{\twsC'})^{k'})^2 & =((h_2^{\twsC'})^{-1}h_2^{-1})^{k'}h_2(h_2h_2^{\twsC'})^{k'}h_1\htwC)^2 
\end{align*}
for $0\leq k,k'<l'$ if $\TheOrder'=2l'$ and $0\leq k<l'$, $0\leq k'\leq l'$ if $\TheOrder'=2l'+1$ and 
\begin{align*}
(h_2h_2^{\twsC'}((\htwC)^{-1}h_1^{-1})^k\htwC(h_1\htwC)^k)^2 & =(((\htwC)^{-1}h_1^{-1})^k\htwC(h_1\htwC)^kh_2h_2^{\twsC'})^2 \text{ and }
\\
(h_2h_2^{\twsC'}((\htwC)^{-1}h_1^{-1})^{k'}h_1(h_1\htwC)^{k'})^2 & =(((\htwC)^{-1}h_1^{-1})^{k'}h_1(h_1\htwC)^{k'}h_2h_2^{\twsC'})^2
\end{align*}
for $0\leq k,k'<l$ if $\TheOrder=2l$ and $0\leq k<l$, $0\leq k'\leq l$ if $\TheOrder=2l+1$. 
\end{enumerate}
These relations are required to deduce that $\phi_\twsC$ and $\phi_{\twsC'}$ both preserve the relation $\htwC h_2^{\twsC'}\htwC=h_2^{\twsC'}\htwC h_2^{\twsC'}$. At first, analogous to \ref{prop:sec_pres_free_prod_2fact_par_cpxbr_rel}, this makes the additional relations $(h_1\htwC h_2^{\twsC'})^2=(h_2^{\twsC'}h_1\htwC)^2$ and $(h_2 h_2^{\twsC'}\htwC)^2=(\htwC h_2h_2^{\twsC'})^2$ necessary. To guarantee that these additional relations are iteratively preserved by $\phi_\twsC$ and $\phi_{\twsC'}$, the relations~\ref{prop:sec_pres_free_prod_2fact_relR3_n3} are sufficient. 

For $\TheOrder=\TheOrder'=2$, the assignments $\phi_\twsC$ map the relation $\htwC h_2^{\twsC'}\htwC=h_2^{\twsC'}\htwC h_2^{\twsC'}$ to $h_1h_2^{\twsC'}h_1=h_2^{\twsC'}h_1h_2^{\twsC'}$ and the assignments $\phi_{\twsC'}$ map it to $\htwC h_2\htwC=h_2\htwC h_2$. Hence, no relations of type \ref{prop:sec_pres_free_prod_2fact_relR3_n3} are required in the case $\TheOrder=\TheOrder'=2$. 
\end{remark}

\newpage 

Omitting the generators $\twC{2}$ in $A$ and $\twsC',\htw$ in $B$, the above proof also yields Theorem \ref{thm-intro:general_Allcock}\ref{thm-intro:general_Allcock_it2}: 

\begin{corollary}
\label{cor:sec_pres_free_prod_1fact}
For $\TheStrand\geq2$, the group $\Z_\TheStrand(D_{\cycm})$ has the semidirect product structure $\tilde{A}(\Delta_\TheStrand^\TheOrder)\rtimes\cycm$. For $\TheOrder=2$, the group $\tilde{A}(\Delta_\TheStrand^2)$ is an Artin group of type~$D_\TheStrand$. For $\TheStrand=2$, the group $\tilde{A}(\Delta_2^\TheOrder)$ is an Artin group of type $I_2(\TheOrder)$. For $\TheStrand,\TheOrder\geq3$, the group $\tilde{A}(\Delta_\TheStrand^\TheOrder)$ is the complex braid group of type $\cpxbr{\TheOrder}{\TheStrand}$. 
\end{corollary}
\begin{proof}
For $\TheStrand\geq4$, this is literally covered by the arguments in Proposition \ref{prop:sec_pres_free_prod_2fact} and Theorem \ref{thm:semidir_prod_2fact}. Moreover, for $\TheStrand=3$, we omit the generator $h_2^{\twsC'}$. Thus, the relation $\htwC h_2^{\twsC'}\htwC=h_2^{\twsC'}\htwC h_2^{\twsC'}$ is not relevant in this case. Following Remark \ref{rem:additional_rel}, this yields that the additional relations from \ref{prop:sec_pres_free_prod_2fact_relR3_n3} are not required and $\langle\htwC,h_1,h_2\rangle$ is isomorphic to $\tilde{A}(\Delta_3^\TheOrder)$. 

For $\TheStrand=2$, Theorem \ref{thm:pres_Z_n} implies that the group $\Z_2(D_{\cycm})$ has a presentation~$A$ with generators $h_1$ and $\twC{1}$ and defining relations $\twC{1}^\TheOrder=1$ and $[h_1\twC{1}h_1,\twC{1}]=1$. Moreover, let 
\[
B:=\langle\twsC,h_1,\htwC\mid \twsC^\TheOrder=1,\langle h_1,\htwC\rangle_\TheOrder=\langle\htwC,h_1\rangle_\TheOrder,\twsC h_1\twsC^{-1}=\htwC, \twsC\htwC\twsC^{-1}=(\htwC)^{-1}h_1\htwC\rangle. 
\]
As in Proposition \ref{prop:sec_pres_free_prod_2fact}, we define assignments 
\begin{align*}
\varphi:A & \rightarrow B, & & \text{and} & \psi:B & \rightarrow A, 
\\
h_1 & \mapsto h_1, & & & h_1 & \mapsto h_1, 
\\
\twC{1} & \mapsto \twsC & & & \htwC & \mapsto\twC{1}h_1\twC{1}^{-1}, 
\\
& & & & \twsC & \mapsto\twC{1} 
\end{align*}
to show that $Z_2(D_{\cycm})$ has the presentation $B$.  

The definitions of $\varphi$ and $\psi$ directly imply that the assignments $\psi\circ\varphi$ describe the identical map on the letters $h_1$ and $\twC{1}$ and the assignments $\varphi\circ\psi$ yield the identical map on the generators of $B$ (modulo relations in $B$).

To deduce that these assignments induce homomorphisms we apply Theorem~\ref{thm:von_Dyck}. Therefore, it remains to check that $\varphi$ and $\psi$ preserve the relations of $A$ and $B$, respectively. 

Under $\varphi$ the relation $\twC{1}^\TheOrder=1$ maps to $\twsC^\TheOrder=1$ which is covered by $B$.  
The same argument as in \ref{prop:sec_pres_free_prod_2fact_par_th1th1=h1th1t} shows that $\twsC h_1\twsC^{-1}=\htwC$ and $\twsC\htwC\twsC^{-1}=(\htwC)^{-1}h_1\htwC$ imply $[h_1\twsC h_1,\twsC]=1$, i.e.\ $\varphi$ preserves the relation $[h_1\twC{1}h_1,\twC{1}]=~1$. 

Under $\psi$ the relation $\twsC^\TheOrder=1$ maps to $\twC{1}^\TheOrder=1$ which is covered by the relations in~$A$. Moreover, $\psi$ maps $\twsC h_1\twsC^{-1}=\htwC$ to $\twC{1}h_1\twC{1}^{-1}=\twC{1}h_1\twC{1}^{-1}$, i.e.\ $\psi$ preserves the first conjugation relation. 
The second conjugation relation $\twsC\htwC\twsC^{-1}=(\htwC)^{-1}h_1\htwC$ maps to $\twC{1}^2h_1\twC{1}^{-2}=\twC{1}h_1^{-1}\twC{1}^{-1}h_1\twC{1}h_1\twC{1}^{-1}$ which follows from relation \ref{cor:pres_orb_braid_free_prod_rel4}. 
As explained in \eqref{prop:sec_pres_free_prod_2fact_eq_even} and \eqref{prop:sec_pres_free_prod_2fact_eq_odd}, the relation $\langle\psi(h_1),\psi(\htwC)\rangle_\TheOrder=\langle\psi(\htwC),\psi(h_1)\rangle_\TheOrder$ holds in the group with presentation $A$. 

Thus, by Theorem \ref{thm:von_Dyck}, $\varphi$ and $\psi$ induce inverse homomorphisms. Hence, the group $\Z_2(D_{\cycm})$ has the presentation $B$. 

It remains to verify that $B$ represents a semidirect product with normal subgroup $\tilde{A}(\Delta_2^\TheOrder)=A(I_2(\TheOrder))$ generated by $h_1$ and $\htwC$ and the quotient $\cycm=\langle\twsC\rangle$. By  Lemma~\ref{lem:semidir_prod_pres}, this requires that $\phi_\twsC^\TheOrder$ induces a trivial map in $A(I_2(\TheOrder))$ and $\phi_\twsC$ induces an automorphism in terms of the $A(I_2(\TheOrder))$-presentation. 

That $\phi_\twsC^\TheOrder(h_1)=h_1$ and $\phi_\twsC^\TheOrder(\htwC)=\htwC$ follows as in (\ref{prop:sec_pres_free_prod_2fact_eq_conj_tau_h1}) and (\ref{prop:sec_pres_free_prod_2fact_eq_conj_tau_htau}). Moreover, $\phi_\twsC$ sends $\langle h_1,\htwC\rangle_\TheOrder=\langle\htwC,h_1\rangle_\TheOrder$ to 
\[
\langle\htwC,(\htwC)^{-1}h_1\htwC\rangle_\TheOrder=\langle(\htwC)^{-1}h_1\htwC,\htwC\rangle_\TheOrder. 
\]
Following \ref{prop:sec_pres_free_prod_2fact_par_conj_Artin-rel_length_m}, this relation is equivalent to $\langle h_1,\htwC\rangle_\TheOrder=\langle\htwC,h_1\rangle_\TheOrder$. This proves that $\Z_2(D_{\cycm})=A(I_2(\TheOrder))\rtimes\cycm$. 
\end{proof}

\subsection*{Orbifold braid groups with one puncture}

If we consider the braid group $\Z_\TheStrand(D_{\cycm}(1))$ with $\TheOrder$ additional punctures, we identify the semidirect product structure from Theorem \ref{thm-intro:general_Allcock}\ref{thm-intro:general_Allcock_it3}. 

\begin{proposition}
\label{prop:sec_pres_free_prod_subgrp_1fact}
For $\TheStrand\geq3$, the presentation of $\Z_\TheStrand(D_{\cycm}(1))$ given in Theorem~\textup{\ref{thm:pres_Z_n}} is isomorphic to the group generated by $h_1,...,h_{\TheStrand-1},h_{\TheStrand-1}^\twsC,\twsP,\bar{\twsC}$ with defining relations 
\begin{enumerate}[label={\textup{(R\arabic*)}},ref={\thetheorem(R\arabic*)}]
\item 
\label{prop:sec_pres_free_prod_subgrp_1fact_relR1} 
braid and commutator relations for the generators $h_1,...,h_{\TheStrand-1},h_{\TheStrand-1}^\twsC$, 
\item 
\label{prop:sec_pres_free_prod_subgrp_1fact_relR2}
$\twsP h_1\twsP h_1=h_1\twsP h_1\twsP$, $\;[\twsP,h_\Strand]=1$ \; for \; $2\leq\Strand<\TheStrand\;$ and $\;[\twsP,h_{\TheStrand-1}^\twsC]=1$, 
\item 
\label{prop:sec_pres_free_prod_subgrp_1fact_relR3}
$\langle h_{\TheStrand-1},h_{\TheStrand-1}^\twsC\rangle_\TheOrder=\langle h_{\TheStrand-1}^\twsC,h_{\TheStrand-1}\rangle_\TheOrder$, 
\item 
\label{prop:sec_pres_free_prod_subgrp_1fact_rel5}
$(h_{\TheStrand-1}h_{\TheStrand-1}^\twsC h_{\TheStrand-2})^2=(h_{\TheStrand-2}h_{\TheStrand-1}h_{\TheStrand-1}^\twsC)^2$, 
\end{enumerate}
\begin{enumerate}[label={\textup{(S\arabic*)}},ref={\thetheorem(S\arabic*)}]
\item 
\label{prop:sec_pres_free_prod_subgrp_1fact_relS1} 
$\bar{\twsC}^\TheOrder=1$, 
\end{enumerate}
\begin{enumerate}[label={\textup{(C\arabic*)}},ref={\thetheorem(C\arabic*)}]
\item 
\label{prop:sec_pres_free_prod_subgrp_1fact_relC1}
$\bar{\twsC}h_\Strand\bar{\twsC}^{-1}=h_\Strand$ \; for \; $1\leq\Strand\leq\TheStrand-2\;$ and $\;\bar{\twsC}\twsP\bar{\twsC}^{-1}=\twsP$ and 
\item 
\label{prop:sec_pres_free_prod_subgrp_1fact_relC2}
$\bar{\twsC}h_{\TheStrand-1}\bar{\twsC}^{-1}=h_{\TheStrand-1}^\twsC\;$ and $\;\bar{\twsC}h_{\TheStrand-1}^\twsC\bar{\twsC}^{-1}=\left(h_{\TheStrand-1}^\twsC\right)^{-1}h_{\TheStrand-1}h_{\TheStrand-1}^\twsC$. 
\end{enumerate}
\end{proposition}
\begin{proof}
Let $A$ be the group given by the presentation from Theorem \ref{thm:pres_Z_n} in the case $\TheCone=1$ and $\ThePct=1$. Moreover, let $B$ denote the group given by the presentation above. As in Proposition~\ref{prop:sec_pres_free_prod_2fact}, we define the following assignments with respect to the definitions of $\twsP, h_{\TheStrand-1}^\twsC$ and $\bar{\twsC}$: 
\begin{align*}
\varphi:A & \rightarrow B, & \text{and} & & \psi: B & \rightarrow A
\\
h_\Strand & \mapsto h_\Strand \; \text{ for } \; 1\leq\Strand<\TheStrand, & & & h_\Strand & \mapsto h_\Strand \; \text{ for } \; 1\leq\Strand<\TheStrand,
\\
\twP{1} & \mapsto \twsP, & & & \twsP & \mapsto \twP{1}, 
\\
\twC{1} & \mapsto h_1...h_{\TheStrand-1}\bar{\twsC}h_{\TheStrand-1}^{-1}...h_1^{-1} & & & \bar{\twsC} & \mapsto h_{\TheStrand-1}^{-1}...h_1^{-1}\twC{1}h_1...h_{\TheStrand-1}, 
\\
& & & & h_{\TheStrand-1}^\twsC & \mapsto \psi(\bar{\twsC})h_{\TheStrand-1}\psi(\bar{\twsC})^{-1}. 
\end{align*}
By \eqref{eq:def_a_kc}, the element $\psi(\bar{\twsC})$ coincides with $\twistC{\TheStrand}{1}$.  

As in the proof of Proposition \ref{prop:sec_pres_free_prod_2fact}, it follows directly that the assignments $\psi\circ\varphi$ describe the identical map on the generators of $A$ and the assignments $\varphi\circ\psi$ describe the identical map on the set of generators of $B$ (modulo relations in $B$). 

To deduce that the assignments $\varphi$ and $\psi$ induce homomorphisms between the groups presented by $A$ and $B$ we will apply Theorem \ref{thm:von_Dyck}. Therefore, it remains to check that $\varphi$ and $\psi$ preserve the relation of $A$ and $B$, respectively. For the assignments $\psi$, this follows from the relations in Lemma~\ref{lem:psi_rel}. We summarize this observation in Table~\ref{tab:subgrp_rel_psi}.  
\renewcommand{\arraystretch}{1.75}
{\small
{\setlength\LTleft{-0.175\textwidth}
\begin{tabularx}{1.35\linewidth}{rl|l}
\caption{}
\label{tab:subgrp_rel_psi} \\
& Relations from $B$ & Images under $\psi$ \\
\cline{2-3}
\noalign{\vskip\doublerulesep\vskip-\arrayrulewidth}
\cline{2-3}
\endfirsthead
& Relations from $B$ & Images under $\psi$ \\
\cline{2-3}
\noalign{\vskip\doublerulesep\vskip-\arrayrulewidth}
\cline{2-3}
\endhead
\multicolumn{3}{r}{\footnotesize(To be continued)}
\endfoot
\cline{2-3}
\noalign{\vskip\doublerulesep\vskip-\arrayrulewidth}
\cline{2-3}
\endlastfoot
\multirow{3}{4em}{\ref{prop:sec_pres_free_prod_subgrp_1fact_relR1}$\left\{\rule{0cm}{1.65cm}\right.$}& \doubletable{braid and commutator relations \\ for $h_1,...,h_{\TheStrand-1}$} & \doubletable{braid and commutator relations \\ for $h_1,...,h_{\TheStrand-1}$ are covered by \ref{cor:pres_orb_braid_free_prod_rel2}}
\\ \cline{2-3} 
& $[h_\Strand,h_{\TheStrand-1}^\twsC]=1$ \; for \; $\Strand\leq\TheStrand-3$ & $[h_\Strand,\twistC{\TheStrand}{1}h_{\TheStrand-1}\twistC{\TheStrand}{1}^{-1}]=1$ by \ref{cor:pres_orb_braid_free_prod_rel2} and \ref{lem:psi_rel_C1}
\\ \cline{2-3} 
& $h_{\TheStrand-2}h_{\TheStrand-1}^\twsC h_{\TheStrand-2}=h_{\TheStrand-1}^\twsC h_{\TheStrand-2}h_{\TheStrand-1}^\twsC$ & \doubletable{$h_{\TheStrand-2}\twistC{\TheStrand}{1}h_{\TheStrand-1}\twistC{\TheStrand}{1}^{-1}h_{\TheStrand-2}=\twistC{\TheStrand}{1}h_{\TheStrand-1}\twistC{\TheStrand}{1}^{-1}h_{\TheStrand-2}\twistC{\TheStrand}{1}h_{\TheStrand-1}\twistC{\TheStrand}{1}^{-1}$ \\ by \ref{cor:pres_orb_braid_free_prod_rel2} and \ref{lem:psi_rel_C1}} 
\\ \cline{2-3} 
\multirow{3}{4em}{\ref{prop:sec_pres_free_prod_subgrp_1fact_relR2}$\left\{\rule{0cm}{1.05cm}\right.$}& $\twsP h_1\twsP h_1=h_1\twsP h_1\twsP$ & $\twP{1}h_1\twP{1}h_1\stackrel{\ref{cor:pres_orb_braid_free_prod_rel4}}=h_1\twP{1}h_1\twP{1}$ 
\\ \cline{2-3} 
& $[\twsP,h_\Strand]=1$ \; for \; $2\leq\Strand<\TheStrand$ & $[\twP{1},h_\Strand]\stackrel{\ref{cor:pres_orb_braid_free_prod_rel3}}=1$ 
\\ \cline{2-3}  
& $[\twsP,h_{\TheStrand-1}^\twsC]=1$ & $[\twP{1},\twistC{\TheStrand}{1}h_{\TheStrand-1}\twistC{\TheStrand}{1}^{-1}]=1$ by \ref{cor:pres_orb_braid_free_prod_rel3} and \ref{lem:psi_rel_C3}
\\ \cline{2-3}  
\multirow{1}{4em}{\ref{prop:sec_pres_free_prod_subgrp_1fact_relR3}$\left\{\rule{0cm}{0.3cm}\right.$}& $\langle h_{\TheStrand-1},h_{\TheStrand-1}^\twsC\rangle_\TheOrder=\langle h_{\TheStrand-1}^\twsC,h_{\TheStrand-1}\rangle_\TheOrder$ & $\langle h_{\TheStrand-1},\twistC{\TheStrand}{1}h_{\TheStrand-1}\twistC{\TheStrand}{1}^{-1}\rangle_\TheOrder\stackrel{\text{\ref{lem:psi_rel_R2}}}=\langle \twistC{\TheStrand}{1}h_{\TheStrand-1}\twistC{\TheStrand}{1}^{-1},h_{\TheStrand-1}\rangle_\TheOrder$ 
\\ \cline{2-3} 
\multirow{1}{4em}{\ref{prop:sec_pres_free_prod_subgrp_1fact_rel5}$\left\{\rule{0cm}{0.3cm}\right.$}& $(h_{\TheStrand-1}h_{\TheStrand-1}^\twsC h_{\TheStrand-2})^2=(h_{\TheStrand-2}h_{\TheStrand-1}h_{\TheStrand-1}^\twsC)^2$ & $(h_{\TheStrand-1}\twistC{\TheStrand}{1}h_{\TheStrand-1}\twistC{\TheStrand}{1}^{-1}h_{\TheStrand-2})^2\stackrel{\text{\ref{lem:psi_rel_R4}}}=(h_{\TheStrand-2}h_{\TheStrand-1}\twistC{\TheStrand}{1}h_{\TheStrand-1}\twistC{\TheStrand}{1}^{-1})^2$
\\ \cline{2-3} 
\multirow{1}{3.75em}{\ref{prop:sec_pres_free_prod_subgrp_1fact_relS1}$\left\{\rule{0cm}{0.3cm}\right.$}& $\bar{\twsC}^\TheOrder=1$ & $\twistC{\TheStrand}{1}^\TheOrder\stackrel{\text{\ref{lem:psi_rel_S1}}}=1$ 
\\ \cline{2-3} 
\multirow{2}{4em}{\ref{prop:sec_pres_free_prod_subgrp_1fact_relC1}$\left\{\rule{0cm}{0.7cm}\right.$}& $[h_\Strand,\bar{\twsC}]=1$ \; for \; $\Strand\leq\TheStrand-2$ & $[h_\Strand,\twistC{\TheStrand}{1}]\stackrel{\text{\ref{lem:psi_rel_C1}}}=1$
\\ \cline{2-3} 
& $[\twsP,\bar{\twsC}]=1$ & $[\twP{1},\twistC{\TheStrand}{1}]\stackrel{\text{\ref{lem:psi_rel_C3}}}=1$ 
\\ \cline{2-3} 
\multirow{2}{4em}{\ref{prop:sec_pres_free_prod_subgrp_1fact_relC2}$\left\{\rule{0cm}{0.7cm}\right.$}& $\bar{\twsC}h_{\TheStrand-1}\bar{\twsC}^{-1}=h_{\TheStrand-1}^\twsC$ & $\twistC{\TheStrand}{1}h_{\TheStrand-1}\twistC{\TheStrand}{1}^{-1}=\twistC{\TheStrand}{1}h_{\TheStrand-1}\twistC{\TheStrand}{1}^{-1}$ 
\\ \cline{2-3} 
& $\bar{\twsC}h_{\TheStrand-1}^\twsC\bar{\twsC}^{-1}=(h_{\TheStrand-1}^\twsC)^{-1}h_{\TheStrand-1}h_{\TheStrand-1}^\twsC$ & $\twistC{\TheStrand}{1}^2h_{\TheStrand-1}\twistC{\TheStrand}{1}^{-2}=\twistC{\TheStrand}{1}h_{\TheStrand-1}^{-1}\twistC{\TheStrand}{1}^{-1}h_{\TheStrand-1}\twistC{\TheStrand}{1}h_{\TheStrand-1}\twistC{\TheStrand}{1}^{-1}$ by \ref{lem:psi_rel_C2} 
\end{tabularx}
}
}

Table \ref{tab:subgrp_rel_varphi} shows which relations are necessary that $\varphi$ preserves the relations from $A$. For the most of these relations, this follows directly from $B$. For the less obvious relations, we refer to the explanations below Table \ref{tab:subgrp_rel_varphi}. 
\renewcommand{\arraystretch}{1.75}
{\setlength\LTleft{-0.125\textwidth}
\begin{tabularx}{1.25\linewidth}{rl|l}
\caption{}
\label{tab:subgrp_rel_varphi} \\
& Relations from $A$ & Images under $\varphi$ \\
\cline{2-3}
\noalign{\vskip\doublerulesep\vskip-\arrayrulewidth}
\cline{2-3}
\endfirsthead
& Relations from $A$ & Images under $\varphi$ \\
\cline{2-3}
\noalign{\vskip\doublerulesep\vskip-\arrayrulewidth}
\cline{2-3}
\endhead
\multicolumn{3}{r}{\footnotesize(To be continued)}
\endfoot
\cline{2-3}
\noalign{\vskip\doublerulesep\vskip-\arrayrulewidth}
\cline{2-3}
\endlastfoot
\multirow{1}{3em}{\ref{cor:pres_orb_braid_free_prod_rel1}$\left\{\rule{0cm}{0.4cm}\right.$}& $\twC{1}^\TheOrder=1$ & $(h_1...h_{\TheStrand-1}\bar{\twsC}h_{\TheStrand-1}^{-1}...h_1^{-1})^\TheOrder=1$ by \ref{prop:sec_pres_free_prod_subgrp_1fact_relS1} 
\\ \cline{2-3} 
\vspace*{-2pt}\multirow{1}{3em}{\ref{cor:pres_orb_braid_free_prod_rel2}$\left\{\rule{0cm}{0.6cm}\right.$}& \doubletable{braid and commutator relations \\ of $h_1,...,h_{\TheStrand-1}$} & \doubletable{braid and commutator relations \\ of $h_1,...,h_{\TheStrand-1}$ are covered by \ref{prop:sec_pres_free_prod_subgrp_1fact_relR1}}
\\ \cline{2-3} 
\multirow{2}{3em}{\ref{cor:pres_orb_braid_free_prod_rel3}$\left\{\rule{0cm}{0.75cm}\right.$}& $[\twP{1},h_\Strand]=1$ \; for \; $\Strand\geq2$ & $[\twsP,h_\Strand]\stackrel{\text{\ref{prop:sec_pres_free_prod_subgrp_1fact_relR2}}}=1$
\\ \cline{2-3} 
& $[\twC{1},h_\Strand]=1$ \; for \; $\Strand\geq2$ & $[h_1...h_{\TheStrand-1}\bar{\twsC}h_{\TheStrand-1}^{-1}...h_1^{-1},h_\Strand]\stackrel{(\ref{prop:sec_pres_free_prod_subgrp_1fact_eq_hj_tau})}=1$, 
\\ \cline{2-3} 
\multirow{2}{3em}{\ref{cor:pres_orb_braid_free_prod_rel4}$\left\{\rule{0cm}{0.75cm}\right.$}& $[h_1\twP{1}h_1,\twP{1}]=1$ & $[h_1\twsP h_1,\twsP]\stackrel{\text{\ref{prop:sec_pres_free_prod_subgrp_1fact_relR2}}}=1$ 
\\ \cline{2-3} 
& $[h_1\twC{1}h_1,\twC{1}]=1$ & $[h_1\varphi(\twC{1})h_1,\varphi(\twC{1})]\stackrel{(\ref{prop:sec_pres_free_prod_subgrp_1fact_eq_h1_tau})}=1$
\\ \cline{2-3} 
\multirow{1}{3em}{\ref{cor:pres_orb_braid_free_prod_rel5}$\left\{\rule{0cm}{0.4cm}\right.$}& $[\twP{1},h_1^{-1}\twC{1}h_1]=1$ & $[\twsP,h_2...h_{\TheStrand-1}\bar{\twsC}h_{\TheStrand-1}^{-1}...h_2^{-1}]=1$ by \ref{prop:sec_pres_free_prod_subgrp_1fact_relR2} and \ref{prop:sec_pres_free_prod_subgrp_1fact_relC1} 
\end{tabularx}
}
Similar to \ref{prop:sec_pres_free_prod_2fact_par_tau_hj} we obtain  
\begin{equation}
\label{prop:sec_pres_free_prod_subgrp_1fact_eq_hj_tau}
h_\Strand\varphi(\twC{1})=\varphi(\twC{1})h_\Strand \text{ for } 1\leq\Strand\leq\TheStrand-2. 
\end{equation}

The relation below follows as in \ref{prop:sec_pres_free_prod_2fact_par_h1_tau_h1_tau} 
\begin{equation}
\label{prop:sec_pres_free_prod_subgrp_1fact_eq_h1_tau}
h_1\varphi(\twC{1})h_1\varphi(\twC{1})=\varphi(\twC{1})h_1\varphi(\twC{1})h_1. 
\end{equation} 

Consequently, by Theorem \ref{thm:von_Dyck}, the assignments $\varphi$ and $\psi$ induce inverse homomorphisms. Hence, the groups $A$ and $B$ are isomorphic, i.e.\ $\Z_\TheStrand(D_{\cycm}(1))$ has the above presentation. 
\end{proof}

\begin{theorem}
\label{thm:semidir_prod_subgrp_1fact}
For $\TheStrand\geq3$, the presentation from Proposition \textup{\ref{prop:sec_pres_free_prod_subgrp_1fact}} implies, that the group $\Z_\TheStrand(D_{\cycm}(1))$ has a semidirect product structure 
\[
\tilde{A}(\bar{\Delta}_\TheStrand^\TheOrder)\rtimes\cycm 
\]
with $\cycm=\langle\bar{\twsC}\rangle$ and $\tilde{A}(\bar{\Delta}_\TheOrder^\TheStrand)=\langle \twsP,h_1,...,h_{\TheStrand-1},h_{\TheStrand-1}^\twsC\rangle$. 

For $\TheOrder=2$, the group $\tilde{A}(\bar{\Delta}_\TheStrand^2)$ is an Artin group of type $\tilde{B}_\TheStrand$. 
\end{theorem}
\begin{proof}
To prove that the presentation from Proposition \ref{prop:sec_pres_free_prod_subgrp_1fact} induces the above semidirect product structures, we apply Lemma \ref{lem:semidir_prod_pres}. 

Therefore, we identify the generators and relations of the semidirect factors: The normal subgroup is generated by $\twsP,h_1,...,h_{\TheStrand-1}$ and $h_{\TheStrand-1}^\twsC$ with relations from \ref{prop:sec_pres_free_prod_subgrp_1fact_relR1}-\ref{prop:sec_pres_free_prod_subgrp_1fact_rel5}. For $\TheStrand\geq3$, i.e.\ the group that we call $\tilde{A}(\bar{\Delta}_\TheStrand^\TheOrder)$. The quotient is generated by $\bar{\twsC}$ with the relation from \ref{prop:sec_pres_free_prod_subgrp_1fact_relS1}, i.e.\ the group $\cycm$. Moreover, the presentation contains the relations \ref{prop:sec_pres_free_prod_subgrp_1fact_relC1} and \ref{prop:sec_pres_free_prod_subgrp_1fact_relC2} which are of the form $\bar{\twsC}h\bar{\twsC}^{-1}=\phi_{\bar{\twsC}}(h)$ with 
\[
h\in\{h_1,...,h_{\TheStrand-1},h_{\TheStrand-1}^{\twsC}\}
\] 
and $\phi_{\bar{\twsC}}(h)$ a word in the generators of $\tilde{A}(\bar{\Delta}_\TheStrand^\TheOrder)$. It remains to check that this presentation satisfies the conditions from Lemma \ref{lem:semidir_prod_pres_it2}. There we require that the assignments $h\mapsto\phi_{\bar{\twsC}}(h)$ induce an automorphism $\phi_{\bar{\twsC}}\in\Aut(\tilde{A}(\bar{\Delta}_\TheStrand^\TheOrder))$ and the assignments $\phi:\bar{\twsC}\mapsto\phi_{\bar{\twsC}}$ induce a homomorphism $\cycm\rightarrow\Aut(\tilde{A}(\bar{\Delta}_\TheStrand^\TheOrder))$. Again, we follow the Steps~\ref{rem:semidir_prod_pres_step1} and \ref{rem:semidir_prod_pres_step2} described in Remark \ref{rem:semidir_prod_pres}.

\begin{step}
\label{thm:semidir_prod_subgrp_1fact_step1}
The assignments $\phi$ induce a homomorphism. 
\end{step}

The relations from \ref{prop:sec_pres_free_prod_subgrp_1fact_relC1} and \ref{prop:sec_pres_free_prod_subgrp_1fact_relC2} induce the following assignments: 
\begin{align*}
\phi_{\bar{\twsC}}: \hspace*{6mm}\twsP & \mapsto \twsP, 
\\
h_\Strand & \mapsto h_\Strand \; \text{ for } \; 1\leq\Strand\leq\TheStrand-2, 
\\
h_{\TheStrand-1} & \mapsto h_{\TheStrand-1}^{\twsC}, 
\\
h_{\TheStrand-1}^{\twsC} & \mapsto (h_{\TheStrand-1}^{\twsC})^{-1}h_{\TheStrand-1}h_{\TheStrand-1}^{\twsC}. 
\end{align*}

Since the group $\cyc{\TheOrder}$ is defined by the relation $\bar{\twsC}^\TheOrder=1$, the assignments $\phi$ induce a homomorphism if $\phi_{\bar{\twsC}}^\TheOrder$ induces the identity. Using the relations $\bar{\twsC}h_{\TheStrand-1}\bar{\twsC}^{-1}=h_{\TheStrand-1}^\twsC$ and $\bar{\twsC}h_{\TheStrand-1}^\twsC\bar{\twsC}^{-1}=(h_{\TheStrand-1}^\twsC)^{-1}h_{\TheStrand-1}h_{\TheStrand-1}^\twsC$ from \ref{prop:sec_pres_free_prod_subgrp_1fact_relC2}, we obtain 
\[
\bar{\twsC}^kh_{\TheStrand-1}\bar{\twsC}^{-k}=\begin{cases}
\langle(h_{\TheStrand-1}^\twsC)^{-1},h_{\TheStrand-1}^{-1}\rangle_{k-1}\langle h_{\TheStrand-1}^\twsC,h_{\TheStrand-1}\rangle_k & \text{ for } k \text{ odd}, 
\\
\langle(h_{\TheStrand-1}^\twsC)^{-1},h_{\TheStrand-1}^{-1}\rangle_{k-1}\langle h_{\TheStrand-1},h_{\TheStrand-1}^\twsC\rangle_k & \text{ for } k \text{ even}. 
\end{cases} 
\]
This follows as in the proof of  (\ref{prop:sec_pres_free_prod_2fact_eq_conj_tau^k}). As in Theorem \ref{thm:semidir_prod_2fact}, the relation $\langle h_{\TheStrand-1}^\twsC,h_{\TheStrand-1}\rangle_\TheOrder=\langle h_{\TheStrand-1},h_{\TheStrand-1}^\twsC\rangle_\TheOrder$ allows us to deduce that $\phi_{\bar{\twsC}}^\TheOrder(h_{\TheStrand-1})=h_{\TheStrand-1}$ and $\phi_{\bar{\twsC}}^\TheOrder(h_{\TheStrand-1}^\twsC)=h_{\TheStrand-1}^\twsC$, see \eqref{prop:sec_pres_free_prod_2fact_eq_conj_tau_h1} and \eqref{prop:sec_pres_free_prod_2fact_eq_conj_tau_htau}. Further, $\phi_{\bar{\twsC}}(h_\Strand)=h_\Strand$ for $1\leq\Strand\leq\TheStrand-2$ and $\phi_{\bar{\twsC}}(\twsP)=\twsP$ by relation \ref{prop:sec_pres_free_prod_subgrp_1fact_relC1}. Thus $\phi_{\bar{\twsC}}^\TheOrder$ fixes each generator of $\tilde{A}(\bar{\Delta}_\TheStrand^\TheOrder)$. Hence, the assignments $\phi$ preserve the defining relations of $\cycm=\langle\bar{\twsC}\rangle$. By Theorem \ref{thm:von_Dyck}, these assignments induce a homomorphism. 

\begin{step}
\label{thm:semidir_prod_subgrp_1fact_step2}
The assignments $\phi_{\bar{\twsC}}:h\mapsto\phi_{\bar{\twsC}}(h)$ induce an automorphism in $\Aut(\tilde{A}(\bar{\Delta}_\TheStrand^\TheOrder)$. 
\end{step}

As explained in Step \ref{rem:semidir_prod_pres_step2} of Remark \ref{rem:semidir_prod_pres}, this only requires to check that the assignments $\phi_{\bar{\twsC}}:h\mapsto\phi_{\bar{\twsC}}(h)$ preserve the relations of $\tilde{A}(\bar{\Delta}_\TheStrand^\TheOrder)$. In Table \ref{tab:subgrp_conj_bar_tau} we list the defining relations of $\tilde{A}(\bar{\Delta}_\TheStrand^\TheOrder)$ and their $\phi_{\bar{\twsC}}$-images. We omit relations $w=w'$ if $\phi_{\bar{\twsC}}(w)=w$ and $\phi_{\bar{\twsC}}(w')=w'$. 
\renewcommand{\arraystretch}{1.75}
{\setlength\LTleft{-0.175\textwidth}
\begin{tabularx}{1.35\linewidth}{rl|l}
\caption{}
\label{tab:subgrp_conj_bar_tau} \\
& Relations from $\tilde{A}(\bar{\Delta}_\TheStrand^\TheOrder)$ & Images under $\phi_{\bar{\twsC}}$ \\
\cline{2-3}
\noalign{\vskip\doublerulesep\vskip-\arrayrulewidth}
\cline{2-3}
\endfirsthead
& Relations from $\tilde{A}(\bar{\Delta}_\TheStrand^\TheOrder)$ & Images under $\phi_{\bar{\twsC}}$ \\
\cline{2-3}
\noalign{\vskip\doublerulesep\vskip-\arrayrulewidth}
\cline{2-3}
\endhead
\multicolumn{3}{r}{\footnotesize(To be continued)}
\endfoot
\cline{2-3}
\noalign{\vskip\doublerulesep\vskip-\arrayrulewidth}
\cline{2-3}
\endlastfoot
\multirow{2}{4em}{\ref{prop:sec_pres_free_prod_subgrp_1fact_relR1}$\left\{\rule{0cm}{0.75cm}\right.$}& $h_{\TheStrand-2}h_{\TheStrand-1}h_{\TheStrand-2}=h_{\TheStrand-1}h_{\TheStrand-2}h_{\TheStrand-1}$ & $h_{\TheStrand-2}h_{\TheStrand-1}^\twsC h_{\TheStrand-2}\stackrel{\text{\ref{prop:sec_pres_free_prod_subgrp_1fact_relR1}}}=h_{\TheStrand-1}^\twsC h_{\TheStrand-2} h_{\TheStrand-1}^\twsC$ 
\\ \cline{2-3} 
& $[h_\Strand,h_{\TheStrand-1}]=1$ \; for \; $\Strand\leq\TheStrand-3$ & $[h_\Strand,h_{\TheStrand-1}^\twsC]\stackrel{\text{\ref{prop:sec_pres_free_prod_subgrp_1fact_relR1}}}=1$ 
\\ \cline{2-3} 
\multirow{2}{4em}{\ref{prop:sec_pres_free_prod_subgrp_1fact_relR1}$\left\{\rule{0cm}{1.2cm}\right.$}& $[h_\Strand,h_{\TheStrand-1}^\twsC]=1$ \; for \; $\Strand\leq\TheStrand-3$ & $[h_\Strand,(h_{\TheStrand-1}^\twsC)^{-1}h_{\TheStrand-1}h_{\TheStrand-1}^\twsC]=1$ by \ref{prop:sec_pres_free_prod_subgrp_1fact_relR1} 
\\ \cline{2-3} 
& $h_{\TheStrand-2}h_{\TheStrand-1}^\twsC h_{\TheStrand-2}=h_{\TheStrand-1}^\twsC h_{\TheStrand-2}h_{\TheStrand-1}^\twsC$ & \doubletable{$h_{\TheStrand-2}(h_{\TheStrand-1}^\twsC)^{-1}h_{\TheStrand-1}h_{\TheStrand-1}^\twsC h_{\TheStrand-2}$
\\
=$\ (h_{\TheStrand-1}^\twsC)^{-1}h_{\TheStrand-1}h_{\TheStrand-1}^\twsC h_{\TheStrand-2}(h_{\TheStrand-1}^\twsC)^{-1}h_{\TheStrand-1}h_{\TheStrand-1}^\twsC$ by (\ref{prop:sec_pres_free_prod_subgrp_1fact_eq_cpxbr_rel})}
\\ \cline{2-3} 
\multirow{1}{4em}{\ref{prop:sec_pres_free_prod_subgrp_1fact_relR2}$\left\{\rule{0cm}{0.35cm}\right.$}& $[\twsP,h_{\TheStrand-1}^\twsC]=1$ & $[\twsP,(h_{\TheStrand-1}^\twsC)^{-1}h_{\TheStrand-1}h_{\TheStrand-1}^\twsC]=1$ by  \ref{prop:sec_pres_free_prod_subgrp_1fact_relR2} 
\\ \cline{2-3} 
\multirow{1}{4em}{\ref{prop:sec_pres_free_prod_subgrp_1fact_relR3}$\left\{\rule{0cm}{0.6cm}\right.$}& $\langle h_{\TheStrand-1},h_{\TheStrand-1}^\twsC\rangle_\TheOrder=\langle h_{\TheStrand-1}^\twsC,h_{\TheStrand-1}\rangle_\TheOrder$ & \doubletable{$\langle h_{\TheStrand-1}^\twsC,(h_{\TheStrand-1}^\twsC)^{-1}h_{\TheStrand-1}h_{\TheStrand-1}^\twsC\rangle_\TheOrder$
\\
=$\ \langle(h_{\TheStrand-1}^\twsC)^{-1}h_{\TheStrand-1}h_{\TheStrand-1}^\twsC,h_{\TheStrand-1}^\twsC\rangle_\TheOrder$ by (\ref{prop:sec_pres_free_prod_subgrp_1fact_eq_conj_Artin-rel_length_m}) }
\\ \cline{2-3} 
\multirow{1}{4em}{\ref{prop:sec_pres_free_prod_subgrp_1fact_rel5}$\left\{\rule{0cm}{0.6cm}\right.$}& $(h_{\TheStrand-1}h_{\TheStrand-1}^\twsC h_{\TheStrand-2})^2=(h_{\TheStrand-2}h_{\TheStrand-1}h_{\TheStrand-1}^\twsC)^2$ & \doubletable{$(\textcolor{\short}{h_{\TheStrand-1}^\twsC(h_{\TheStrand-1}^\twsC)^{-1}}h_{\TheStrand-1}h_{\TheStrand-1}^\twsC h_{\TheStrand-2})^2$
\\
$\stackrel{\text{\ref{prop:sec_pres_free_prod_2fact_relR3}}}=(h_{\TheStrand-2}\textcolor{\short}{h_{\TheStrand-1}^\twsC(h_{\TheStrand-1}^\twsC)^{-1}}h_{\TheStrand-1}h_{\TheStrand-1}^\twsC)^2$ }
\end{tabularx}}
In particular, in Table \ref{tab:subgrp_conj_bar_tau} we refer to the following observations: 
\begin{align*}
\label{prop:sec_pres_free_prod_subgrp_1fact_eq_cpxbr_rel}
& (h_{\TheStrand-1}^\twsC)^{-1}h_{\TheStrand-1}h_{\TheStrand-1}^\twsC h_{\TheStrand-2}(h_{\TheStrand-1}^\twsC)^{-1}h_{\TheStrand-1}h_{\TheStrand-1}^\twsC=h_{\TheStrand-2}(h_{\TheStrand-1}^\twsC)^{-1}h_{\TheStrand-1}h_{\TheStrand-1}^\twsC h_{\TheStrand-2} \numbereq 
\\
\Leftrightarrow \ & (h_{\TheStrand-1}h_{\TheStrand-1}^\twsC h_{\TheStrand-2})^2=(h_{\TheStrand-2}h_{\TheStrand-1}h_{\TheStrand-1}^\twsC)^2
\end{align*}
and 
\begin{align*}
\label{prop:sec_pres_free_prod_subgrp_1fact_eq_conj_Artin-rel_length_m}
& \langle h_{\TheStrand-1}^\twsC,(h_{\TheStrand-1}^\twsC)^{-1}h_{\TheStrand-1}h_{\TheStrand-1}^\twsC\rangle_\TheOrder=\langle(h_{\TheStrand-1}^\twsC)^{-1}h_{\TheStrand-1}h_{\TheStrand-1}^\twsC,h_{\TheStrand-1}^\twsC\rangle_\TheOrder \numbereq
\\
\Leftrightarrow \ & \langle h_{\TheStrand-1},h_{\TheStrand-1}^\twsC\rangle_\TheOrder=\langle h_{\TheStrand-1}^\twsC,h_{\TheStrand-1}\rangle_\TheOrder. 
\end{align*}
The equivalence in (\ref{prop:sec_pres_free_prod_subgrp_1fact_eq_cpxbr_rel}) follows as its analog from \ref{prop:sec_pres_free_prod_2fact_par_cpxbr_rel}. 
The observation in (\ref{prop:sec_pres_free_prod_subgrp_1fact_eq_conj_Artin-rel_length_m}) is analogous to \ref{prop:sec_pres_free_prod_2fact_par_conj_Artin-rel_length_m}. 
Consequently, the assignments $\phi_{\bar{\twsC}}$ induce an automorphism of $\tilde{A}(\bar{\Delta}_\TheStrand^\TheOrder)$. Hence, the presentation from Proposition \ref{prop:sec_pres_free_prod_subgrp_1fact} satisfies the conditions from Lemma \ref{lem:semidir_prod_pres_it2} and yields the semidirect product structure $\tilde{A}(\bar{\Delta}_\TheStrand^\TheOrder)\rtimes\cycm$. 

For $\TheOrder=2$, the relation \ref{prop:sec_pres_free_prod_subgrp_1fact_relR3} is given by $h_{\TheStrand-1}h_{\TheStrand-1}^\twsC=h_{\TheStrand-1}^\twsC h_{\TheStrand-1}$. Analogous to Theorem \ref{thm:semidir_prod_2fact}, this simplifies the $\phi_{\bar{\twsC}}$-image of $h_{\TheStrand-1}^{\twsC}h_{\TheStrand-2}h_{\TheStrand-1}^{\twsC}=h_{\TheStrand-2}h_{\TheStrand-1}^{\twsC}h_{\TheStrand-2}$ to $h_{\TheStrand-1}^{\twsC}h_{\TheStrand-2}h_{\TheStrand-1}^{\twsC}=h_{\TheStrand-2}h_{\TheStrand-1}^{\twsC}h_{\TheStrand-2}$ which is covered by \ref{prop:sec_pres_free_prod_subgrp_1fact_relR1}. Thus, we can read off from Table \ref{tab:subgrp_conj_bar_tau} that the defining relations of $A(\tilde{B}_\TheStrand)$ are preserved by $\phi_{\bar{\twsC}}$. As claimed above, this implies that the group $\Z_\TheStrand(D_{\cyc{2}}(1))$ has the semidirect product structure $A(\tilde{B}_\TheStrand)\rtimes\cyc{2}$. 
\end{proof}

\section{Central elements in $\Z_\TheStrand(D_{\cycm})$ and $\tilde{A}(\Delta_\TheStrand^\TheOrder)=\B(\TheOrder,\TheOrder,\TheStrand)$}
\label{sec:center}

Recall from \cite[Theorem 1.24]{KasselTuraev2008} that the center of $\B_\TheStrand$ is infinite cyclic generated by the braid pictured in Figure \ref{fig:central_elements} (left). Similarly, we want to show that the orbifold braid $\theta_\TheStrand$ in Figure \ref{fig:central_elements} (right) is contained in the center of $\Z_\TheStrand(D_{\cycm})$. Moreover, we will show that a certain power of $\theta_\TheStrand$ is contained in the center of $\tilde{\mathcal{A}}(\Delta_\TheStrand^\TheOrder)$. Throughout Section \ref{sec:center} we will use $\twsC:=\twC{1}$ as an abbreviation. 

\begin{figure}[H]
\resizebox{0.55\textwidth}{!}{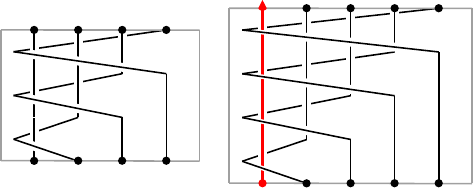}
\caption{Central elements in $\B_\TheStrand$ (left) and $\Z_\TheStrand(D_{\cycm})$ (right) for $\TheStrand=4$.}
\label{fig:central_elements}
\end{figure}

\begin{proposition}
\label{prop:D_Z_m_central_el}
For $\TheStrand\geq1$ and $\TheOrder\geq2$, the center of $\Z_\TheStrand(D_{\cycm})$ contains the element 
\begin{equation}
\label{eq:def_theta_n}
\theta_\TheStrand:=\gamma_\TheStrand\gamma_{\TheStrand-1}...\gamma_1 
\end{equation}
with 
\begin{equation}
\label{eq:def_gamma_j}
\gamma_\Strand:=h_{\Strand-1}...h_1\twsC h_1...h_{\Strand-1}
\end{equation}
 for $1\leq\Strand\leq\TheStrand$. 
\end{proposition}
\begin{proof}
We use the product decomposition of $\theta_\TheStrand$ and consider the products $h_\Str\gamma_\Strand$ and $\twsC\gamma_\Strand$ for $1\leq\Str<\TheStrand$ and $1\leq\Strand\leq\TheStrand$. For $1\leq\Strand,\NStrand<\TheStrand, \Strand<\NStrand$, the relations \ref{cor:pres_orb_braid_free_prod_rel2} and \ref{cor:pres_orb_braid_free_prod_rel3} imply that 
\begin{align*}
\label{prop:D_Z_m_central_el_eq_h_k_gamma_j}
h_\NStrand\gamma_\Strand & \mystackrel{\eqref{eq:def_gamma_j}}=\textcolor{\col}{h_\NStrand h_{\Strand-1}...h_1}\twsC h_1...h_{\Strand-1} \numbereq 
\\
& \mystackrel{\ref{cor:pres_orb_braid_free_prod_rel2}}= h_{\Strand-1}...h_1\textcolor{\col}{h_\NStrand\twsC} h_1...h_{\Strand-1}
\\
& \mystackrel{\ref{cor:pres_orb_braid_free_prod_rel3}}= h_{\Strand-1}...h_1\twsC\textcolor{\col}{h_\NStrand h_1...h_{\Strand-1}} 
\\
& \mystackrel{\ref{cor:pres_orb_braid_free_prod_rel2}}= h_{\Strand-1}...h_1\twsC h_1...h_{\Strand-1}h_\NStrand \mystackrel{\eqref{eq:def_gamma_j}}= \gamma_\Strand h_\NStrand. 
\end{align*}
For $1\leq\Str,\Strand\leq\TheStrand, \Str<\Strand-1$, we observe that 
\begin{align}
h_\Str\gamma_\Strand & \mystackrel{(\ref{eq:def_gamma_j})}= \textcolor{\col}{h_\Str h_{\Strand-1}..}..h_1\twsC h_1...h_{\Strand-1} 
\label{prop:D_Z_m_central_el_eq_h_i_gamma_j}
\\
& \mystackrel{\ref{cor:pres_orb_braid_free_prod_rel2}}= h_{\Strand-1}...h_{\Str+2}\textcolor{\col}{h_\Str h_{\Str+1}h_\Str}...h_1\twsC h_1...h_{\Strand-1} \nonumber
\\
& \mystackrel{\ref{cor:pres_orb_braid_free_prod_rel2}}= h_{\Strand-1}...h_\Str\textcolor{\col}{h_{\Str+1}h_{\Str-1}...h_1\twsC h_1..}..h_{\Strand-1} \nonumber
\\
& \mystackrel{\ref{cor:pres_orb_braid_free_prod_rel2},\ref{cor:pres_orb_braid_free_prod_rel3}}= h_{\Strand-1}...h_1\twsC h_1...h_{\Str-1}\textcolor{\col}{h_{\Str+1}h_\Str h_{\Str+1}}...h_{\Strand-1} \nonumber
\\
& \mystackrel{\ref{cor:pres_orb_braid_free_prod_rel2}}= h_{\Strand-1}...h_1\twsC h_1...h_{\Str+1}\textcolor{\col}{h_\Str h_{\Str+2}...h_{\Strand-1}} \nonumber
\\
& \mystackrel{\ref{cor:pres_orb_braid_free_prod_rel2}}= h_{\Strand-1}...h_1\twsC h_1...h_{\Strand-1}h_\Str \mystackrel{(\ref{eq:def_gamma_j})}= \gamma_\Strand h_\Str, \nonumber
\end{align}

Moreover, we observe: 
\begin{align}
\twsC\gamma_\Strand & \mystackrel{(\ref{eq:def_gamma_j})}= \textcolor{\col}{\twsC h_{\Strand-1}...}h_1\twsC h_1...h_{\Strand-1} 
\label{prop:D_Z_m_central_el_eq_tau_gamma_j}
\\
& \mystackrel{\ref{cor:pres_orb_braid_free_prod_rel3}}=  h_{\Strand-1}...h_2\textcolor{\col}{\twsC h_1\twsC h_1}...h_{\Strand-1} \nonumber
\\
& \mystackrel{\ref{cor:pres_orb_braid_free_prod_rel4}}=  h_{\Strand-1}...h_1\twsC h_1\textcolor{\col}{\twsC h_2...h_{\Strand-1}} \nonumber
\\
& \mystackrel{\ref{cor:pres_orb_braid_free_prod_rel3}}=  h_{\Strand-1}...h_1\twsC h_1...h_{\Strand-1}\twsC \mystackrel{(\ref{eq:def_gamma_j})}= \gamma_\Strand\twsC, \nonumber
\\[5mm]
h_\Strand\gamma_\Strand & \mystackrel{(\ref{eq:def_gamma_j})}= h_\Strand h_{\Strand-1}...h_1\twsC h_1...h_{\Strand-1}
\label{prop:D_Z_m_central_el_eq_h_j_gamma_j}
\\
& \mystackrel{$\vee$}= h_\Strand h_{\Strand-1}...h_1\twsC h_1...h_{\Strand-1}(h_\Strand h_\Strand^{-1}) \mystackrel{(\ref{eq:def_gamma_j})}= \gamma_{\Strand+1}h_\Strand^{-1} \text{ and } \nonumber
\\[5mm]
h_\Strand^{-1}\gamma_{\Strand+1} & \mystackrel{(\ref{eq:def_gamma_j})}= \textcolor{\short}{h_\Strand^{-1}h_\Strand}h_{\Strand-1}...h_1\twsC h_1...h_\Strand 
\label{prop:D_Z_m_central_el_eq_h_j_gamma_j+1}
\\
& \mystackrel{}= h_{\Strand-1}...h_1\twsC h_1...h_{\Strand-1}h_\Strand \mystackrel{(\ref{eq:def_gamma_j})}= \gamma_\Strand h_\Strand \nonumber
\end{align}

In the following, it will be convenient to express $\gamma_\Strand$ in terms of the pure generators from Theorem \ref{thm:pres_pure_free_prod}. We obtain: 
\begin{align}
\label{prop:D_Z_m_central_el_def_gamma_j}
& \twist{\Strand,}{\Strand-1}\twist{\Strand,}{\Strand-2}...\twist{\Strand}{1}\twistC{\Strand}{1} \numbereq
\\
\mystackrel{\eqref{eq:def_a_ji},\eqref{eq:def_a_kc}}= & h_{\Strand-1}^{\textcolor{\short}{2}}\textcolor{\short}{(h_{\Strand-1}^{-1}}h_{\Strand-2}^{\textcolor{\short}{2}}\textcolor{\short}{h_{\Strand-1})..}..\textcolor{\short}{..(h_{\Strand-1}^{-1}...h_2^{-1}}h_1^{\textcolor{\short}{2}}\textcolor{\short}{h_2...h_{\Strand-1})(h_{\Strand-1}^{-1}...h_1^{-1}}\twsC h_1...h_{\Strand-1}) \nonumber
\\
\mystackrel{}= & h_{\Strand-1}...h_1\twsC h_1...h_{\Strand-1} \stackrel{\eqref{eq:def_gamma_j}}= \gamma_\Strand. \nonumber
\end{align}

Further, recall from Theorem \ref{thm:pres_pure_free_prod} that the pure braids inter alia satisfy the following relations: For $1\leq\Str,\Strand,\NStrand,\NNStrand\leq\TheStrand, \Str<\Strand<\NStrand<\NNStrand$ and $1\leq\nu\leq\TheCone$ 
\begin{align*}
\label{eq:rel_pure_free_prod_rel2}
[\twist{\Strand}{\Str},\twist{\NNStrand}{\NStrand}] & =1, \numbereq 
\\
\label{eq:rel_pure_free_prod_rel2_1}
[\twistC{\Strand}{\nu},\twist{\NNStrand}{\NStrand}] & =1, \numbereq 
\\
\label{eq:rel_pure_free_prod_rel3}
[\twist{\NNStrand}{\Str},\twist{\NStrand}{\Strand}] & =1, \numbereq 
\\
\label{eq:rel_pure_free_prod_rel3_1}
[\twistC{\NNStrand}{\nu},\twist{\NStrand}{\Strand}] & =1, \numbereq 
\\
\label{eq:rel_pure_free_prod_rel4}
[\twist{\NNStrand}{\NStrand}\twist{\NNStrand}{\Strand}\twist{\NNStrand}{\NStrand}^{-1},\twist{\NStrand}{\Str}] & =1, \numbereq 
\\
\label{eq:rel_pure_free_prod_rel4_1}
[\twist{\NStrand}{\Strand}\twist{\NStrand}{\Str}\twist{\NStrand}{\Strand}^{-1},\twistC{\Strand}{\nu}] & =1, \numbereq 
\\
\label{eq:rel_pure_free_prod_rel5}
\twist{\Strand}{\Str}\twist{\NStrand}{\Strand}\twist{\NStrand}{\Str} & =\twist{\NStrand}{\Strand}\twist{\NStrand}{\Str}\twist{\Strand}{\Str}, \numbereq 
\\
\label{eq:rel_pure_free_prod_rel5_1}
\twist{\Strand}{\Str}\twistC{\Strand}{\nu}\twistC{\Str}{\nu} & =\twistC{\Str}{\nu}\twist{\Strand}{\Str}\twistC{\Strand}{\nu} \numbereq  
\end{align*}

This prepares us to observe that $\gamma_\Strand$ and $\gamma_\Str$ commute for $1\leq\Str<\Strand\leq\TheStrand$. For $1\leq\Strr,\Str,\Strand\leq\TheStrand, \Strr<\Str<\Strand$, we observe 
\begin{align*}
\label{prop:D_Z_m_central_el_eq_gamma_j_a_ih}
\gamma_\Strand\twist{\Str}{\Strr} & \mystackrel{\eqref{prop:D_Z_m_central_el_def_gamma_j}}= \twist{\Strand,}{\Strand-1}..\textcolor{\col}{..\twist{\Strand}{1}\twistC{\Strand}{1}\twist{\Str}{\Strr}} \numbereq
\\
& \mystackrel{\eqref{eq:rel_pure_free_prod_rel3},\eqref{eq:rel_pure_free_prod_rel3_1}}= \twist{\Strand,}{\Strand-1}...\twist{\Strand}{\Strr}\twist{\Str}{\Strr}\twist{\Strand,}{\Strr-1}...\twist{\Strand}{1}\twistC{\Strand}{1} \nonumber
\\
& \mystackrel{$\vee$}= \twist{\Strand,}{\Strand-1}...\twist{\Strand,}{\Strr+1}(\twist{\Strand}{\Str}^{-1}\textcolor{\col}{\twist{\Strand}{\Str})\twist{\Strand}{\Strr}\twist{\Str}{\Strr}}\twist{\Strand,}{\Strr-1}...\twist{\Strand}{1}\twistC{\Strand}{1} \nonumber
\\
& \mystackrel{\eqref{eq:rel_pure_free_prod_rel5}}= \twist{\Strand,}{\Strand-1}...\textcolor{\col}{\twist{\Strand}{\Str}\twist{\Strand,}{\Str-1}...\twist{\Strand,}{\Strr+1}\twist{\Strand}{\Str}^{-1}\twist{\Str}{\Strr}}\twist{\Strand}{\Str}\twist{\Strand}{\Strr}\twist{\Strand,}{\Strr-1}...\twist{\Strand}{1}\twistC{\Strand}{1} \nonumber
\\
& \mystackrel{\eqref{eq:rel_pure_free_prod_rel4}}= \textcolor{\col}{\twist{\Strand,}{\Strand-1}...\twist{\Strand,}{\Str+1}\twist{\Str}{\Strr}}\twist{\Strand}{\Str}\twist{\Strand,}{\Str-1}...\twist{\Strand,}{\Strr+1}\textcolor{\short}{\twist{\Strand}{\Str}^{-1}\twist{\Strand}{\Str}}\twist{\Strand}{\Strr}\twist{\Strand,}{\Strr-1}...\twist{\Strand}{1}\twistC{\Strand}{1} \nonumber
\\
& \mystackrel{\eqref{eq:rel_pure_free_prod_rel2}}= \twist{\Str}{\Strr}\twist{\Strand,}{\Strand-1}...\twist{\Strand,}{\Str+1}\twist{\Strand}{\Str}\twist{\Strand,}{\Str-1}...\twist{\Strand}{1}\twistC{\Strand}{1} \mystackrel{\eqref{prop:D_Z_m_central_el_def_gamma_j}}= \twist{\Str}{\Strr}\gamma_\Strand. \nonumber
\end{align*}

Using a similar observation for $\twistC{\Str}{1}$, we obtain: 

\begin{align*}
\label{prop:D_Z_m_central_el_eq_gamma_i_gamma_j}
\gamma_\Strand\gamma_\Str & \mystackrel{\eqref{eq:def_gamma_j}}= \textcolor{\col}{\gamma_\Strand\twist{\Str,}{\Str-1}...\twist{\Str}{1}}\twistC{\Str}{1} \numbereq
\\
& \mystackrel{\eqref{prop:D_Z_m_central_el_eq_gamma_j_a_ih}}= \twist{\Str,}{\Str-1}...\twist{\Str}{1}\gamma_\Strand\twistC{\Str}{1} \nonumber
\\
& \mystackrel{\eqref{eq:def_gamma_j}}= \twist{\Str,}{\Str-1}...\twist{\Str}{1}\twist{\Strand,}{\Strand-1}...\twist{\Strand}{1}\twistC{\Strand}{1}\twistC{\Str}{1} \nonumber
\\
& \mystackrel{$\vee$}= \twist{\Str,}{\Str-1}...\twist{\Str}{1}\twist{\Strand,}{\Strand-1}...\twist{\Strand}{1}(\twist{\Strand}{\Str}^{-1}\textcolor{\col}{\twist{\Strand}{\Str})\twistC{\Strand}{1}\twistC{\Str}{1}} \nonumber
\\
& \mystackrel{\eqref{eq:rel_pure_free_prod_rel5_1}}= \twist{\Str,}{\Str-1}...\twist{\Str}{1}\twist{\Strand,}{\Strand-1}...\twist{\Strand,}{\Str+1}\textcolor{\col}{\twist{\Strand}{\Str}\twist{\Strand,}{\Str-1}...\twist{\Strand}{1}\twist{\Strand}{\Str}^{-1}\twistC{\Str}{1}}\twist{\Strand}{\Str}\twistC{\Strand}{1} \nonumber
\\
& \mystackrel{\eqref{eq:rel_pure_free_prod_rel4_1}}= \twist{\Str,}{\Str-1}...\twist{\Str}{1}\twist{\Strand,}{\Strand-1}...\twist{\Strand,}{\Str+1}\twistC{\Str}{1}\twist{\Strand}{\Str}\twist{\Strand,}{\Str-1}...\twist{\Strand}{1}\textcolor{\short}{\twist{\Strand}{\Str}^{-1}\twist{\Strand}{\Str}}\twistC{\Strand}{1} \nonumber
\\
& \mystackrel{}= \twist{\Str,}{\Str-1}...\twist{\Str}{1}\textcolor{\col}{\twist{\Strand,}{\Strand-1}...\twist{\Strand,}{\Str+1}\twistC{\Str}{1}}\twist{\Strand}{\Str}\twist{\Strand,}{\Str-1}...\twist{\Strand}{1}\twistC{\Strand}{1} \nonumber
\\
& \mystackrel{\eqref{eq:rel_pure_free_prod_rel2_1}}= \twist{\Str,}{\Str-1}...\twist{\Str}{1}\twistC{\Str}{1}\twist{\Strand,}{\Strand-1}...\twist{\Strand,}{\Str+1}\twist{\Strand}{\Str}\twist{\Strand,}{\Str-1}...\twist{\Strand}{1}\twistC{\Strand}{1} \mystackrel{\eqref{prop:D_Z_m_central_el_def_gamma_j}}= \gamma_\Str\gamma_\Strand. \nonumber
\end{align*}

This implies 
\begin{align*}
h_\Strand\theta_\TheStrand & \mystackrel{\eqref{eq:def_theta_n}}= h_\Strand\gamma_\TheStrand...\textcolor{\col}{\gamma_{\Strand+1}\gamma_\Strand}...\gamma_1
\\
& \mystackrel{(\ref{prop:D_Z_m_central_el_eq_gamma_i_gamma_j})}= \textcolor{\col}{h_\Strand\gamma_\TheStrand...\gamma_{\Strand+2}}\gamma_\Strand\gamma_{\Strand+1}\gamma_{\Strand-1}...\gamma_1
\\
& \mystackrel{(\ref{prop:D_Z_m_central_el_eq_h_i_gamma_j})}= \gamma_\TheStrand...\gamma_{\Strand+2}\textcolor{\col}{h_\Strand\gamma_\Strand}\gamma_{\Strand+1}\gamma_{\Strand-1}...\gamma_1
\\
& \mystackrel{(\ref{prop:D_Z_m_central_el_eq_h_j_gamma_j})}= \gamma_\TheStrand...\gamma_{\Strand+1}\textcolor{\col}{h_\Strand^{-1}\gamma_{\Strand+1}}\gamma_{\Strand-1}...\gamma_1
\\
& \mystackrel{(\ref{prop:D_Z_m_central_el_eq_h_j_gamma_j+1})}= \gamma_\TheStrand...\gamma_\Strand\textcolor{\col}{h_\Strand\gamma_{\Strand-1}...\gamma_1}
\\
& \mystackrel{(\ref{prop:D_Z_m_central_el_eq_h_k_gamma_j})}= \gamma_\TheStrand...\gamma_\Strand \gamma_{\Strand-1}...\gamma_1h_\Strand \mystackrel{\eqref{eq:def_theta_n}}= \theta_\TheStrand h_\Strand. 
\end{align*}
Thus, $h_\Strand$ commutes with $\theta_\TheStrand$ and $\twsC$ commutes with each $\gamma_\Strand$ and consequently with~$\theta_\TheStrand$. This yields that $\theta_\TheStrand$ is contained in the  center of $\Z_\TheStrand(D_{\cycm}))$.  
\end{proof}

Under the epimorphism 
\[
\PZ_\TheStrand(D_{\cycm})\rightarrow\ZZ
\]
that maps $a_{\Strand1}$ to $1$ for $2\leq\Strand\leq\TheStrand$ and all other generators from Theorem \ref{thm:pres_pure_free_prod} to $0$, the element $\theta_\TheStrand$ maps to $\TheStrand-1$. Thus, for $\TheStrand\geq2$, the element $\theta_\TheStrand$ is of infinite order. Moreover, we determine the powers of $\theta_\TheStrand$ which are contained in the subgroup~$\tilde{\mathcal{A}}(\Delta_\TheStrand^\TheOrder)$. 

\begin{lemma}
\label{lem:cbr_grp_Delta_n_m_cont_power_theta_n}
Let $\TheStrand\geq2$. Moreover, let $\Order>0$ be the minimal number such that $\TheOrder$ divides $\Order\TheStrand$. The group $\tilde{\mathcal{A}}(\Delta_\TheStrand^\TheOrder)$ contains $\theta_\TheStrand^k$ if and only if $k$ is a multiple of $\Order$. 
\end{lemma}
\begin{proof}
Let us assume that $k$ is a multiple of $\Order$. We recall that 
\[
\theta_\TheStrand\stackrel{\eqref{eq:def_theta_n}}=\gamma_\TheStrand\gamma_{\TheStrand-1}...\gamma_1 \; \text{ with } \; \gamma_\Strand\stackrel{\eqref{eq:def_gamma_j}}=h_{\Strand-1}...h_1\twsC h_1...h_{\Strand-1}. 
\]
In particular, $\gamma_1=\twsC$ is an element of order $\TheOrder$. Since $\TheOrder$ divides $\Order\TheStrand$, one has $\twsC^{\Order\TheStrand}=1$ and consequently $\twsC^\Order=\twsC^{-\Order(\TheStrand-1)}$. 
\begin{align*}
\label{lem:cbr_grp_Delta_n_m_cont_power_theta_n_eq_theta_n_m_in_A}
\theta_\TheStrand^\Order & \mystackrel{\eqref{eq:def_theta_n}}=(\gamma_\TheStrand\gamma_{\TheStrand-1}...\gamma_1)^\Order \numbereq
\\
& \mystackrel{(\ref{prop:D_Z_m_central_el_eq_gamma_i_gamma_j})}=\twsC^\Order\gamma_2^\Order...\gamma_\TheStrand^\Order \nonumber
\\
& \mystackrel{}=\twsC^{-\Order(\TheStrand-1)}\gamma_2^\Order...\gamma_\TheStrand^\Order \nonumber
\\
& \mystackrel{\eqref{prop:D_Z_m_central_el_eq_gamma_i_gamma_j}}=((\gamma_2\twsC^{-1})(\gamma_3\twsC^{-1})...(\gamma_\TheStrand\twsC^{-1}))^\Order \nonumber
\\
& \mystackrel{\eqref{eq:def_gamma_j}}=((h_1\twsC h_1\twsC^{-1})(h_2h_1\twsC h_1\textcolor{\col}{h_2\twsC^{-1}})...(h_{\TheStrand-1}...h_1\twsC h_1\textcolor{\col}{...h_{\TheStrand-1}\twsC^{-1}}))^\Order \nonumber
\\
& \mystackrel{\ref{cor:pres_orb_braid_free_prod_rel3}}=((h_1\twsC h_1\twsC^{-1})(h_2h_1\twsC h_1\twsC^{-1}h_2)...(h_{\TheStrand-1}...h_1\twsC h_1\twsC^{-1}h_2...h_{\TheStrand-1}))^\Order \nonumber
\\
& \mystackrel{}=((h_1\htwC)(h_2h_1\htwC h_2)...(h_{\TheStrand-1}...h_1h_1^\twsC h_2...h_{\TheStrand-1}))^\Order\in\tilde{\mathcal{A}}(\Delta_\TheStrand^\TheOrder). \nonumber
\end{align*}
This implies that $\theta_\TheStrand^k\in\tilde{\mathcal{A}}(\Delta_\TheStrand^\TheOrder)$ for any multiple $k$ of $\Order$. 

Now let us assume that $\theta_\TheStrand^k$ is contained in $\tilde{\mathcal{A}}(\Delta_\TheStrand^\TheOrder)$ and we consider the following homomorphism induced by the semidirect product structure from Corollary \ref{cor:sec_pres_free_prod_1fact} :  
\begin{equation*}
p:\Z_\TheStrand(D_{\cycm})\rightarrow\cycm, 
\end{equation*}
where the homomorphism $p$ is counting the number of generators $\twsC$ (modulo $\TheOrder$), i.e.\ how many times the braid in total encircles the cone point. Under $p$ the element $\theta_\TheStrand^k$ maps to $\overline{k\TheStrand}$ while the generators $\htwC$ and $h_\Strand$ for $1\leq\Strand<\TheStrand$ of $\tilde{\mathcal{A}}(\Delta_\TheStrand^\TheOrder)$ map to $\bar{0}$. Since $\theta_\TheStrand^k$ is contained in $\tilde{\mathcal{A}}(\Delta_\TheStrand^\TheOrder)$, $\TheOrder$ divides $k\TheStrand$. By assumption, $\TheOrder$ also divides $\Order\TheStrand$. Hence, the minimality of $\Order$ implies that $\Order$ divides $k$. 
\end{proof}

\begin{corollary}
\label{cor:cbr_grp_Delta_n_m_center}
For $\TheStrand,\TheOrder\geq2$, the center of $\B(\TheOrder,\TheOrder,\TheStrand)$ contains the non-trivial element $\theta_\TheStrand^\Order$ where $\Order>0$ is the minimal number such that $\TheOrder$ divides $\Order\TheStrand$. 
\end{corollary}

\bibliographystyle{plain}
\bibliography{paper_general_Allcock}

\end{document}